\documentclass{amsart}[12pt]
\usepackage{amssymb} 
\usepackage{amsmath}
\usepackage{hyperref}

\usepackage{color}
\usepackage{graphics}
\usepackage[dvips]{graphicx}

\begin{document}

\newcommand{\barint}{\overline{\hspace{.65em}}\!\!\!\!\!\!\int}
\newcommand{\ep}{\varepsilon}
\newcommand{\logeps}{|\log\ep|}
\newcommand{\jep}{j_\ep}
\newcommand{\loc}{ {\mbox{\scriptsize{loc}}} }
\newcommand{\R}{{\mathbb R}}
\newcommand{\C}{{\mathbb C}}
\newcommand{\J}{{\mathbb J}}
\newcommand{\T}{{\mathbb T}}
\newcommand{\Z}{{\mathbb Z}}
\newcommand{\N}{{\mathbb N}}
\newcommand{\wt}{\widetilde}
\newcommand{\Q}{{\mathbb Q}}
\newcommand{\Hdf}{{\mathcal{H}}}
\newcommand{\calD}{{\mathcal{D}}}
\newcommand{\calC}{{\mathcal{C}}}
\newcommand{\calS}{{\mathcal{S}}}
\newcommand{\calL}{{\mathcal{L}}}
\newcommand{\dg}{\operatorname{dg}}
\newcommand{\calF}{{\mathcal{F}}}
\newcommand{\calG}{{\mathcal{G}}}
\newcommand{\calB}{\mathcal{B}}
\newcommand{\calH}{{\mathcal{H}}}
\newcommand{\calJ}{{\mathcal{J}}}
\newcommand{\calT}{{\mathcal{T}}}
\newcommand{\calM}{{\mathcal{M}}}
\newcommand{\dist}{\operatorname{dist}}
\newcommand{\spt}{\operatorname{spt}}
\def\rest{\hskip 1pt{\hbox to 10.8pt{\hfill
\vrule height 7pt width 0.4pt depth 0pt\hbox{\vrule height 0.4pt
width 7.6pt depth 0pt}\hfill}}}
\newcommand{\bd}{\partial}

\newcommand{\vol}{\,\mbox{vol}}

\newcommand{\bn}[1]{ ({\texttt{#1}})}

\newcommand{\beq}{\begin{equation}}
\newcommand{\eeq}{\end{equation}}

\def\logeps{{|\!\log\varepsilon|}}

\theoremstyle{plain}
\newtheorem{theorem}{Theorem}
\newtheorem{goal}{Goal}
\newtheorem{proposition}{Proposition}
\newtheorem{lemma}{Lemma}
\newtheorem{corollary}{Corollary}

\theoremstyle{definition}
\newtheorem{definition}{Definition}

\theoremstyle{remark}
\newtheorem{remark}{Remark}
\newtheorem{example}{Example}
\newtheorem{warning}{Warning}

\numberwithin{equation}{section}
\setcounter{tocdepth}{3}

\newcommand{\bl}{\color{blue}}
\definecolor{darkgreen}{rgb}{0,0.55,0}
\newcommand{\blue}[1]{{\textcolor{blue}{#1}}} 
\newcommand{\red}[1]{{\textcolor{red}{#1}}} 
\newcommand{\green}[1]{{\textcolor{darkgreen}{#1}}}

\newcommand{\Om}{\Omega}
\newcommand{\om}{\omega}
\newcommand{\e}{\varepsilon}
\newcommand{\abs}[1]{\left\vert{#1}\right\vert}
\newcommand{\norm}[1]{\left\Vert{#1}\right\Vert}
\newcommand{\p}{\partial}
\newcommand{\calO}{\mathcal{O}}
\newcommand{\vp}{\varphi}
\newcommand{\Tgi}{T_{\gamma_i}}

\title[Nearly parallel vortex filaments]{Nearly Parallel Vortex Filaments in the 3D Ginzburg-Landau Equations}

\author{Andres Contreras \and Robert L. Jerrard  }
\address{Department of Mathematical Sciences, New Mexico State University, Las Cruces, New Mexico, USA}\email{acontre@nmsu.edu}
\address{Department of Mathematics, University of Toronto,
Toronto, Ontario, Canada}\email{rjerrard@math.toronto.edu}

\maketitle

\begin{abstract} We introduce a framework to study the occurrence of vortex filament concentration in $3D$ Ginzburg-Landau theory. We derive a functional that describes the free-energy 
of a collection of nearly-parallel quantized vortex filaments
in a cylindrical $3$-dimensional domain, in certain
scaling limits; it is shown to arise
as the $\Gamma$-limit of a sequence of scaled
Ginzburg-Landau functionals. Our main result establishes for the first time a long believed connection between the Ginzburg-Landau functional and the energy of nearly parallel filaments that applies to many mathematically and physically relevant situations where clustering of filaments is expected. In this setting it also constitutes a higher-order asymptotic expansion of the Ginzburg-Landau energy, a refinement over the arclength functional approximation. Our description of the vorticity region significantly improves on previous studies and enables us to rigorously distinguish a collection of multiplicity one vortex filaments from an ensemble of fewer higher multiplicity ones. 
As an application,
we prove the existence of solutions of the Ginzburg-Landau equation
that exhibit clusters of vortex filaments whose small-scale
structure is governed by the limiting free-energy
functional. 
\end{abstract}

\maketitle

\section{Introduction}


Let $\Om\subseteq \R^3$ and $\e>0$ small.
For $u\in H^1(\Om;\C),$ the Ginzburg-Landau energy is given by
\beq
F_\e(u; \Omega) = F_\ep(u) :=\int_\Om e_\e(u) \ ,
\qquad
e_\e(u):=\frac 12\abs{\nabla u}^2+\frac{1}{4\e^2}(1-\abs{u}^2)^2.
\eeq
We want to derive an effective interaction energy for $n \ge 2$ vortex filaments in the context of Ginzburg-Landau theory; this will allow us to prove the existence of solutions of the 
Ginzburg-Landau equations  --- that is, critical points of $F_\ep(\cdot)$ --- with a particular geometric structure that we detail  below.

A model setting for studying  nearly parallel vortex filaments is one found in fluid dynamics  \cite{KMD} which we adopt. 
Thus, we will always consider $\Omega$ of the form
\begin{equation}\label{formofOmega}
\Om:=\om\times(0,L), \qquad\quad \omega\subset \R^2 \mbox{ bounded, open, simply connected, $\partial \omega$ smooth.}
\end{equation}
Throughout this paper, we write points in $\Omega$ in the form $(x,z)$
with $x\in \omega$ and $z\in (0,L)$.
We always assume that $0\in \omega$,
and the configurations of interest to us are those with $n$ vortex lines close to the vertical $\{0\}\times(0,L).$

Given $u = u^1+iu^2\in H^1(\Om;\C),$ we define
the {\em momentum} and {\em vorticity} vector fields\footnote{The vorticity can also be defined as a $2$-form, and indeed this is the perspective we will adopt throughout most of this paper. 
}, denoted $j(u)$ and $\mathcal Ju$ respectively, by
\begin{equation}\label{Ju.def}
j(u)  := Im(\bar u \,\nabla u), \qquad \mathcal Ju := \frac 12 \nabla\times \, j(u) = \nabla u^1 \times \nabla u^2.
\end{equation}
It is known  (see  \cite{MSZ, abo}) that for every $n\in \mathbb N$, 
there exist solutions
$(u_\e)$ of the Ginzburg-Landau equations for which the energy 
and vorticity concentrate around  $\{0\}\times (0,L)$ in the sense that
\begin{equation}
\int_\Omega \phi \frac{e_\e(u_\e)}{\logeps}dx\, dz \rightarrow
n\pi \int_0^L \phi(0,z) \ dz,  \qquad \mbox{for all }\phi \in C(\Omega)
\label{econn}\end{equation}
and
\begin{equation}
\int_\Omega \varphi\cdot \mathcal Ju_\ep\,  dx\,dz\ \rightarrow  n \pi \int_0^L \varphi(0,z)  \cdot e_z   dz
\qquad\mbox{ for all  $\varphi \in C^1_c(\Omega;\R^3)$}, 
\label{vconn}\end{equation}
where $e_z$ is the standard unit vector in the $z$ direction. 
These are interpreted 
as stating that the solutions $(u_\e)$ exhibit $1$ or more vortex
filaments, carrying a total of n quanta of vorticity, clustering near
the segment $\{0\}\times (0,L) .$

Our results give a precise description of the way in which this clustering occurs. 
In particular for $0<\ep\ll 1$ we find solutions with the following properties: 
\begin{itemize}
\item each solution possesses $n$ distinct filaments, identified
as curves along which the vorticity concentrates, each of multiplicity
$1$ (rather than a smaller number of filaments of higher multiplicity).
\item these filaments are separated by distances of order $|\log\ep|^{-1/2}$.
\item after dilating horizontal distances by a factor of $\logeps^{1/2}$, 
the limiting geometry of the vortex filaments is governed by a particular free-energy functional,
see \eqref{G0.def} below.
\end{itemize}

For the solutions we find,  the limiting vorticity (after rescaling in the horizontal variables)
is concentrated on $n$ curves of the form
\[
z\in (0,L) \ \ \longmapsto \ \ (f_i(z),z) 
\]
where the function $f = (f_1,\ldots, f_n)$ minimizes
\begin{equation}\label{G0.def}
G_0(f) :=\pi \int_0^L \left(\frac 12 \sum_{i=1}^n| f_i'|^2 -\sum_{i\ne j}\log|f_i-f_j| \right) dz
\end{equation}
in $H^1\left((0,L);(\R^2)^n\right)$,
subject to certain boundary conditions. 
The length scale of vortex separation
\begin{equation}
h_\ep = \logeps^{-1/2}
\label{hep.def}\end{equation}
 is critical in the sense that it gives rise to a limiting functional in which the $\frac 12 |f'|^2$ term and the logarithmic repulsion terms roughly balance.


\medskip

The solutions we find with the above properties will be obtained as 
minimizers and local minimizers of the Ginzburg-Landau energy with
suitable boundary conditions. The description of the fine structure
of the vorticity in these solutions will be deduced from 
a detailed asymptotic description of the energy and vorticity of
sequences $(u_\ep) \subset H^1(\Omega;\C)$ of functions
with $n$ vortex filaments clustering on a scale $h_\ep$
around the segment
$\{0\}\times (0,L)$. 
Very roughly speaking, we will prove that in certain regimes,
if $(u_\ep)_{\ep\in (0,1]}$ is a sequence with limiting rescaled vorticity 
described by $f\in H^1((0,L); (\R^2)^n)$, then
\begin{equation}\label{rough}
`` F_\ep(u_\ep) \approx \mbox{ logarithmically divergent term } + G_0(f) '' \quad\mbox{ as }\ep\to 0.
\end{equation}
See Theorem \ref{main} below for a precise statement. 
Thus, the functional $G_0$ may be seen as an
asymptotic energy associated to the family
$(F_\ep)_{\ep\in (0,1]}$, after renormalizing by subtraction
of the divergent term.
In fact, $G_0(\cdot)$ has already been identified as a candidate 
for the asymptotic renormalized energy of a family of nearly
parallel vortex filaments in \cite{DK}. A related effective
energy funtional is found via
formal arguments, and in a somewhat different setting, in \cite{Aft-Riv}.

The divergent term in \eqref{rough} is related to the arclength (with multiplicity)
of the limiting vorticity. This
reflects  the well-known connection between
the Ginzburg-Landau energy and the arclength of limiting
vortex filaments. Numerous specific results of this sort are known,
including for example \cite{Riv, S2, LR, BBO,  JSo,  BBM,
 abo}. In this context, the term $\frac \pi 2\int_0^L \sum_{i=1}^n |f_i'|^2 dz$
 in $G_0(\cdot )$ may be seen as the linearization of arclength, the leading-order
 asymptotic energy.

In $2D$, an asymptotic
expansion of the energy, in the spirit of \eqref{rough},
was first carried out in the seminal work
of Bethuel, Brezis and H\'elein  \cite{BBH} and later extended to several other two-dimensional contexts (see for example \cite{SandSerfbook}). The term
$-\pi \int_0^L\sum_{i\ne j} \log|f_i-f_j| dz$ in $G_0(\cdot)$ in essence arises from 
a fundamental object, the $2D$ ``renormalized energy", introduced in
\cite{BBH}.

Until now, higher dimensional counterparts of the results of Bethuel {\em et al}
\cite{BBH}  with a comparable degree of precision have been very elusive. 
The corresponding 3D results -- the rigorous version of \eqref{rough}, 
stated in Theorem \ref{main} ---  are the main
contribution of this paper.
However, since they require considerable notation
to state, we first describe our existence theorems  in more detail.

\subsection{Solutions of the Ginzburg-Landau equations with vortex clustering}\label{S1.1}

We will study minimizers and local minimizers of the Ginzburg-Landau energy $F_\ep(\cdot \;;\Omega)$, for $\Omega = \omega\times (0,L)$ as described in \eqref{formofOmega}, with Dirichlet data on $\omega \times  \{0 , L\}$ and natural boundary conditions on $\partial \omega \times (0,L)$; see \eqref{dirichlet} below
for the precise formulation. The Dirichlet conditions may be understood
to require that $n$ vortex filaments enter and leave $\Omega$ at certain points
within a distance at most $O(h_\ep)$ from  the ends of the segment
$\{0\}\times (0,L)$. 

More precisely, we will consider boundary data $w_\ep^z$, for $z\in \{0,L\}$, of the form
\begin{equation}\label{formofw}
w_\ep^z(x) = \prod_{j=1}^n  \Big[ \exp\left( i \beta(x, p_{\ep, j}(z)) \right)
\zeta_\e(x- p_{\ep, j}(z))\Big]
\end{equation}
where
\begin{itemize}
\item $\beta(\cdot, \cdot)$ is defined so that $\frac {w^z_\ep}{|w^z_\ep|}$
is {\em exactly} the canonical harmonic map of Bethuel {\em et al} (see \cite{BBH},
section I.3) with singularities $(p_{\ep,j}(z))_{j=1}^n$
and natural boundary conditions on $\partial \omega$;
see \eqref{beta.def} for the details.
\item $\zeta_\e$ has the form $\zeta_\ep(x) = \rho_\ep(|x|)  \frac{ x_1+ix_2}{|x|}$,
and $\rho_\ep:[0,\infty)\to [0,1]$ satisfies
\begin{equation}\label{zetaep}
\rho_\ep(0) = 0, \qquad 0 \le \rho_\ep' \le C/\ep, \qquad \rho_\ep(s) \ge \left(1 - C \ep/s\right)_+
\end{equation}
\item $(p_{\ep,j}(z))$ are sequences in $\omega$ such that 
\begin{equation}\label{qep.lim}
q_{\ep,j}(z) := h_\ep^{-1} p_{\ep, j}(z)  \longrightarrow  q^0_j (z)\qquad\mbox{ as $\ep\to 0$,  \ for some $q_j^0(z), j=1,\ldots, n$. }
\end{equation}
\end{itemize}


Note that we  do {\em not} assume that the points $(p_{\ep, j}(z))$ are
distinct. For example, $p_{\ep, j}(z) = 0$ for all $j$ and for $z = 0,L$
is allowed by our assumptions.

In considering minimizers, we will assume that $\Omega = \omega\times (0,L)$
satisfies
\begin{equation}\label{Gzero.min}
L < 2\dist(0,\partial\omega) .
\end{equation}
This condition is close to necessary; see Remark \ref{R.nec} below. 

Throughout this paper we  write $B(r,x)$ or $B_r(x)$ to denote the {\em open} ball in $\R^2$ of radius $r$ and center $x$,
and $B(r) := B(r,0)$.

We present our first result relating minimizers of $F_\e$ to those of the reduced functional $G_0 .$ The $W^{-1,1,}$ norm, which appears in the statement, is
defined in \eqref{FvsW11} below.

\begin{theorem}\label{minimizers}
Assume that $\Omega = \omega\times(0,L)$ satisfies \eqref{Gzero.min},
and that  $(u_\ep)$ minimizes $F_\ep(\cdot ; \Omega)$ in the space 
\begin{equation}
\mathcal A_\ep := \{ u\in H^1(\Omega;\C) : \ u(x, 0) = w_\e^0(x), \ \ u(x,L) = w_\e^L(x) \}
\label{dirichlet}\end{equation}
for a sequence of boundary data $\{ w_\ep^{0}, w_\ep^L\}_{\ep\in (0,1]}\subset H^1(\omega;\C)$ 
as described in \eqref{formofw} -- \eqref{qep.lim}.

\smallskip

Then setting $v_\ep(x,z) = u_\ep(h_\ep x,z)$,
the  vorticities $\{  \mathcal Jv_\ep\}_{\ep\in (0,1]}$
are precompact in $W^{-1,1}(B(R)\times (0,L))$ for every $R>0$,
and any limit $J^*$ of a convergent subsequence, 
as $\ep\to 0$,
is a vector-valued measure of the form
\begin{equation}
\int \varphi \cdot dJ^* \ = \ 
\pi \sum_{i=1}^n \int_0^L \varphi(\gamma^*_i(z)) \cdot \gamma^*_i{}'(z) \, dz 
\qquad
\mbox{ for }\varphi \in C_c(\R^2\times (0,L); \R^3).
\label{formJstar}\end{equation}
Here $ \gamma^*_i(z) = (f_i^*(z),z)$, and
$f^*= (f^*_1,\ldots, f^*_n)\in H^1((0,L);\R^2)$ minimizes $G_0(\cdot)$
in 
\begin{equation}
\mathcal A_0 := \left\{ f\in H^1((0,L); (\R^2)^n ) \ : \sum_i \delta_{f_i(z)} = \   \sum_i \delta_{q^0_i(z)} \mbox{ for }z \in \{0,L\}  \right \}
\label{Azero.def}\end{equation}
for $q^0_i(z)$ appearing in \eqref{qep.lim}.

 \end{theorem}

After we present Theorem \ref{localmin} about local minimizers, we discuss both results in the context of $3D$ Ginzburg-Landau theory and on the study of concentration phenomena in elliptic PDE's at large.

\begin{remark}\label{R.nec}
Given any $\bar x\in \partial \Omega$ and $\delta<L/2$,
one can construct 
a sequence of functions $(u_\ep)\subset H^1(\Omega;\C)$ satisfying the 
boundary conditions of Theorem \ref{minimizers} above,
with energy and vorticity concentrating around the
straight line segments connecting
$(0,0)$ to $(\bar x,\delta)$ and $(\bar x, L-\delta)$ to $(0,L)$,
and such that
\[
\lim_{\ep\to 0} \frac 1 \logeps\int_\Omega e_\e(u_\ep) =  n \pi 2(|\bar x|^2+\delta^2)^{1/2}.
\]
This follows from results in \cite{abo}.
In particular, if $\dist(0,\partial \omega)< \frac 12 L$, one can choose $\bar x$ and $\delta$
so that 
the right-hand side is strictly less that $n \pi L$. Then Theorem
\ref{main} below implies that for any minimizing sequence,
vorticity cannot concentrate around
$\{0\}\times(0,L)$. 

\end{remark}

Next we state a result about local minimizers.
We will say that $f^*\in \mathcal A_0$ is a strict local minimizer of $G_0$
if there exists $\delta>0$ such that 
\[
G_0(f^*) < G_0(f)\quad \mbox{ for all $f\in \mathcal A_0$ such that $0< \|f - f^*\|_{H^1((0,L);(\R^2)^n)} < \delta$ }.
\] 
The point is that we understand ``local" with respect to the natural topology
which here
is $H^1((0,L); (\R^2)^n)$. Similarly, $u_\ep$ is a strict local minimizer of $F_\ep$
in $\mathcal A_\ep$ if there exists $\delta>0$ such that 
\[
F_\ep(u_\ep) < F_\ep(u)\quad\mbox{ for all $u\in \mathcal A_\ep$ such that  }0< \| u - u_\ep\|_{H^1(\Omega;\C)} < \delta \, .
\]
In particular, a local minimizer $u_\ep$ of $F_\ep$ is a solution of the Ginzburg-Landau equations,

\begin{equation}\label{eqn.GL}
\boxed{-\Delta u_\e =\frac{1}{\e^2}(1-\vert u_\e\vert^2) u_\e ,}
\end{equation}
and a local minimizer $f^*$ of $G_0(\cdot)$ satisfies 
\begin{equation}\label{eqn.NPVF}
\boxed{- f_i^*{}'' - \sum_{j\ne i}\frac{f_i^* - f_j^*}{|f_i^*-f_j^*|^2} = 0 \qquad \qquad\mbox{ for }i=1,\ldots, n.}
\end{equation}

The system \eqref{eqn.NPVF} appears in various contexts; its solutions are equilibria of reduced systems in fluid mechanics \cite{KMD, KPV, LM} and supplemented with suitable periodic conditions they correspond to trajectories in the planar $n$-body problem with logarithmic potential (examples of periodic orbits with or without collisions and for different potentials may be found in \cite{CM, FT, Chen, BFT}). 

\begin{theorem}\label{localmin}
Let $(w_\ep^{0,L})_{\e\in (0,1]} \subset H^1(\omega;\C)$ be sequences satisfying \eqref{formofw} -- 
\eqref{qep.lim}, and 
assume that $f^*$ is a strict local minimizer of $G_0$ in  $\mathcal A_0$, and that $f^*_i(z)\ne f^*_j(z)$
for all $i\ne j$ and $z\in (0,L)$.

Then there exists a sequence of local minimizers of $F_\ep(\cdot;\Omega)$ in
$\mathcal A_\ep$ such that for $v_\ep(x,z) := u_\ep(h_\ep x,z)$,
the vorticities $\mathcal  Jv_\ep$ converge in $W^{-1,1}(B(R)\times (0,L))$
for all $R>0$ to the vector-valued measure of the form \eqref{formJstar}, where
$\gamma^*_i(z) = (f_i^*(z),z)$ for $i=1,\ldots, n$.
 \end{theorem}
 
Theorems \ref{minimizers} and \ref{localmin} are the first results that show existence of solutions to the Ginzburg-Landau equation \eqref{eqn.GL} whose vorticity asymptotically concentrates along the graphs of solutions of the system \eqref{eqn.NPVF}. A connection between \eqref{eqn.GL} and \eqref{eqn.NPVF} has long been suspected and supporting evidence can be found in \cite{MSZ,DK, JSt} and the references therein. It can be seen from the proofs that our results do not only describe the asymptotic geometry of the vorticity, but as a by-product also give a very precise asymptotic expansion of the energy of the maps $u_\e$ in terms of the functional $G_0$ from \eqref{G0.def}, which may be seen as a $3D$ renormalized energy of the vortex filaments, in the spirit of the $2D$ renormalized
energy introduced in \cite{BBH}.  One of our achievements here is an improved compactness for the vorticity of configurations satisfying \eqref{scaling1}--\eqref{scaling1aa} below. 

In the scalar setting of the Allen-Cahn equation, recent results that describe interface clustering have been obtained, 
see for example \cite{dPKW, dPKPW, dPKWY}. In the scalar case interfaces are of codimension $1;$ higher condimension defects introduce new difficulties we  have to deal with when proving our results, which may be seen as first analogs of the clustering phenomenon for the vector-valued Ginzburg-Landau equation.

We note that Theorem \ref{localmin} does not require condition
\eqref{Gzero.min}, and also has the additional advantage of not requiring that
$f^*_i(z)\ne f^*_j(z)$ for $i\ne j$, when $z \in \{0,L\}$. 
We remark however that if one wants to allow collisions between filaments (that is, values of $z\in (0,L)$ for which
$f^*_i(z) = f^*_j(z)$ for some $i\ne j$),
then the right definition of ``local minimizer" for
our purposes would become more complicated, since there are then
multiple essentially different $f\in H^1((0,L);(\R^2)^n)$ that represent
the same vortex paths. As a result, one would need a notion of local minimizers
in a suitable quotient space of $H^1((0,L);(\R^2)^n)$. 
We prefer to avoid these technicalities here, since we do not know any
examples of local minimizers, in this sense, for which the filaments collide.
However, related considerations appear in the proofs of
Lemmas \ref{L.X} and \ref{L.choosef} for example.

\subsection{Some definitions and notation}\label{somen}

As remarked above, the vorticity can be realized as either a vector field or a $2$-form.
For $u = u^1+iu^2$, we will write 
\[
Ju = du^1 \wedge du^2  = \frac 12 d (Im(\bar u \, du))
\]
for the vorticity $2$-form, compare \eqref{Ju.def}.

It is convenient to state some of our results in the language of 
geometric measure theory.

If $U$ is an open subset of some Euclidean space, then a $k$-current  $U$
is a bounded linear functional on the
space $\calD^k(U)$ of smooth $k$-forms with compact support in $U$.

We will often encounter $1$-currents (as well as $0$-currents, which
can be identified with distributions).
We will be especially interested in some particular classes of 
$1$-currents.
First, to any Lipschitz curve $\gamma:(a,b)\to U$, there is a 
corresponding $1$-current $T_\gamma$, defined by
\[
T_\gamma(\vp):= \int_\gamma \vp \ = \ 
\int_a^b \langle \vp(\gamma(t)) , \gamma'(t)\rangle dt.
\]
Here and below, $\langle \cdot , \cdot \rangle$
denotes the dual pairing between  covectors and vectors, so that
if $\vp = \phi_1 dx^1+ \phi_2 dx^2 + \phi_3 dz$ and $v = (v^1,v^2, v^3)$,
then 
\[
\langle \varphi , v \rangle = \ \phi_i  v^i .
\]
(Throughout this work we implicitly sum over repeated indices.)
Thus, $T_\gamma$ acts on a $1$-form $\vp$ via integration of $\vp$ over the curve
parametrized by $\gamma$.

Given a Lipschitz function $f:(0,L)\to \R^2$, we will write 
\begin{equation}\label{Gammaf.def}
\Gamma_f := T_\gamma \quad\mbox{ for $\gamma:(0,L)\to \R^2\times(0,L)$
defined by $\gamma(z) = (f(z),z).$}
\end{equation}
Thus, $\Gamma_f$ is the $1$-current in $\R^2\times (0,L)$ corresponding to
the graph of $f$ over the segment $(0,L)$ of the $z$-axis.
More generally, if  $f\in H^1((0,L);\R^2)$, then we define
$\Gamma_f = T_\gamma$, where $\gamma$ is a Lipschitz reparametrization
of the curve $z\mapsto (f(z),z)$. Such a reparametrization exists,
since the condition $f\in H^1$ guarantees that the curve has finite arclength.

As a special case of the above, we will write $\Gamma_0$ to denote the current corresponding to the
vertical segment $z\in (0,L)\mapsto (0,z)\in \omega\times (0,L)$. With this notation, the limit in \eqref{vconn}
is written as $n\pi \Gamma_0(\vp)$ (if we view it as acting on $1$-forms rather than vector fields).

\smallskip

The other class of $1$-currents that often arises in this paper
is the following:
given $u = u^1+iu^2\in H^1(\Omega; \C)$, we will write $\star Ju$ to denote the
$1$-current defined by
\beq\label{starJ.def}
\star Ju(\vp):=\int_\Om \langle \vp,  \mathcal Ju \rangle  dx
\ = \ \int_\Omega\varphi\wedge Ju , \qquad \qquad \vp \in \calD^1(\Omega).
\eeq

In what follows, it will often be the case that most of the information
encoded in the vorticity $Ju$ (or its distributional
realization $\star Ju$) is already contained in its  $z$ component, that is, the 
part of the vorticity that describes rotation in the $x_1x_2$ plane, orthogonal to the
$z$-axis. This will be denoted by 
\[
J_x u := \partial_1 u^1 \,\partial_2u^2 - \partial_1u^2 \, \partial_2u^1 .
\]
This is  the Jacobian determinant  with respect to the $x$ variables. 
Observe that 
if $\vp \in \calD^1(\Omega)$
has the form $\vp= \phi \, dz$, for $\phi$ a smooth compactly supported function, then
\begin{equation}\label{starJ.Jx}
\star Ju(\phi \, dz) = \int_\Omega \phi(x,z) J_xu(x,z) dx\, dz
\end{equation}
as follows directly from the definitions \eqref{Ju.def},  \eqref{starJ.def}.

\smallskip

If $S$ is  a  $k$-current on an open subset $U$ of a Euclidean space, we use the notation
\beq\label{F.def}
\norm{S}_{F(U)}:=
\sup\{S(\vp):\vp\in D^k(U), \max(\norm{\vp}_\infty,\norm{d\vp}_\infty)\leq 1\}
\eeq
for the {\em flat norm} of $S$. This quantity will be finite
for every current $S$ that arises in this paper.

We note that a $0$-current $S$ on a set $U\subset \R^m$ is just a distribution -- that is,
a bounded linear functional on the space $\calD^0(U)$ of smooth, compactly supported  $0$-forms, 
or functions. 
If $\vp$ is a $0$-form, then $\| d\vp \|_\infty = \|\nabla \vp \|_\infty$, 
and it follows that
\begin{align}
\| S \|_{F(U)} 
&=\sup \{ S(\vp)  : \vp\in D^0(\Om), \max(\norm{\vp}_\infty,\norm{\nabla\vp}_\infty)\leq 1\} 
\nonumber\\
&=: \| S\|_{W^{-1,1}(U)},
\label{FvsW11}
\end{align}
so we will sometimes use these two notations interchangeably for $0$-currents.
For a $1$-current, the flat norm is somewhat stronger than the $W^{-1,1}$ norm.

The {\em mass} of a $k$-current $S$ on an open set $U$ is defined by
\[
M(S) = 
\sup\{S(\vp):\vp\in D^k(U), \norm{\vp}_\infty \leq 1\} . 
\]
If $\gamma$ is an  injective Lipschitz curve then it is easy to check that
\[
M(T_\gamma) = \mbox{length}(\gamma).
\]

\subsection{Configurations with nearly parallel vortex lines}

As suggested above, our main PDE results will be obtained
from a careful study of the energy $F_\e(u_\ep)$
for  certain sequences of functions $(u_\ep)\subset H^1(\Omega;\C)$
with properties that will be
easily verified in PDE applications.

First, we are interested in sequences
that exhibit $n\ge 2$ vortex lines clustering around the segment $\{0\}\times (0,L)$.
More precisely, we will assume that
\begin{equation}
\label{scaling1}
\int_0^L \| J_x u_\ep (\cdot, z) - \pi n \delta_0\|_{F(\om)}\, dz\to 0
\qquad\mbox{ as }\e\to 0.
\end{equation}

\begin{remark}\label{rem2}
We show  below that 
$\int_0^L \| J_x u_\ep (\cdot, z) - \pi n \delta_0\|_{F(\om)}\, dz \  \le \| \star Ju_\ep  -n\pi \Gamma_0\|_{F(\Omega)}$, see Section \ref{S.currents}.
As a result, \eqref{scaling1} follows from
the assumption
\[
\| \star Ju_\ep - n \pi \Gamma_0\|_{F(\Omega)}  \rightarrow 0  \qquad\mbox{ as }\e\to 0,
\]
which may appear more natural than \eqref{scaling1}.
\end{remark}

We will also assume that there is at least one height $z_0\in [0,L]$
and points $q^0_1,\ldots, q^0_n\in \R^2$, not necessarily distinct, such that 
\begin{equation}
\|  J_x u_\e(\cdot, z_0) - \pi\sum_{i=1}^n \delta_{h_\ep q^0_i}\|_{F(\omega)}  = o(h_\ep)
\qquad\mbox{ as $\ep\to 0$, \ \ and }
\label{scaling1a}
\end{equation}
\begin{equation}
\int_\omega e^{2d}_\e (u_\ep)(x,z_0) \ dx \le M \logeps
\qquad\mbox{ for some }M>0,
\label{scaling1aa}
\end{equation}
where
\[
e_\e^{2d}(u):=\frac 12\abs{\nabla_xu}^2+\frac{1}{4\e^2}(1-\abs{u}^2)^2,
\quad\qquad \nabla_x:=(\p_{x_1},\p_{x_2}).
\]
Condition \eqref{scaling1a} implies that at the height $z_0$, there are exactly 
$n$ vortices, all  within distance $O(h_\ep)$  --- the critical length scale --- 
of the vertical axis.
The energy bound \eqref{scaling1aa} acts to ensure that
the behaviour of $u_\e$ at height $z_0$ carries meaningful information about its
behaviour at nearby heights.

\begin{center}
\begin{figure}[!ht]
\begin{minipage}{0.8\linewidth}
\includegraphics[trim = 0mm 0mm 0mm 0mm, clip, width=10cm,
  height=10cm,
  keepaspectratio,]{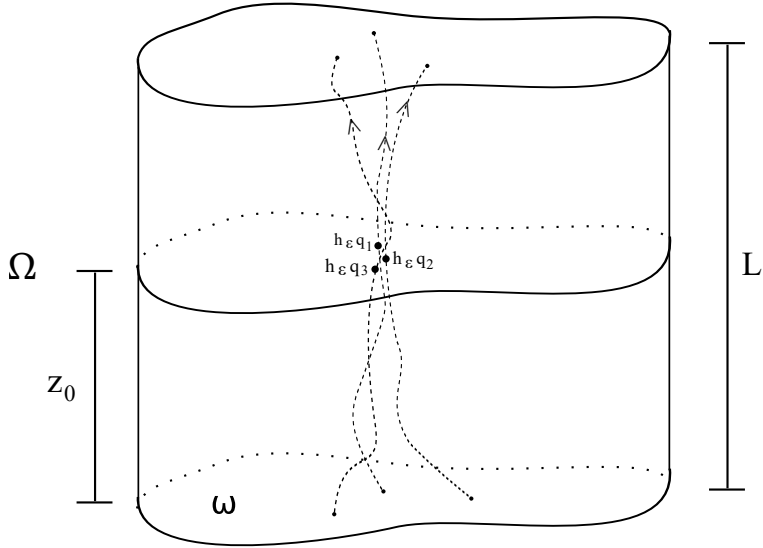}
\end{minipage}
\caption{Example of a vortex configuration satisfying \eqref{scaling1}--\eqref{scaling1aa}.}
\end{figure}
\end{center}

It turns out  --- this is a consequence of Theorem \ref{main} below --  
that under the above assumptions, 
\[
F_\e(u_\e) \ge   n\pi L\abs{\log\e}+\pi n(n-1)L\abs{\log h_\e} - O(1).
\]
We therefore introduce
\begin{equation}\label{Gep.def}
G_\e(u):=F_\e(u)
-\left[ n\pi L\abs{\log\e}+\pi n(n-1)L\abs{\log h_\e}+\kappa_n(\Om)\right],
\end{equation}
where the constant $\kappa_n(\Omega)$ is defined in \eqref{kappan.def} below; it is connected to the {\em renormalized energy} introduced in \cite{BBH}.
Finally, we will assume\footnote{We will number  certain specific constants $c_1, c_2, \ldots$ that appear repeatedly in our arguments,
whereas other generic constants will be denoted just by $C$. Throughout, all 
constants are independent of $\e$, but may depend for example on $\Omega$ and on
parameters such as $c_1$.}
 that there exists some constant $c_1$ such that 
\begin{equation}
\label{scaling2}
G_\e(u_\ep) \le c_1\  . 
\end{equation}
This is a stringent  energy bound that will be shown to require that
the vortices are nearly parallel. Our arguments will show that
once there is some height
$z_0$ 
at which there are $n$ vortices at distance $O(h_\ep)$ from
the $z$ axis, as in \eqref{scaling1a}, \eqref{scaling1aa}, the energy bound \eqref{scaling2} together with \eqref{scaling1}
essentially forces the filaments to remain
$O(h_\ep)$ from the $z$-axis throughout their entire length.

We remark that \eqref{scaling2} implies the much less precise bound
\begin{equation}\label{ref.asked}
F_\e(u_\ep) \le C \logeps
\end{equation}
which is used numerous times throughout this paper.

Assumptions  \eqref{scaling1} - \eqref{scaling1aa} above
are adapted to the study of  problems for which 
boundary data is prescribed on $ \omega\times \{0,L\}$.
Natural assumptions on the boundary data, such as those 
described in \eqref{formofw} -- \eqref{qep.lim}, then guarantee that
\eqref{scaling1a}, \eqref{scaling1aa} are satisfied. 
Most of our results
will remain valid if one does not assume \eqref{scaling1a}, \eqref{scaling1aa}, 
but assumption \eqref{scaling1} is replaced by the stronger condition
\begin{equation}\label{old.scaling1}
\int_0^L \| J_x u (\cdot, z) - \pi n \delta_0\|_{F(\om)}\, dz
\le C h_\ep \,  .
\end{equation}
This is often adequate for the construction of {\em local} minimizers of $F_\ep$,
and under this assumption some of our arguments
can be simplified. However, \eqref{old.scaling1} is hard to
verify directly for sequences of energy-minimizers,
as considered in Theorem \ref{minimizers}.                                                                                                                                          
As in Remark \ref{rem2}, assumption \eqref{old.scaling1}
follows directly if one instead assumes the
(arguably more natural) condition
$\| \star Ju_\ep - n\pi \Gamma_0\|_{F(\Omega)}
\le C h_\ep$.

\subsection{Asymptotic behavior  of energy and vorticity}

Our main result is the $\Gamma$-convergence of $G_{\e}$ to $G_0.$

\begin{theorem}
(a) 
Assume that $(u_\ep)\subset H^1(\Omega;\C)$ is a sequence satisfying 
\eqref{scaling2}, together with either
\eqref{scaling1} -- \eqref{scaling1aa}  or \eqref{old.scaling1}.

Then, setting $v_\e(x,z) = u_\e(h_\e x, z)$ for $x\in \omega_\ep = h_\e^{-1}\omega$
and $z\in (0,L)$, there exists some $f = (f_1,\ldots, f_n)\in H^1((0,L), (\R^2)^n)$ such 
that  after passing to a subsequence
if necessary:
\begin{equation}\label{comp}
J_x v_\e \to\pi \delta_{[ f(\cdot)]} \quad  \mbox{ in $W^{-1,1}(B(R)\times(0,L))$ for every $R>0$}
\end{equation}
where $\delta_{[f(\cdot)]}$ denotes the measure on $\R^2\times (0,L)$ defined by
\begin{equation}
\int \phi \delta_{[f(\cdot)]} :=  \sum_{i=1}^n \int_0^L \phi( f_i(z), z) \ dz \ = \ 
\sum_{i=1}^n \Gamma_{f_i} (\phi \, dz).
\label{deltavecf}\end{equation}
Moreover, there exists some $\sigma\in S^n$ (the symmetric group of degree $n$) such that $f_i(z_0) = q^0_{\sigma(i)}$ for
$i=1,\ldots, n$, where
$z_0,  (q^0_i)$ appear in \eqref{scaling1a}. And finally,
\begin{align}
\label{uniflowse}
G_0(f) \le \liminf_{\e\to 0} G_\e(u_\e).
\end{align}

(b) Conversely, given $f\in H^1((0,L), (\R^2)^n)$, there exists $(u_\e)\subset H^1(\Omega,\C)$
such that \eqref{comp} holds, together with \eqref{scaling1} -- \eqref{scaling1aa},
and in addition \footnote{In the special case when $\omega$ is a ball, results similar to those of part $(b)$ above
are proved in \cite{DK}.}
\begin{equation}\label{recseqbd}
G_0(f) \ge \limsup_{\e\to 0} G_\e(u_\e).
\end{equation}

(c) In addition, whenever $(u_\e)$ satisfies \eqref{comp}, \eqref{recseqbd},
we have the improved compactness
\begin{equation}\label{ICITCOT}
\| \star Jv_\ep -  \pi  \sum_{i=1}^n \Gamma_{f_i}\|_{F(B(R)\times (0,L))}  \rightarrow 0  \ \ \mbox{ as }\e\to 0,\qquad\mbox{ for every }R>0.
\end{equation}
\label{main}
\end{theorem}

In its full strength, Theorem \ref{main} does not only lend itself to applications such as Theorems \ref{minimizers} and \ref{localmin} but it also provides a framework to analyze clustering of filaments in more general $3D$ Ginzburg-Landau problems. Instances where this occurrence is expected are considered in the works \cite{MSZ, DK}. Moreover, the lengthscale $h_\e$ is rather natural and it can be imposed by the geometry \cite{MSZ} or by physical effects as in \cite{AC.ARMA} where distinct filaments are shown to concentrate at this lengthscale due to an external magnetic field in a much simpler setting. It is also a starting point towards establishing a correspondence between solutions, not necessarily stable, of \eqref{eqn.NPVF} and \eqref{eqn.GL} (see \cite{JSt} for an appropriate strategy).

 Some of the most salient features of the $\Gamma$-convergence of $G_\e$ to $G_0$ that may be applied to other situations are the accurate characterization of the vorticity and the expansion of the energy $F_\e(u_\e)$ up to $o(1)$ for the solutions predicted in Theorems \ref{minimizers} and \ref{localmin}. Following \cite{JSo, abo} we would be only able to conclude that
\[\frac{1}{\pi}\star J u_\e \to n \Gamma_0 \mbox{ in }W^{-1,1}(\Om),\]
and that
 \[F_\e (u_\e)= n\pi L\vert\log \e\vert +o(\vert\log \e\vert) .\] In comparison \eqref{ICITCOT} allows us to identify the vorticity in a very precise manner that in particular lets us distinguish $n$ distinct filaments, while \eqref{uniflowse} and \eqref{recseqbd} capture a renormalized energy of the filaments that stores fine information about the geometry of the filaments and their mutual interaction. This realization of minimizers of $F_\e$ is a $3D$ analog, in the present setting, of the asymptotic description in \cite{BBH}.

Theorem \ref{main} is also robust in the sense that rather small adaptations of the proof
could be used to establish variants, such as a corresponding
$\Gamma$-convergence result with Dirichlet boundary conditions on 
$\partial \omega\times (0,L)$. For example, given
$g\in C^\infty(\partial \omega; S^1)$  of degree $n$, 
the conclusions of Theorem \ref{main} still
hold for a sequence
$(u_\ep)$ in
\begin{equation}\label{d0}
\{ u\in H^1(\Omega;\C) : u(x,z) = g(x)\mbox{ for }x\in \partial \omega, z\in (0,L)\}
\end{equation}
satisfying the hypotheses of Theorem \ref{main},
as long as the constant $\kappa_n(\Omega)$
appearing in the definition \eqref{Gep.def} 
of $G_\ep$ is changed 
to a new constant $\kappa_n(\Omega;g)$,
which we display in \eqref{knog.def}.
Briefly, the point is that $\kappa_n(\Omega)$
is constructed from the Bethuel-Brezis-H\'elein
renormalized energy on $\omega$ with natural boundary conditions,
whereas $\kappa_n(\Omega;g)$ uses the renormalized
energy on $\omega$ with Dirichlet data $u=g$ on $\partial\omega$. The only places where this requires any changes in the 
proof of Theorem \ref{main} are the following:
\begin{enumerate}
\item  the proof of Lemma \ref{L3b}. Here
\eqref{L3a.sharper} would need to be modified
by changing the natural (Neumann) renormalized energy 
to the Dirichlet renormalized energy for $g$, which is then
carried through the rest of the proof. 

We remark that the verification of \eqref{L3a.sharper} 
relies on Theorem 2  in \cite{JSp2}. This result
remains true for Dirichlet boundary conditions, with
only cosmetic changes in the proof, although we do not know
any reference that presents all the details.

\item the construction of the recovery
sequence in Section \ref{section5}.  Here one
would  need to build the recovery sequence out of the
canonical harmonic map with Dirichlet, rather than 
Neumann, boundary conditions. This would simply entail
a change in the boundary conditions for the phase factor $\beta$,
defined in \eqref{beta.def}.
\end{enumerate}
More generally, if one fixes $A\subset \partial \Omega$ and considers
a sequence of functions of the form
\begin{equation}
(u_\ep)\subset \{ u\in H^1(\Omega;  \C) \ : \ u = g_\ep \mbox{ on }A\}
\label{formofdata}\end{equation}
and satisfying the hypotheses of Theorem \ref{main},
then one expects parallel results to hold, if
$\kappa_n$ in the definition of $G_\ep$ is modified to a suitable 
$\kappa_n(\Omega; A; g_\ep) \ge \kappa_n(\Omega)$.

In particular, we consider boundary data of the form \eqref{formofdata}
in Theorems \ref{minimizers} and \ref{localmin}, with $A = \omega \times \{0,L\}$.
In order to avoid technicalities related to estimates of the impact of
general boundary data on the energy, we focus on
carefully chosen data, for which in fact 
$\kappa_n(\Omega; A; g_\ep) = \kappa_n(\Omega)$.

\subsection{Outline of the paper}

In the following lines, we explain the main ideas in the proofs. We begin by presenting the key steps in establishing Theorem \ref{main}. In what follows, all the statements about the asymptotic behavior of objects depending on the family $(u_\e)$ are understood to hold up to passing to a subsequence.

\medskip

{\bf Proof of Theorem \ref{main}}

\medskip

The general strategy consists in splitting the contributions of the energy into two parts, one coming from $e^{2d}_\e(u) ,$ the other from $|\partial_z u|^2 ,$ and to obtain sharp lower bounds for each piece.

Roughly speaking, if one knew the vorticity of a sequence of maps $(u_\e)$ obeyed \eqref{comp} and \eqref{deltavecf} for some $f\in H^1((0,L), (\R^2)^n)$, then one would expect that on a typical height $z_0 ,$ the graph of $f$ should cross $\R^2\times \{z_0\}$ transversally and, accordingly, for small enough $\e,$ the restricted maps $(v_\e(\cdot\, , z_0))$ should have $n$ well defined vortices close to  $f_1 (z_0), \ldots, f_n (z_0).$ Then, arguments in Ginzburg-Landau theory should yield the lower bound for $u_\e$

\[\int_{\om} e^{2d}_\e (u_\e(x, z_0)) dx \geq \pi n \vert \log \e\vert - \pi n(n-1) \log \|J_x (u_\e(\cdot, z_0))-n\pi \delta_0\|_{F(\om)}-C  \]
on such a height. From this one should be able to deduce after some work

\[ \int_\Om e^{2d}_\e (u_\e) \geq \left(\pi n \vert \log \e\vert + \pi n(n-1)\vert\log h_\e\vert\right)\vert S\vert -\mbox{logarithmic interaction terms}-C, \]
where $S$ is the set of these typical heights, which one expects to be close to having full measure in $(0, L) .$ On the other hand, for maps 
independent of the $z$-variable, this lower bound holds with $\vert S\vert=L$ and corresponds to the total energy, up to a constant. This suggests that the remainder should account for variations in the energy from that of perfectly parallel filaments, that is

\begin{equation}\label{dzu.lbd}
\frac 12 \int_\Om \vert \partial_z u_\e \vert^2 \geq  \frac \pi 2\int_0^L \sum_i\vert f_i' \vert^2 dz - o(1) \qquad\mbox{ as } \ep\to 0.
\end{equation}

In light of this, the first step to a rigorous argument is to establish some preliminary lower bounds under assumptions \eqref{scaling1}, \eqref{scaling2} and a corresponding initial compactness for the vorticity, somewhat weaker than \eqref{comp} and \eqref{deltavecf}. This will later allow for improving both the lower bounds and compactness property, upon subsequent analysis using \eqref{scaling1a} and \eqref{scaling1aa}. 

{\bf First estimates.} In Section \ref{section2} we define a notion of `good height' (see \eqref{G.def} below). Heuristically speaking, a height $z_\e$ is considered good for our purposes if $J_x u_\e (\cdot, z_\e)$ is close to $n\pi \delta_0$ on many balls centered at the origin. This choice makes possible to deduce that for any such height either
\[\int_\om e^{2d}_\e (u_\e(x, z_\e))dx \geq \pi(n+\theta) \vert \log \e\vert
\qquad\mbox{ for some $\theta>0$, see Lemma \ref{L3b} below,}
\]
in which case we have a considerable energy excess of order $\mathcal{O}(\vert\log \e\vert),$ or else a very detailed description of the vortex structure is available which allows for very precise lower bounds in terms of the renormalized energy \eqref{expforren} (see Lemma \ref{L3b}). Right from the outset, elementary inequalities using this fact tell us that \eqref{scaling1} implies that the set $\mathcal{G}^\e$ of good heights has measure $\vert \mathcal{G}^\e\vert= L-o(1),$ which is not enough for obtaining a lower bound that accounts for all the divergent part of the energy, but at least let us conclude in Lemma \ref{c.subset} below that for any interval $(a, b) \subseteq (0,L),$

\begin{equation}\label{intro.exp1}\int_a^b \int_\om e^{2d}_\e (u_\e)dx dz =n\pi (b-a)\vert \log\e\vert+o(\vert\log \e\vert).\end{equation}

This fact will be used to prove the key Proposition \ref{P1} and a first compactness property of the vorticity of suitably rescaled maps.

{
{\bf A key estimate.}
We define $\calB^\ep := (0,L)\setminus \calG^\ep$.
 The next task is to exploit the definition of good height and an auxiliary result, Lemma \ref{L8}, to show that if $z_\e \in \calB^\ep,$ somewhat surprisingly, we have the much stronger lower bound on the $2d$ energy. }

\begin{equation}\label{intro.exp2}\int_{\om}e^{2d}_\e (u_\e)(x, z_\e) dx \geq \e^{-\alpha},\end{equation}
for some absolute positive constant $\alpha>0.$  { Thus $\calB^\ep$ can be understood
either as the set of ``bad heights", on which we do not have detailed information
about vortex structure, or as the set of ``better heights", which enjoy a very 
strong lower energy bound.
Note that \eqref{intro.exp2} immediately implies, thanks to \eqref{scaling2}, that 
$\calB^\ep$ is smaller in measure than a power of $\e .$}

The idea here is that with the aid of Lemma \ref{L8}, we can show that there is a constant $c$ such that for any {``bad" }height $b,$ there are many cylinders $\om\times (a,b) ,$ such that at least one of the following holds
\[\int_{\om\times\{b,a\}}e^{2d}_\e (u_\e) \ dx\geq 2\e^{-\alpha}, \mbox{ or }\int_{\om\times (a,b)}e^{2d}_\e (u_\e) \ dx dz\geq c\vert\log\e \vert .\]
However, choosing $a$ close enough to $b$, which we show can be done, we can rule out the latter thanks to \eqref{intro.exp1}. We also show that we can choose $a$ so that $\int_{\om\times\{a\}} e^{2d}_\e (u_\e) \ dx<\e^{-\alpha}$ and thus we deduce the desired bound \eqref{intro.exp2} from the former possibility.

{\bf Characterizing the vorticity.} Section \ref{section3} deals with the compactness of $(u_\e),$ in particular we show the vorticity concentrates along $n$ vortex filaments which we identify as $H^1$-curves over $(0,L)$, and we prove the lower bound \eqref{dzu.lbd}. 
We remark that results similar to \eqref{dzu.lbd} are known in somewhat different
contexts, in which one has for example strong bounds on
$\int_\omega e^{2d}_\ep(u_\ep(x,  z))dx$ that are {\em uniform} with respect to the $z$
variable, as well as limiting vortex curves that are known not to intersect; see
\cite{waveJerrard, LinWave, SSprod}. Our proof here 
borrows some ideas from these earlier works.
We choose to follow \cite{waveJerrard}, but one could also give a different
proof that relies more on ideas introduced in \cite{SSprod}.

We start by noticing that thanks to \eqref{intro.exp1} we have

\[\int_\Om \vert \partial_z u_\e \vert^2 =o(\vert\log\e\vert),\]
and therefore we may find a natural scale $\ell_\e$ such that $h_\e\leq\ell_\e=o(1),$ which we show at the end of the proof to be equal to $h_\e$,  such that the rescaled maps $v_\e (x,z):=u_\e (\ell_\e x, z)$ satisfy that their normalized energies

\[\frac{1}{\vert \log \e'\vert}\int_{\Om_\e}  e_{\e '}(v_\e)\]
are uniformly bounded, where $\e':=\e/\ell_\e$ and $\Om_\e:=\om_\e\times(0,L) ; \, \om_\e:=\ell_\e^{-1} \om .$ This puts us in a position to apply an abstract compactness result from \cite{JSo}(see also \cite{abo}), which we will do at several stages in the proof of part $(c)$ of Theorem \ref{main}.  A key point is to show that there is a dense subset $H$ of $(0,L),$ such that for any heights $z_1< z_2 \in H,$ we have that

\begin{equation}\label{zzz.intro}\frac{\pi}{2}\min_{\sigma\in S^n}\sum_{i=1}^n\frac{\abs{p_i(z_1)-p_{\sigma(i)}(z_2)}^2}{z_2-z_1}
\ \le \ \liminf_{\e\to 0} \int_{\om_\e\times (z_1, z_2)} \frac{\vert\partial_z v_\e(x, z)\vert^2}{2 \vert\log \e\vert} dx dz,\end{equation}
where for $j=1,2,$ it holds that

\begin{eqnarray*}
&&J_x v_\ep(\cdot, z_j)  \to \pi \sum_{i=1}^{n} \delta_{p_i(z_j)} 
\qquad
\mbox{ in }W^{-1,1}(B(R)), \ \ \mbox{ for all }R>0.
\end{eqnarray*}
In fact something stronger is proved in Lemmas \ref{labels}--\ref{L.choosef}. This type of estimates let us relate the information of the vorticity of nearby heights, whose variations can be controlled by $\vert\partial_z v_\e\vert^2/\vert\log\e\vert$ at any scale. In particular, we make use of \eqref{zzz.intro} to identify the limiting vortex filaments and to gain the anticipated control on the modulus of continuity of the corresponding $f.$

{\bf Refined lower bounds and compactness.} We complete the proof of the compactness and lower bounds, that is part $(a)$ of Theorem \ref{main}, in Subsection \ref{lyhcoinciden} by combining the previous steps and appealing to \eqref{scaling1a}, \eqref{scaling1aa}. We obtain a very precise lower bound for $e^{2d}_\e(u_\e)$ in terms of the intermediate scale 
$\ell_\e$ in Lemma \ref{lb.by.Fatou}; it follows from a fact we establish a bit earlier: given $z\in H,$  it holds that 

\begin{eqnarray}\label{2dfatou.intro}\int_\omega e^{2d}_\e(u_\ep(x,z)) \, dx -  \big[
  n(\pi \logeps+\gamma)   + n(n-1)\pi  |\log\ell_\ep |
+ n^2 H_\omega(0,0)\big] \nonumber\\
 \geq  -\pi \sum_{i\ne j} \log|q_i(z) - q_j(z)|, \end{eqnarray}
where $J_x v_\ep(\cdot, z)  \to \pi \sum_{i=1}^{n} \delta_{q_i(z)}\mbox{ in }W^{-1,1}(B(R))$ for all $R>0.$ 
For more details see Lemmas \ref{lemma11} and \ref{L.X}.

An application of Fatou's lemma to \eqref{2dfatou.intro} and all the previous analysis yield the lower bound

\[ G_\e (u_\e) \ge \pi n(n-1) \log (h_\e /\ell_\e ) +G_0 (f)+o(1).\]

Finally, we proceed to conclude the proof of part $(a)$ by using \eqref{scaling1a}, \eqref{scaling1aa} to show that the intermediate scale $\ell_\e$ actually coincides with $h_\e$ in this case.

{\bf Stronger compactness and the construction of a recovery sequence.} We turn to the proof of part $(c)$ of Theorem \ref{main} in Section \ref{icfts}. So far we know that $J_x v_\e\to \pi \delta_{[f(\cdot)]}\mbox{ in } W^{-1,1} (B(R)\times (0,L)) \mbox{ for every } R>0;$ here we show that if there is no loss in \eqref{uniflowse} then in fact $f$ captures all of the limiting vorticity, that is

\[\frac{1}{\pi}\star J v_\e \mbox{ converges to }\sum_{i=1}^n \Gamma_{f_i} \mbox{ in the flat norm }F (B(R)\times (0,L)), \mbox{ for every } R>0. \] 
The proof hinges upon the measure-theoretic Lemma \ref{L.levelset}, which tells us that the limiting vorticity $J$ (whose existence is guaranteed thanks to Theorem \ref{UKC}) is given by $\sum_{i=1}^n \Gamma_{f_i}$ plus some indecomposable pieces supported on a union of horizontal planes. However, the tightness in the energy precludes the presence of these horizontal components.

To conclude the proof of Theorem \ref{main}, it remains to show the existence of a recovery sequence for the effective energy $G_0(f)$ for any admissible $f.$ 
We construct such families $(u_\e)$ based on canonical harmonic maps and estimates from \cite{JSp1, JSp2}. This is done in Subsection \ref{section5}.

\medskip

{\bf Theorems \ref{minimizers} and \ref{localmin}}

\medskip

These results rest on the $\Gamma$-convergence result. The extra technical difficulty we need to address is the refinement of the construction of a recovery sequence that is now additionally required to satisfy prescribed boundary data; this is a delicate matter because for us boundary conditions with multiple degree vortices are perfectly acceptable. This is taken care of in Section \ref{section6} where we present the proof of Theorem \ref{minimizers}. To prove Theorem \ref{localmin} we need to relate local minimizers with respect to two different topologies; after this is done the existence of local minimizers with the desired properties follows from standard arguments (see \cite{MSZ}). This is achieved in Section \ref{sectionforlocalmin}.

The article concludes with an appendix where the proofs of some technical results from Section \ref{section2} are included.
\vskip.3in
{\it \bf Acknowledgments}
The research of both authors was partially supported by the National Science and
Engineering Research Council of Canada under operating grant 261955. The first author wishes to thank the
Fields Institute hosting him during the Fall of 2014, and the stimulating environment provided by the thematic program ``Variational Problems in Physics, Economics and Geometry'' where part of this work was carried out. 

\section{Preliminaries}

A key object in our study, and more generally in the asymptotic analysis of the
Ginzburg-Landau functional, is the {\em renormalized energy},
introduced by Bethuel, Brezis, and H\'elein \cite{BBH}, which in our context may be defined
as follows.

For $y\in \om$, we write 
$H_{\om}(\cdot,y)$ to denote the solution to 
\begin{equation}\label{H.def}
\Delta_x H_\omega(\cdot,y)=0\mbox{ in }\om,\; \qquad
H_\omega(x,y)=-\log\abs{x-y}\mbox{ for }x\in \p\om.
\end{equation}
For distinct points $p_1,\ldots, p_n\,\in\om$
we define 
\beq\label{expforren}
W_\om(p_1,\ldots, p_n)=-\pi \left(\sum_{i\neq j}\log\abs{p_i-p_j}+\sum_{i,j}H_{\om}(p_i, p_j)\right).
\eeq
The constant $\kappa_n(\Omega)$ appearing in \eqref{Gep.def} 
is given by 
\begin{equation}\label{kappan.def}
\kappa_n(\Om):=-\pi n^2 LH_\omega(0,0)+nL\gamma,
\end{equation}
Here
$\gamma$ is the constant defined in \cite{BBH} as
\beq
\gamma:=\lim_{\e\to 0}(I(1,\e)+\pi\log\e),\label{BBHconstant1}
\eeq
where
\beq
I(R,\e):=\min 
\left\{\frac{1}{2}\int_{B(0,R)}e_\ep^{2d}(u)dx \ ; \ 
u\in H^1(B(R), \C),u=\frac x{|x|}\mbox{ on }\p B(R) \right\}
\label{BBHconstant2}
\eeq

\begin{remark}\label{rem.d1}
We remark that natural boundary conditions on $\partial \omega$ are built into the 
definition of $H_\omega$ and hence also $W_\omega$ and $\kappa_n(\Omega)$.
In particular, if we wish to consider Dirichlet data on $\partial \omega\times (0,L)$
of the form contemplated in \eqref{d0}, then the 
correct substitute for $\kappa_n$ is
\begin{equation}
\kappa_n(\Omega, g)  = -\pi n^2 L H(0,0;g) + nL\gamma
\label{knog.def}
\end{equation}
where $H(\cdot, y;g)$ is harmonic in $\omega$ for every  $y$, and
\begin{equation}\label{mercury}
\nu(x)\cdot \nabla_xH(x,y;g) = j_\tau g(x) + \nu \cdot \frac{(x-y)}{|x-y|^2}\quad\mbox{ for }x\in \partial \omega .
\end{equation}
Here $j_\tau g(x) = Im(\bar g, \nabla_\tau g)(x)$.
Similarly, the correct renormalized energy for \eqref{d0} is defined as in \eqref{expforren}, but with $H_{\omega}(p_i, p_j;g)$ in place of $H_\omega(p_i,p_j)$.
\end{remark}

\subsection{More about currents}\label{S.currents}

In Section \ref{somen}, we defined $k$-currents, and we introduced
several particular classes of  currents of interest, including for example the $1$-current $T_\gamma$
associated to a Lipschitz curve $\gamma$.  

In general, if $T$ is a $k$-current, then $\partial T$ is the $(k-1)$-current defined by
$\partial T (\phi) := T(d\phi)$. For example, if $U$ is an open set in $\R^n$ and
$\gamma:[a,b]\to \bar U$ is a Lipschitz curve such that $\gamma(s)\in U$ for $s\in (a,b)$, then
\[
\partial T_\gamma(\phi) = \phi(\gamma(b)) - \phi(\gamma(a))\qquad\mbox{ for }\phi \in \calD^0(U).
\]
In particular, $\partial T_\gamma=0$ in $U$ if $\gamma(a)$ and $\gamma(b)$ belong to $\partial U$, or if $\gamma(a)=\gamma(b)$.

We will frequently encounter  {\em integer multiplicity rectifiable}  $1$-currents
in various $3$-dimensional open sets $U$. Such a current $T$ 
admits a moderately complicated
description in general, but if $M(T)<\infty$ and $\partial T = 0$ in $U$, which will always be the case for us, then it can always be represented in the form
\begin{align}
T&=\sum_{i\in I}T_{\gamma_i},\qquad\qquad \p\Tgi=0\mbox{ in }U \mbox{  for all }i\in I, \nonumber\\
M(T)&=\sum_{i\in I}M(\Tgi)   \ = \sum_{i\in I}\calH^1(\gamma_i) < \infty.
\label{characofonecurrents}
\end{align}
for some family of Lipschitz maps $\{\gamma_i\}_{i\in I}$,
with $I$ at most countable. (Here and below, $\calH^k$ and $\calL^k$
denote $k$-dimensional Hausdorff and Lebesgue measure, respectively.)

The remainder of this section is used only rarely and can be skipped until needed.

It will be useful to consider slices of
various currents, including $\star Ju$ and $\Gamma_0$, by the function $\zeta(x,z) = z$.
In general a slice\footnote{In order for  the slices of $S$ to be well-defined, $S$ must satisfy some mild
regularity hypotheses, which will always hold for us.}
 of a $1$-current $S$
in $\Omega$ by $\zeta^{-1}\{z\}$
is a $0$-current
supported in  $\zeta^{-1}\{z\}$, which is denoted $\langle S, \zeta, z\rangle$.
For the currents we are interested in, we have explicit formulas:
\[
\langle \star Ju, \zeta, z\rangle(g) = \int_{\om} J_x u(x,z) g(x,z) \ dx
 \ \ \mbox{ a.e. }z,
\qquad\quad\mbox{ for }u\in H^1(\Omega;\C)
\]
\[
\langle \Gamma_{f} , \zeta, z\rangle  =
\delta_{(f(z),z)}  \ \ \mbox{ a.e. }z,
\qquad\quad\mbox{ for }f\in H^1((0,L);\R^2)
\]
using notation introduced in \eqref{Gammaf.def}.
In particular,
\[
\langle \Gamma_0, \zeta, z\rangle = \delta_{(0,z)} 
\]
Note that 
$\langle \star Ju, \zeta, z\rangle$ can (for {\em a.e.} $z$) be identified with
the Jacobian of $u(\cdot, z)\in H^1(\om; \C)$, and as a result
\[
\|  \langle \star J u - n\pi \Gamma_0 , \zeta, z\rangle\|_{F(\Omega)} 
= \| J_{x} u(\cdot, z) - n\pi \delta_0 \|_{F(\om)}.
\]
In Remark \ref{rem2} above, we have asserted that
\[
\int_0^L \| J_x u (\cdot, z) - \pi n \delta_0\|_{F(\om)}\, dz \  \le \| \star Ju  -n\pi \Gamma_0\|_{F(\Omega)}.
\]
This estimate, which is not really used in this paper, is a direct consequence of the above considerations and the  general fact
that
\begin{equation}
\int_0^L  \| \langle S, \zeta, z\rangle \|_{F(\omega)}dz \  \le \  \|S\|_{F(\Omega)}
\qquad \mbox 
{for any current $S$ in $\Omega$.}
\label{slice.est}\end{equation}
proved in Federer \cite{Federer} 4.2.1.

\section{Preliminary lower bounds for the $2d$ energy}\label{section2}

In this section we prove lower bounds for the $2d$ energy
under assumption \eqref{scaling1}.

\subsection{A criterion for vorticity}

We first introduce a criterion that will allow us to detect, roughly speaking, when
a function $w\in H^1(\omega;\C)$ has $n$ vortices rather near the origin.
There are many ways of doing this; the one we choose is designed to
facilitate the proof of a key estimate that we establish in Proposition \ref{P1} below.

We henceforth write
\[
r^* := 1 \wedge \dist(0,\partial \omega) = \min\{1,  \dist(0,\partial \omega)\}.
\]
Given $w \in H^1(\omega;\C)$,
we will use the notation
\begin{equation}
\calS_n(w) := \left\{ s\in \left
(\frac {r^*}2,  r^* \right) : \ \left| \int_{B(s)} J_xw(x) \ dx - n\pi \right| \le   1 \right\}.
\label{calS.def}\end{equation}

We will later show that $w$ has various good properties  if $\calS_n(w)$ is large in the sense that 
\begin{equation}\label{Slarge}
|\calS_n(w)|\ge \frac {r^*}4.
\end{equation}
This says that on a majority of balls $B(s)$, $\frac {r^*}2 <s<{r^*}$, the vorticity contained in $B(s)$ is not too far from $n\pi$.

We first establish some estimates that
will later allow us to show that
if $(u_\e)\subset H^1(\Omega, \C)$ satisfies assumption \eqref{scaling1} 
and $\e$ is small, then
$w(\cdot)  = u_\e(\cdot , z)$ satisfies \eqref{Slarge} for most 
values of $z$.

\begin{lemma} If $w\in H^1(\omega;\C)$,
then
\begin{equation}
|\calS_n(w)| \ge  \frac {r^*}2 - \| J_x w  - n\pi\delta_0\|_{F(\om)} \, .
\label{GvsF1}\end{equation}

Also, if  $|\calS_n(w)| \ge \frac {r^*}4$, then there exists $\phi\in W^{1,\infty}_0(\om)$
such that $0\le \phi \le 1$, $\|\nabla\phi \|_\infty \le 4/{r^*}$, supp$(\phi)\subset B(r^*)$, $\phi = 1$ in $B(r^*/2)$,  and
\begin{equation}
\left| \int_\omega \phi(x) \, J_xw(x) \, d x - n\pi\right| \le 1.
\label{GvsF2}\end{equation}
\label{GvsFlat}\end{lemma}

\begin{proof}
Consider a  measurable subset $A\subset \left(\frac {r^*}2,{r^*} \right)$
and let  $g:(0,\infty)\to \R$
be the (unique) compactly supported  Lipschitz continuous function 
such that
\[
g'(s) = -{\bf 1}_{s\in A} \operatorname{sign}\left(
 \int_{B(s)} J_x w \, dx - n\pi 
\right) \quad \mbox{a.e. } s.
\]
Since $g'(s)=0$ for $s\ge {r^*}$, it is clear that $g(s)=0$ for $s\ge {r^*}$,
and hence
\[
g(s) = -\int_s^{r^*} g'(t)\ dt \qquad\mbox{ for }0\le s \le {r^*}.
\]
Thus
\begin{align*}
\int_{s\in A} \left|\int_{B(s)} (J_x w  - n\pi \delta_0)\right| ds 
&=
- \int_0^{r^*} g'(s) \left(  \int_{B(s)} (J_xw - n\pi \delta_0)\right) ds\\
&= - \int_{B({r^*})} \int_{|x|}^{r^*} g'(s)  (J_xw - n\pi \delta_0) \ ds
\\
&=
\int_{\om} g(|x|)(J_xw - n\pi \delta_0).
\end{align*}
If we take  $A := \left(\frac {r^*}2,{r^*} \right)\setminus \calS_n(w)$, then it follows from
the definition of $\calS_n(w)$ that 
\[
\frac  {r^*}2 - |\calS_n(w)| \ = \ | A|  \ \le \ 
\int_{s\in A} \left|\int_{B(s)} (J_x w - n\pi \delta_0)\right| ds .
\]
Also,  $\max( \| g \|_\infty, \| g'\|_\infty) \le 1$
so
\[
\int_{\om} g(|x|)(J_x w - n\pi \delta_0)
\le
\| J_xw - n\pi\delta_0\|_{F(\om)}.
\]
The first conclusion of the lemma follows by combining these inequalities.

To prove the final conclusion of the lemma, assume that
$|\calS_n(w)|\ge \frac {r^*}4$, and let $g(s)$ be the (unique) compactly
supported Lipschitz continuous  function such that
\[
g'(s) = - |\calS_n(w)|^{-1} {\bf 1}_{\calS_n(w)} \qquad \mbox{\em a.e. }s.
\]
Then $0\le g \le 1$, and $\|g'\|_\infty = |\calS_n(w)|^{-1} \le 4/{r^*}$.
Moreover, arguing as above one computes that
\begin{multline*}
\int_{\om} g(|x|) (J_xw - n\pi\delta_0) 
= -\int_0^{r^*} g'(s) 
\left(\int_{B(s)} J_xw-n\pi\delta_0\right)ds\\
= \frac 1{|\calS_n(w)|} \int_{\calS_n(w)} 
\left(\int_{B(s)} J_xw-n\pi\delta_0\right)ds.
\end{multline*}
Then the final conclusion \eqref{GvsF2}, with $\phi(x) = g(|x|)$, now follows directly
from the definition of $\calS_n(w)$.
\end{proof}

We now identify  the above-mentioned good properties
enjoyed by a function $w$ satisfying \eqref{Slarge}.
We show that such a function either
has a very well-defined vortex structure and
associated sharp lower energy bounds,
or else it has excess energy, in the sense that 
condition \eqref{L3b.h2} below fails.
More precisely, we have

\begin{lemma}
There exist positive numbers $\theta, a,b$, 
depending on $n$ and $r^*$, such that $b<a$, and the following holds:

Assume that $w\in H^1(\omega;\C)$ and that $|\calS_n(w)| \ge \frac {r^*}4$ and
\begin{equation}\label{L3b.h2}
\int_\omega e^{2d}_\e (w)(x) \,dx \le  \pi(n+\theta)\logeps \ .
\end{equation}
If  $0<\e<\e_0� = \e_0 (\theta,a,b,n),$ then   there exist $p^\ep_1,\ldots, p^\ep_n\in \om,$ such that
\begin{equation}
\| J_x w - \pi \sum \delta_{p^\ep_i}\|_{F(\om)} \le \ep^a \ ,
\label{L3b.c1}\end{equation}
\begin{equation}
\dist(p^\ep_i, \partial \omega)\ge \frac {r^*} {8}\mbox{ for all }i, \qquad\qquad
|p^\ep_i - p^\ep_j| \ge \ep^b \ \mbox{ for }i\ne j, \qquad \mbox{ and }
\label{L3b.c2}\end{equation}
\begin{equation}
\int_{\omega} e^{2d}_\ep(w) dx \ \ge \  n(\pi\vert \log\e \vert +\gamma) +  W_{\omega}(p^\ep_1,\ldots, p^\ep_n)
- C(n,\theta)\e^{(a-b)/2} \ 
\label{L3a.sharper}\end{equation}
where $W_\omega$ is the renormalized energy defined in \eqref{expforren}
and $\gamma$ is defined in \eqref{BBHconstant1}.
\label{L3b}\end{lemma}

The proof of Lemma \ref{L3b} is postponed to Appendix \ref{App.A}.
For now we only remark that
the verification of \eqref{L3b.c1}, \eqref{L3b.c2} involves
a vortex ball argument, and that we will deduce 
\eqref{L3a.sharper} by appealing to results in 
\cite{JSp2}, for which \eqref{L3b.c1} 
and \eqref{L3b.c2} supply the hypotheses.
Incidentally, the reason we have assumed that
$\omega$ is simply connected is that this
condition is imposed in \cite{JSp2}.

Our next result is a corollary of Lemma \ref{L3b}.

\begin{lemma}\label{L3a}
Assume that $u_\ep(\cdot, z)\in H^1(\omega;\C)$, that $|\calS_n(u_\ep(\cdot, z))| \ge \frac {r^*}4$,
and that \eqref{L3b.h2} holds.

Then for $0<\ep<\ep_0$, 
\begin{equation}
\label{L3a.c1}
\int_{\om} e^{2d}_\ep(u_\e(x,z)) dx 
\ge n\pi \abs{\log\e} - C  -\pi n(n-1) \log \| J_x u_\e(\cdot, z) - n\pi\delta_0\|_{F(\om)} 
\end{equation}
and
\begin{equation}
\int_{\om} e^{2d}_\ep(u_\e(x,z)) dx 
\ge n\pi \abs{\log\e} - C  \ .
\label{L3a.c2}\end{equation}
\end{lemma}

\begin{proof}
We wish to deduce the conclusions of the lemma from \eqref{L3a.sharper}. To do this,
we must estimate
\[
 W_{\omega}(p^\ep_1,\ldots, p^\ep_n)
=-\pi \left(\sum_{i\neq j}\log\abs{p^\ep_i-p^\ep_j}+\sum_{i,j}H_{\om}(p^\ep_i, p^\ep_j)\right)
\]
where $H_\om$ is defined in \eqref{H.def}. 

It is clear that $H_\om$ is smooth in the interior
of $\om\times \om$, so \eqref{L3b.c2} implies that $|H_\om(p^\ep_i, p^\ep_j)|\le C$
for all $i,j$.

To estimate the other terms, we write 
$s_\e :=  \| J_x u_\e(\cdot, z) - n\pi\delta_0\|_{F(\om)}$
for convenience.
Then it follows from \eqref{L3b.c1} and the triangle inequality that
\[
\| \pi \sum_{i=1}^n ( \delta_{p^\ep_i} - \delta_0) \|_{W^{-1,1}} \le s_\ep + \e^a.
\]
On the other hand, setting $\phi(x) = (r^*-|x|)^+$, 
\[
\| \pi \sum_{i=1}^n ( \delta_{p^\ep_i} - \delta_0) \|_{W^{-1,1}} \ge \int_{\omega} \phi(x)
\pi \sum_{i=1}^n ( \delta_0 - \delta_{p^\ep_i} ) =  \pi \sum |p^\ep_i|.
\]
Thus $|p^\ep_i - p^\ep_j| \le s_\e+ \ep^a \le 2s_\e $ for all $i\ne j$.
(The last inequality follows from \eqref{L3b.c2}, which with the above shows that 
$s_\e \ge  \e^b -  \e^a \ge \frac 12 \e^b$.)
These imply that 
\[
 W_{\omega}(p^\ep_1,\ldots, p^\ep_n)
\ge
-\pi n (n-1) \log s_\e - C.
\]
Also, since $|p^\e_i| \le \mbox{diam}(\omega)$ for all $i$, it is clear that 
$W_\om(p^\ep_1,\ldots, p^\ep_n) \ge -C$.
\end{proof}

\subsection{Good heights}

We now fix $u_\e\in H^1(\Omega;\C)$ satisfying \eqref{scaling1}. 

We will say that a height $z$ is {\em good} if $u_\e (\cdot, z)$ satisfies
\eqref{Slarge}, and we define $\calG^\ep_1$ to be the set of good heights:
\begin{equation}
\calG^\ep_1 := \left\{ z\in (0,L) : \left| \calS_n(u_\e(\cdot, z)) \right| \ge \frac {r^*}4  \right\}\ .
\label{G.def}\end{equation}
We also define
\[
\calB^\ep_1 := (0,L)\setminus \calG^\ep_1.
\]
The results of the previous sections immediately imply that the good set is big:

\begin{lemma}
\label{L.B1small}
For $u_\ep \in H^1(\Om;\C)$ satisfying \eqref{scaling1}, we have the estimates
\[
L - \left| \calG^\ep_1 \right| =  \left| \calB^\ep_1 \right|\  \le   \frac  4{r^*}
 \int_0^L\| J_xu_\e(\cdot, z) -n\pi\delta_0\|_{F(\omega)}dz \ .
\]
\end{lemma}

\begin{proof}
By \eqref{GvsF1}, we know that if $z\in \calB^\ep_1$, then
$ \| J_xu_\e(\cdot, z) -n\pi\delta_0\|_{F(\omega)} > \frac {r^*}4$.
In addition,  Chebyshev's inequality implies that for any $s>0$,
\begin{equation}
\calL^1\left(\left\{ 
z\in (0,L) :  \| J_x u (\cdot, z) - \pi n \delta_0\|_{F(\om)}  \ge s 
\right\} \right)  \ \le   \frac 1s
\int_0^L \| J_x u (\cdot, z) - \pi n \delta_0\|_{F(\om)}\, dz .
\label{Fed2}\end{equation}
The lemma follows by combining these facts.
\end{proof}

As shown in Lemma \ref{L3a}, if $z\in \calG^\ep_1$, then $u_\e(\cdot,z)$ satisfies good lower energy 
bounds. 
Perhaps surprisingly, under hypotheses that are 
satisfied by minimizers and local minimizers
whose vorticity clusters around the vertical segment, $u_\e(\cdot,z)$ satisfies {\em much stronger}
lower energy bounds for $z\in \calB^\ep_1$. Indeed, in Subsection \ref{proofprop1} we will
prove

\begin{proposition}
Assume that $u_\ep$ satisfies \eqref{scaling1} 
and  \eqref{scaling2}. Then for every $\alpha\in \left(0,\frac 23\right)$, there exists $\ep_0$ such that 
if $0<\ep \le \ep_0$, then
\begin{equation} \label{P1prop}
\int_{\omega} e_\ep(u_\ep(x,z)) \ dx \ \ge \ \  \ep^{-\alpha}
\qquad \mbox{ for every $z\in \calB^\ep_1 .$}
\end{equation}
As a result, using the upper bound \eqref{ref.asked} on $F_\ep$, if $\ep < \ep_0$ then
$\left| \calB^\ep_1 \right|\ \le C \e^{\alpha} \vert \log \e\vert $.
 \label{P1}\end{proposition}

This proposition  will play an important role in
the proof of Theorem \ref{main}, although there
in fact a much weaker estimate than \eqref{P1prop} would suffice.

We continue with some easier estimates that will be used several times,
including in the proof of 
Proposition \ref{P1}.

\begin{lemma}Assume that $(u_\ep)$ satisfies \eqref{scaling1}, \eqref{scaling2}.
If $A$ is any measurable subset of  $(0,L)$, then 
\beq\label{Uon2D}\lim_{\e\to 0} \frac 1\logeps \int_A\int_{\om} e_\ep(u_\ep) dx\ dz=
\lim_{\e\to 0} \frac 1\logeps \int_A\int_{\om} e^{2d}_\ep(u_\ep) dx\ dz  =  n\pi |A| 
\eeq
and as a consequence
\begin{equation}
\int_\Omega |\partial_z u_\e|^2 \, dx \, dz \ = \ o(\logeps).
\label{uzolog}\end{equation}
\label{c.subset}\end{lemma}

\begin{proof}
If $z\in \calG^\ep_1$, then
\[
\int_\omega e_\e(u_\e(x,z)) dx \ge 
\int_\omega e_\e^{2d}(u_\e(x,z)) dx \ge  n\pi \logeps - C.
\]
This follows from Lemma \ref{L3a} if \eqref{L3b.h2} holds, and if not  it is immediate.
Thus
\begin{align*}
\frac 1 \logeps\int_A\int_{\om} e_\ep(u_\ep) dx\ dz 
&
\ge  \frac 1\logeps \int_{\calG^\ep_1\cap A} \int_\omega e^{2d}_\e(u_\e) \, dx \, dz
\\
&\ge 
|A\cap \calG^\ep_1| (n\pi - o(1)) \quad\mbox{ as }\ep\to 0.
\end{align*}
Also, it is clear from \eqref{scaling1} and Lemma \ref{L.B1small} that
$|A \cap \calG^\ep_1| \ge  |A| - |\calB^\ep_1| \to |A|$ as $\ep\to 0$.
Thus
\begin{equation}
\liminf_{\e\to 0} \frac 1 \logeps\int_A\int_{\om} e_\ep(u_\ep) dx\ dz \ge\liminf_{\e\to 0} \frac 1 \logeps\int_A\int_{\om} e^{2d}_\ep(u_\ep) dx\ dz  \ge n\pi |A|.
\label{unif1}\end{equation}
Then using \eqref{scaling2} and applying \eqref{unif1} to $\tilde A = (0,L)\setminus A$, 
\begin{eqnarray} 
\limsup_{\e\to 0} \frac 1 \logeps\int_A\int_{\om} e^{2d}_\ep(u_\ep) dx\ dz  
&\le &
\limsup_{\e\to 0} \frac 1 \logeps\int_A\int_{\om} e_\ep(u_\ep) dx\ dz \
\nonumber\\
& \le &  n\pi L -  n\pi |\tilde A| = n\pi |A|.\nonumber
\end{eqnarray}
This is \eqref{Uon2D}. Since $\frac 12|\partial_z u|^2 = e_\ep(u) - e^{2d}_\ep(u)$,
we obtain \eqref{uzolog} as a direct consequence of \eqref{Uon2D} and \eqref{scaling2}.
\end{proof}

We next define another ``good set", consisting of the set of points $z$ such that $u_\ep(\cdot, z)$ satisfies the hypotheses of Lemmas \ref{L3b} and \ref{L3a}.
This will appear often in the proof of the $\Gamma$-limit lower bound and compactness assertions.

\begin{lemma}\label{L.BadE.est}Assume that $(u_\e)$ satisfies \eqref{scaling1} and
\eqref{scaling2}, and define 
\begin{equation}\label{GBep2.def}
\calG^\ep_2 := \left\{ z\in \calG^\ep_1 \ :  \int_\omega e^{2d}_\ep(u_\ep(x,z))dx  \le \pi (n+\theta)\logeps
\right\}, \qquad 
\calB^\ep_2 := (0,L)\setminus \calG^\ep_2,
\end{equation}
where $\theta$ is the constant from Lemma \ref{L3b}. 
Then
\[
|\calB^\ep_2| = o(1) 
\qquad\qquad
|\calG^\ep_2| =  L- o(1) \qquad\mbox{ as }\ep\to 0 .
\]
\end{lemma}

\begin{proof}
We will write 
\[
 \widetilde \calB^\ep_2 :=
 \left\{ z\in \calG^\ep_1 \ :  \int_\omega e^{2d}_\ep(u_\ep(x,z))dx  > \pi (n+\theta)\logeps
\right\} .
\]
Note that $\calB^\ep_2 = \calB^\ep_1\cup \widetilde \calB^\ep_2$.
 Lemma \ref{L.B1small} and \eqref{scaling1} imply
that $L - |\calG_1^\ep| = |\calB^\ep_1| \to 0$, so we only need to prove that
$|\widetilde \calB^\ep_2| \to 0$ as $\ep\to 0$.
Toward this end we use Lemma \ref{L3a} and the definition of
$\widetilde \calB^\ep_2$ to compute
\begin{align*}
\logeps(Ln\pi + o(1))
&\overset{\eqref{scaling2}}\ge
\int_\Omega e^{2d}_\ep(u_\ep)\,dx\,dz \\
&\ge 
\int_{\calG^\ep_1\setminus \widetilde \calB^\ep_2}
\int_\omega e^{2d}_\ep(u_\ep)\,dx\,dz 
+
\int_{\widetilde \calB^\ep_2}
\int_\omega e^{2d}_\ep(u_\ep)\,dx\,dz 
\\
&\ge
(|\calG^\ep_1| -  |\widetilde\calB^\ep_2|)( n\pi\logeps - C)
+
 |\widetilde\calB^\ep_2|( \pi(n+\theta)\logeps \\
&=
|\calG^\ep_1| n \pi \logeps + \theta \pi \logeps |\widetilde \calB^\ep_2| - C \ .
\end{align*}
As already noted, $|\calG^\ep_1|\to L$ as $\ep \to 0$, 
so the conclusions follow.
\end{proof}

\subsection{A Jacobian estimate with weak boundary control}

The following result plays an important role in the proof of Proposition \ref{P1}.

\begin{lemma}
Let $D := (\R/\ell\Z)\times I$ for some interval $I$ and some $\ell>0$, and assume that
$w\in H^1(D;\C)$.
Then for every $\alpha\in (0,\frac 23)$, there exists $c_2 = c_2( \alpha)$
such that for all sufficiently small $\e>0$ (where ``sufficiently small"
depends on $\alpha$ and $D$),  at least one of the
following holds:
\begin{equation}
\int_{\partial D} e_\e(w) d\calH^1 \ \ge \e^{-\alpha}
\label{L8.d1}\end{equation}
or
\begin{equation}
\int_{D} e_\e(w) \ \ge  c_2 \abs{\log\e}  \left| \int_D Jw  \right|  - c_2 \e^{1-\alpha}\abs{\log\e}.
\label{L8.d2}\end{equation}
In fact we will show that one may take $c_2 = \frac \pi 8 (1-\frac {3\alpha}2)$.
\label{L8}\end{lemma}

Conclusion \eqref{L8.d2} is an example of a Jacobian estimate, relating the
Ginzburg-Landau energy and
the Jacobian of a complex-valued function $w$ on a domain $D$.
Well-known examples show that an like \eqref{L8.d2} cannot hold
without some information about the behaviour of $w$ on $\partial D$.
The lemma shows that the very weak bounds
$\int_{\partial D} e_\e(w) d\calH^1 < \e^{-\alpha}$ --- that is, the failure of
\eqref{L8.d1} --- is sufficient information for this purpose.

\begin{remark}
A very similar argument shows that the same result is
true if $D$ is a bounded, open subset of $\R^2$, with smooth boundary.
In this case the geometry of
sets such as  $\partial B \cap D$ and $B\cap\partial D$ 
(if $B$ is a ball) becomes
more complicated. However, on small scales, which are all that matter for this
argument, 
these complications basically vanish.
\label{rem:L8}\end{remark}

In the rest of the subsection we will make use of the symbol $\partial$ to represent either the topological boundary of a set, or the boundary in the sense of  Stoke's theorem, the meaning being clear from the context.

\begin{proof}
{\bf 1}.  ({\it Preliminaries}).
By a density argument, we may assume that $w$ is smooth in $\bar D$. 
We assume that 
\begin{equation}
e_{\partial D}  := \int_{\partial D}e_\ep(w) \, d\calH^1   
< \ep^{-\alpha} ,
\label{ebdy.small}\end{equation}
and  we will show that then \eqref{L8.d2} holds for all sufficiently small $\ep$.

We first recall  that 
if $\Sigma$ is a Lipschitz arc (that is, the image of an injective Lipschitz curve)
in $\bar D$ and
$\calH^1(\Sigma) \ge \ep$, then
\begin{equation}
\int_\Sigma e_\ep(|w|) d \calH^1 \ \ge \ \frac c \ep \| 1-|w|\  \|_{L^\infty(\Sigma)}^2.
\label{modulus.bound}\end{equation}
See for example \cite[Lemma 2.3]{JSo} for a proof.
%
In particular this applies to $\Sigma := \partial D$, 
so  \eqref{ebdy.small}  implies that 
\begin{equation}
\| 1 - |w| \,\|_{L^\infty(\partial D)} \  \le \    c \ep^{(1-\alpha)/2} \le \frac 13 \qquad
\mbox{ for sufficiently small }\ep.
\label{modw.bounds}\end{equation}
We record a couple of consequences of this fact.
First, basic properties 
of the degree
imply that for $v := w/|w|$ (well-defined on $\partial D$)
\[
2\pi d^*: = 2\pi \deg(w;\partial D) = \int_{\partial D} j(v)\cdot \tau \ = \ \int_{\partial D} \frac 1{|w|^2} j(w)\cdot \tau \ .
\]
Also, an integration by parts shows that
\[
\int_D Jw \ = \ \frac 12 \int_{\partial D} j(w)\cdot \tau.
\]
Thus, since  that $|j(w)| \le |w| \ | \nabla w|$
we deduce from \eqref{modw.bounds} that 
\begin{multline}\label{dstar.J}
\left| \pi d^*-  \int_DJw\right| \ = \ 
\left|
\frac 12\int_{\partial D} \left( \frac 1{|w|^2}-1\right) j(w)\cdot \tau
\right| \\
\le
2\int_{\partial D} \big| 1-|w|^2 \big| \  | \nabla w|
\le 4\ep e_{\partial D} \overset{\eqref{ebdy.small}}
\le C \ep^{1-\alpha}.
\end{multline}
We may thus assume that 
\begin{equation}\label{dstar.nonzero}
d^*\ne 0 
\end{equation}
since otherwise \eqref{L8.d2} follows immediately from \eqref{dstar.J}.
Next, we again use \eqref{modw.bounds} to find that
\begin{equation}\label{J.bound}
\left|
\int_D Jw \right| = 
\left|
\frac 12 \int_{\partial D}j(w)\cdot \tau \right| \le 
\int_{\partial D}| \nabla w|  \ \le \ 
C \sqrt {e_{\partial D}} \le C\ep^{-\alpha/2}\ .
\end{equation}

What follows is a modification, in which we exploit \eqref{ebdy.small} to control
certain boundary terms, of the classical procedure of obtaining lower bounds for Ginzburg-Landau functionals in terms of the Jacobian by means of a ball construction \cite{Jlower, S}.

{\bf 2}.  ({\it Basic estimates}).
For $v = \frac w{|w|}$ as above, recall that 
\begin{equation}
e_\ep(w) = \frac 12  |w|^2 | \nabla v|^2 + e_\ep(|w|),
\qquad
e_\ep(|w|) := \frac 12 |\nabla |w||^2+ \frac 1 {4\ep^2}(|w|^2-1)^2.
\label{L8.0}\end{equation}
In particular this implies that $| \nabla v|\le \frac 1{|w|}| \nabla w|$.

Next, let $B$ be a ball\footnote{We have in mind a ball
with respect to the natural notion of distance in $(\R/\ell\Z)\times \R$.} of radius $r$. 
Fix $r_0>0$, depending on $D$, such that if $r<r_0$, then $\partial D \cap B$
consists of at most one line segment, necessarily of length at most $2r$ 
and $\partial B\cap D$ is an arc of a circle (It suffices to
take $r_0 < \frac 12\min(\ell, |I|)$).
Then
\eqref{modw.bounds} and elementary inequalities imply that 
\[
\int_{\partial D \cap B} |\partial_\tau v| d\calH^1 \ \le
\ 2 \int_{\partial D \cap B} |\partial_\tau w| d\calH^1 \ \le  4 \sqrt r \sqrt{e_{\partial D}},
\]
where $\partial_\tau w$ denotes the tangential derivative.
Thus 
\begin{equation}
\int_{\partial D \cap B} |\partial_\tau v| d\calH^1 \ \le \pi ,
\qquad\qquad\qquad \mbox{ if } r \le
 r_1 :=\min( \frac {\pi^2}{16} e_{\partial D}^{-1},  r_0) \, .
\label{L8.1}\end{equation}
On the other hand, if $d := \deg(w ; \partial(B\cap D))$ is well-defined 
and nonzero, then
\[
d =  \frac 1{2\pi} \int_{\partial (B\cap D)} j(v)\cdot \tau. 
\]
Since $|j(v)\cdot \tau|\le |v|  \ |\partial_\tau v| \ = \  |\partial_\tau v|$, we combine this with \eqref{L8.1} to find that
\[
\int_{\partial B\cap D} |\partial_\tau v| \ \ge \  (2\pi |d| - \pi) \ge  \pi |d|
\qquad\qquad \mbox{ if } r \le r_1.
\]
In particular, if $m := \min_{\partial B\cap D}|w|$, it follows from this,
\eqref{L8.0} and \eqref{modw.bounds} that 
\begin{align*}
\int_{\partial B\cap D} e_\ep(w)
&\ge
\frac{m^2}2\int_{\partial B\cap D}|\partial_\tau v|^2 + \int_{\partial B\cap D} e_\ep(|w|)\\\
&\ge 
\frac {m^2}{2 \calH^1(\partial B\cap D)}\left(\int_{\partial B\cap D} |\partial_\tau v| \right)^2 + \frac c \ep (1-m)^2\\
&\ge
\frac {m^2 \pi d^2}{4 r} + \frac c \ep (1-m)^2 \
\qquad\qquad\qquad\qquad \mbox{ if } r \le  r_1.
\end{align*}

If we define
\begin{equation}
 \lambda_\ep(r, d) := 
\min_{m\in [0,1]}
\frac {m^2 \pi d^2}{4 r} + \frac c \ep (1-m)^2 \
\label{lambda.def}\end{equation}
then it follows that for any ball $B$, 
\begin{multline}
\int_{\partial B\cap D} e_\ep(w) \ge \lambda_\ep(r,d) \\
\mbox
{ if } d := \deg(w;\partial (B\cap D)) \mbox{ and }r = \mbox{radius}(B)\le  r_1 :=\min( \frac {\pi^2}{16} e_{\partial D}^{-1},  r_0).
\label{basic.estimate}\end{multline}

{\bf 3}. ({\it Lower bounds via a ball construction}) With estimate \eqref{basic.estimate} in hand, we can carry out a 
vortex ball construction, as described in  Appendix \ref{App.A},
as long as all balls have radius at most $r_1$.
We sketch the main steps, following the presentation in \cite{Jlower, JSo}.
To get started, we invoke  Proposition 3.3 in \cite{Jlower}, which shows that there exists a finite collection $\{ B^0_i\}$
of closed, pairwise disjoint balls such that
\begin{align}
S_E&\subset \cup B^0_i \qquad\mbox{ for $S_E$ defined in \eqref{SE.def}},
\label{initial.balls1}
\\
r_i^0 &\ge \ep\mbox{ for all }i,
\label{initial.balls2}
\\
\int_{B_i^0\cap S_E} e_\ep(w) &\ge \frac {c_0}\ep r_i^0 \ge \Lambda_\ep (r_i^0) := \int_0^{r_i^0} \lambda_\ep(r,1)\wedge \frac{c_0}\ep \ dr .
\label{initial.balls3}
\end{align}
Here $r^0_i$ denotes the radius of $B^0_i$ and $c_0$ a constant, independent of $w$ and $\ep$.

If $\sum_i r_i^0 \ge r_1$, then it follows from \eqref{initial.balls3}, the choice of $r_1$
(see \eqref{basic.estimate}) and \eqref{ebdy.small}  that
\[
\int_D e_\ep(w) \ \ge \frac {c_0}\ep r_1 \ge \ c \ep^{\alpha-1}.  
\]
Since $\alpha<2/3$, this together with \eqref{J.bound} implies that \eqref{L8.d2} holds for all small $\ep$. We may therefore assume that $\sum r_i^0 < r_1$. 
Next, from the  additivity properties of the degree, \eqref{initial.balls1}, and  the definition  \eqref{SE.def} of $S_E$, we see that\footnote{
Strictly speaking, if $V$ is a set such that $\partial V \cap S_E \ne \emptyset$, then
here and below, $\deg(w; \partial V)$ should be replaced by 
$\dg(w;\partial V_i)$, see \eqref{dg.def} for the definition. This does not change the argument
in any essential way.}
\[
0\overset{\eqref{dstar.nonzero}}\ne d^* = \deg(w;\partial D) = \sum_i \deg (w; \partial (B^0_i\cap D)) =: \sum_i d^0_i.
\]
Thus at least one ball has nonzero degree, and hence
\[
\sigma^0 := \min_i \frac{r^0_i}{|d^0_i|} < r_1.
\]
%

We may now follow a standard vortex ball argument construction as summarized in
Lemma \ref{L.vballs}, but using \eqref{lambda.def}, \eqref{basic.estimate} in place
of the usual estimates \eqref{lambdastandard.def}, \eqref{gllbd.1}. 
For a range of $\sigma>\sigma^0$, this yields
a collection $\calB(\sigma) = \{ B^\sigma_k\}_{k=1}^{k(\sigma)}$ 
of balls such that 
\begin{align}
S_E
&\subset \cup_k B^\sigma_k
\label{vballs1}\\
\int_{B^\sigma_k\cap \om} e^{2d}_\e(w)\, dx  &\ge \frac {r_k^\sigma}\sigma \Lambda_\ep (\sigma),\qquad\hspace{1em}
\mbox{ for }
\, r^\sigma_k := \operatorname{radius}(B^\sigma_k) \, ,
\label{vballs2}\\
r_k^\sigma &\ge \sigma |d^\sigma_k|
\hspace{5em}\mbox{ for }
d^\sigma_k := \deg(w ;\partial (B^\sigma_k\cap D))  \,.
\label{vballs3}
\end{align}
Moreover, $\sigma\mapsto  \sum_k r^\sigma_k$ is a continuous, nondecreasing
function. This process may be continued as long as all balls in
the collection have radius at most $r_1$.

The estimates we obtain in this way are both worse and better than
the classical ones, summarized in Lemma \ref{L.vballs} for example.
They are worse in that we have a somewhat
weaker lower bound (compare  \eqref{lambda.def} and
\eqref{lambdastandard.def})
and we can only continue as long as every ball has radius at most $r_1$;
but better in that all the estimates we obtain apply to all balls, even those that intersect
$\partial D$. This is not true in the classical case, compare for example \eqref{vballs3}.

We stop the ball construction when  the sum of the radii is exactly $r_1$. Then \eqref{vballs2}, \eqref{vballs3}, and the fact\footnote{This is easily checked, and a proof can be for example \cite{Jlower}, Proposition 3.1.} that $s\mapsto \frac 1 s \Lambda_\ep(s)$ is nonincreasing, imply
the lower bound\begin{equation}
\int_D e_\ep(w) \ \ge  \  |d^*| \Lambda_\ep(\frac {r_1} {|d^*|}), \qquad
\quad\mbox{where $d^* = \deg(w;\partial D) = \sum d_k^\sigma$. }
\label{ball.est}\end{equation}
The last equality follows from \eqref{vballs1} and the additivity of the degree.

{\bf 4}. ({\it  Estimating the right-hand side of \eqref{ball.est}})

In view of \eqref{dstar.J}, in order to prove \eqref{L8.d2}, it suffices to show that 
\[
\Lambda_\ep(\frac{r_1}{|d^*|})
\ge 
c \logeps \qquad\mbox{ for all sufficiently small $\ep>0$}.
\]
It is straightforward to check (see \cite{Jlower} or \cite{JSo}) that 
\[
\Lambda_\ep(\sigma) \ge  \frac \pi 4 \log \frac \sigma\ep - C.
\]
Also, it follows from \eqref{dstar.J} and \eqref{J.bound} that 
$|d^*| \le C \sqrt{e_{\partial D}} \le C\ep^{-\alpha/2}$.
As a result, if $r_1 = r_0$, then
\[
\Lambda(\frac {r_1}{|d^*|}) \ge  \frac \pi 2\log( \frac {r_0} {C\ep^{1-\alpha/2}}) - C
=  \frac \pi 2(1-\frac \alpha 2)\logeps - C 
\ge \frac \pi 4 (1-\frac \alpha 2)\logeps
\]
if $\ep$ is small enough. On the other hand, if $r_1 = \frac {\pi^2}{16}e_{\partial D}^{-1}$ then
\begin{multline*}
\Lambda_\ep(\frac {r_1}{|d^*|}) \ge 
\Lambda_\ep(\frac {e_{\partial D}^{-1}}{C\sqrt{e_{\partial D}}}) \ge 
\frac \pi 2 \log (\frac{ e_{\partial D}^{-3/2}} \ep) - C \\
\overset{\eqref{ebdy.small}}\ge
\frac \pi 2  (1-\frac {3\alpha}2) \logeps  - C \ge
\frac \pi 4(1-\frac {3\alpha}2) \logeps 
\end{multline*}
if $\ep$ is small enough.
\end{proof}

\begin{remark}
The cylinder  $\partial B(t)\times I \subset \R^3$ may be parametrized by the map
\[
i:(\R/2\pi t)\times I \to \R^3, \qquad
i(s,z) = ( t \cos (\frac st), t\sin(\frac st), z).
\]
and then
\begin{multline*}
\int_{\partial B(s) \times I} Ju =
 \int_{(\R/2\pi s)\times I } i^* Ju 
=  \int_{(\R/2\pi s)\times I } i^* u^* (d\, \mbox{area})\\
=  \int_{(\R/2\pi s)\times I }(u\circ i)^* (d\, \mbox{area})
=  \int_{(\R/2\pi s)\times I } J(u\circ i).
\end{multline*}
Here $u^*(d \,\mbox{area})$ denotes the pullback by $u:\Omega\to \C\cong \R^2$
of the area form $dx_1 \wedge dx_2$ on $\R^2$. Thus, writing $u = (u_1, u_2)$, 
we have $u^*(dx_1 \wedge dx_2) = du_1\wedge du_2 = Ju$. Similarly $i^* Ju$
denotes the pullback by $i$ of $Ju$.

Also, the area formula implies that 
\[
\int_{\partial B(s)\times I} e_\ep(u_\ep) d\calH^2 \ \ge 
\int_{(\R/2\pi s)\times I } e_{\ep}^{2d}(u_\ep \circ i) d\calH^2.
\]
Thus Lemma \ref{L8} immediately implies the same result for 
cylinders in $\R^3$.
\label{rem.cyl}\end{remark}

\subsection{Proof of Proposition \ref{P1}}\label{proofprop1}

Here we combine the definition of a good height and the implicit global information about the vorticity provided by Lemma \ref{L8}.

\begin{proof}[Proof of Proposition \ref{P1}]

{\bf Step 1}. Fix $\alpha\in (0,\frac 23)$, and fix $z_b\in \calB^\ep_1$.

We define a new ``good set" :
\[
 \calG^\ep_3 := \left\{ z\in (0,L) : \| J_x u_\ep(\cdot, z) - \pi n \delta_0\|_{F(\omega)} \le  \frac {r^*}{16},
\int_\omega e_\ep(u_\ep(x,z)) dx \le  \frac {r^*}{32} \ep^{-\alpha}\right\}.
\]
It follows from  \eqref{Fed2}, and \eqref{scaling2} and Chebyshev's inequality (for the condition involving the energy)   that 
\[
|(0,L)\setminus  \calG^\ep_3| \le C( 
\int_0^L \| J_x u_\ep(\cdot, z) - n\pi\delta_0\|_{F(\omega)}+ \e^{\alpha/2}).
\]
Thus if $\ep$ is small enough, in view of \eqref{scaling1},  we may fix $z_g\in  \calG^\ep_3$ such that
\begin{equation}
\frac 12 \delta   \le |z_b - z_g| \le  \delta
\label{where.zg}\end{equation}
for $\delta = \delta(\alpha)$ to be fixed below.

If we define 
\[
\widetilde \calS^\ep(z)
:= \left\{s\in \left(\frac {r^*}2,{r^*}\right) : \left |\int_{B(s)} J_xu_\ep(x,z) \ dx - n\pi \right| \le \frac 12 \right\},
\]
then the proof of Lemma \ref{GvsFlat} shows that  for every $z$,
\[
\left| \widetilde \calS^\ep(z)\right|
\ge \frac {r^*}2 - 2\|  J_xu(\cdot, z) - n\pi\delta_0\|_{F(\om)},
\]
and thus 
\[
\mbox{ if $z\in  \calG^\ep_3$, then } \left| \widetilde \calS^\ep(z) \right| \ge \frac {3{r^*}}8.
\]
In particular, this holds for $z=z_g$.

On the other hand, the definition of $\calB^\ep_1$ implies that the set
\[
\left \{s\in \left(\frac {r^*}2,{r^*}\right) : \left |\int_{B(s)} J_xu_\ep(x,z_b) \ dx - n\pi \right| \ge  1\right\}
\]
has measure at least ${r^*}/4$.
Thus the intersection of this set with $\widetilde \calS^\ep(z_g)$ has measure 
at least ${r^*}/8$. As a result, the set
\begin{equation}
\calT := 
\left \{s\in \left(\frac {r^*}2,{r^*}\right) : \left|
\int_{B(s)} J_xu_\ep(x,z_b) \ dx -  \int_{B(s)}J_xu_\ep(x,z_g) dx 
\right| 
\ge  \frac 12 \right\}
\label{calT.def}\end{equation}
satisfies
\begin{equation}
\left| \calT \right| \ge \frac {r^*}8.
\label{calT.big}\end{equation}

{\bf Step 2}. In what follows, if $M$ is an oriented $2$d submanifold of $\Omega$,
we write $\int_M Ju_\ep$ to denote the integral of $Ju_\ep$ over $M$ in the standard sense of
differential geometry (recalling that $Ju_\ep$ is a $2$-form).
Then in particular\footnote{This 
identity specifies our convention for orienting $B(s)\times \{z\} $, which
is the standard one.}
\begin{equation}
\int_{B(s)}J_x u _\ep(x,z)\ dx = \int_{ B(s)\times \{z\} } J u _\ep \, .
\label{L7.e2}\end{equation}
Now let $I = (z_0,z_1)$, where $z_0 = \min (z_b,z_g)$ and $z_1 = \max(z_g,z_b)$.
Since $d \, Ju_\ep =0$, we find from Stokes Theorem that
\[
0
 \ = \ 
\int_{B(s)\times I} d \, Ju_\ep 
\ = \ \int_{\partial(B(s)\times I)} Ju_\ep ,
\]
and upon breaking $\partial(B(s)\times I)$ into pieces in the natural way, we find that
\[
\int_{\partial B(s) \times I} Ju_\ep =  \int_{B(s) \times \{z_0\}} Ju_\ep - \int_{B(s) \times \{z_1\}} Ju_\ep .
\]
It thus follows from  \eqref{L7.e2} and the definition \eqref{calT.def} of $\calT$
that
\begin{equation}
\left|\int_{\partial B(s) \times I} Ju_\ep\right| \ge \frac 12 \qquad\mbox{ if }s\in \calT.
\label{calT1}\end{equation}

\vskip.3in
\begin{center}
\begin{figure}[!ht]
\begin{minipage}{0.8\linewidth}
\includegraphics[trim = 0mm 0mm 0mm 0mm, clip, width=10cm,
  height=10cm,
  keepaspectratio,]{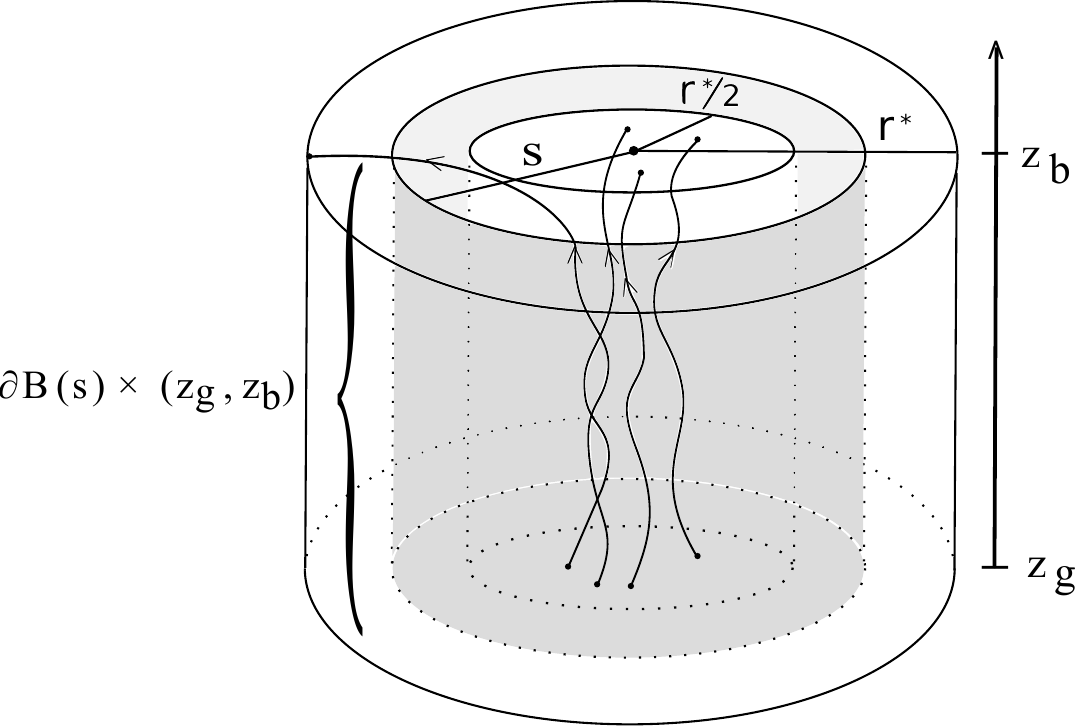}
\end{minipage}
\caption{Configuration with good height $z_g$ and bad height $z_b$.}
\end{figure}
\end{center}
\vskip.3in

{\bf Step 3}. Now define
\[
\calT_* := \left \{s\in \calT \ : \ \int_{\partial B(s)\times \{ z_g, z_b\} } e_\ep(u_\ep) d\calH^1 <
\ep^{-\alpha} \right\}, \quad\qquad \calT^* := \calT\setminus \calT_*.
\]
We will show that $|\calT^*| \ge \frac {r^*}{16}$. To do this, in view of \eqref{calT.big}, it
suffices to show that $|\calT_*| < \frac {r^*}{16}$ for all small $\ep$.
To do this, note that the coarea formula, Lemma \ref{L8} (see also Remark \ref{rem.cyl})
and \eqref{calT1}
imply that
\begin{align*}
\int_{\omega \times I}e_\ep(u_\ep) \ dx \ dz
&\ge
\int_{s\in \calT_*}
\left( \int_{\partial B(s)\times I} e_\ep(u_\ep) \ d\calH^2\right)\\
&\ge 
\frac 12 c_2 \logeps  |\calT_*|- c \ep^{1-\alpha}\logeps.
\end{align*}
On the other hand, we know from Lemma \ref{c.subset}
and \eqref{where.zg} 
that
\[
\int_{\omega \times I}e_\ep(u_\ep) \ dx \ dz
\le n  \pi \logeps (|I|+  o(1))  \  \le n\pi \logeps(\delta + o(1)).
\]
By a suitable choice of $\delta$, we can therefore guarantee that
$|\calT_*|< \frac {r^*}{16}$ for all sufficiently small $\ep>0$.

{\bf Step 4}. 
Since $|\calT^*|   \ge \frac {r^*}{16}$, 
we easily see that
\[
\int_{\omega \times \{z_g, z_b\}} e_\ep(u_\ep) d\calH^2
\ge
\int_{s\in \calT^*} \int_{\partial B(s)\times \{z_g, z_b\}} e_\ep(u_\ep) \ge \frac {r^*}{16} \ep^{-\alpha}.
\]
Since 
\[
\int_{\omega \times \{z_g\}} e_\ep(u_\ep) d\calH^2 \le \frac {r^*}{32}\ep^{-\alpha}
\]
by definition of $\calG^\ep_3$,
we conclude that 
\[
\int_{\omega \times \{z_b\}} e_\ep(u_\ep) d\calH^2 \ge \frac {r^*}{32}\ep^{-\alpha}
\]
for all sufficiently small $\ep$.
This is the conclusion \eqref{P1prop} of the Proposition, up to the factor $\frac{r^*}{32}$, which can be absorbed
by taking a  larger choice of  $\alpha$ and a correspondingly smaller $\ep_0 .$
\end{proof}

\begin{remark}
The proof shows 
that $e_\ep(u_\ep)$ may be replaced by 
$e_\ep^{2d}(u_\ep)$ in the conclusion
\eqref{P1prop}, since in fact only tangential components of the boundary energy
appear in the $\int_{\partial D} e_\ep d\calH^1 \ge \ep^{-\alpha}$ part of the possibilities contemplated in
Lemma \ref{L8}.
\end{remark}

\subsection{Alternate hypothesis \eqref{old.scaling1}}

Recall that in Theorem \ref{main}, we have assumed energy bounds \eqref{scaling2},
together with either \eqref{scaling1} -- \eqref{scaling1aa} or \eqref{old.scaling1}.
In this section we demonstrate a couple of ways in which the latter assumption is 
stronger than the former.
The first shows that the \eqref{old.scaling1} case of Theorem \ref{main} 
implies the  \eqref{scaling1} -- \eqref{scaling1aa} case, and the second 
indicates some ways in which our arguments can be simplified if we assume
\eqref{old.scaling1}.

Following this section we will focus on hypotheses  \eqref{scaling1} -- \eqref{scaling1aa} which, in addition to being more subtle, are also the hypotheses we need for our applications in Theorem \ref{minimizers}.

\begin{lemma}
If $(u_\ep)\subset H^1(\Omega;\C)$ is a sequence satisfying \eqref{scaling2} and \eqref{old.scaling1},
then there is a subsequence along which 
it also satisfies \eqref{scaling1} -- \eqref{scaling1aa}.
\label{old.new}\end{lemma}

\begin{proof}
Since \eqref{scaling1} follows immediately from \eqref{old.scaling1}, we only
need to find a point $z\in (0,L)$
and a subsequence along which assumptions \eqref{scaling1a} -- \eqref{scaling1aa}
hold.
 
To do this,
we define yet another ``good set",
 \[
\calG^\ep_4 :=
\{ z\in \calG^\ep_2: \| J_x u_\e(\cdot, z) - \pi n \delta_0 \|_{F(\omega)} \le 3 C h_\ep \},
\]
where $C$ is the same constant appearing in \eqref{old.scaling1}.
It follows from \eqref{old.scaling1}, \eqref{Fed2} and  Lemma \ref{L.BadE.est} that
$|\calG^\ep_4| \ge \frac L2$.
By the Borel-Cantelli Lemma, we can therefore 
find some $z_0\in (0,L)$ and a subsequence $\e_k$
such that $z_0\in \calG^{\e_k}_4$ for all $k$.

The fact that $z_0\in \calG^{\ep_k}_2$ for all $k$ implies that \eqref{scaling1aa} holds.
To prove \eqref{scaling1a},
Lemma \ref{L3b} implies that 
for all sufficiently large $k$ there
exist points
$\{ p^\ep_i\}_{i=1}^n$ such that \eqref{L3b.c1} holds, {\em i.e.}
\[
\| J_xu_\ep(\cdot, z_0) - \pi\sum_{i=1}^n\delta_{p^\ep_i}\|_{F(w)} \le \e^a.
\]
(Here and below, we often write $\e$ instead of $\e_k$, to reduce clutter.)
Since $z_0\in \calG^{\e_k}_4$, 
\[
\| n\pi \delta_0 -  \pi\sum_{i=1}^n\delta_{p^\ep_i} \|_{W^{-1,1}(\omega)} \le \ep^a+ 3 Ch_\e \le C h_\ep.
\]
It follows from the definition \eqref{FvsW11} of the flat norm (for example by testing
with $\varphi = (1-|x|)^+$ or a regularization thereof) that $\sum|p^\ep_i| \le C h_\ep$.
If we let $q^\ep_i = p^\ep_i/h_\ep$, then $\sum |q^\ep_i|\le C$, and we may pass to a further subsequence such that $q^\ep_i \rightarrow q^0_i$ for $i=1,\ldots, n$.
Then for this subsequence, by the triangle inequality
\[ 
\| J_xu_\ep(\cdot, z_0) - \pi\sum_{i=1}^n\delta_{h_\ep q^0_i } \|_{F(w)}
 \le \e^a + 
 \pi \sum_{i=1}^n \|\delta_{p^\ep_i} - \delta_{h_\ep q^0_i } \|_{F(\omega)}.
\]
It follows again from the definition \eqref{FvsW11} of the flat norm that the right-hand side
is bounded by $\e^a+ \sum|p^\ep_i - h_\ep q^0_i| = \e^a + h_\ep |q^\ep_i - q^0_i|$,
and this is clearly $o(h_\ep)$ as $\ep\to 0$.

\end{proof}

Our next result is not used anywhere in this paper. Its significance is that
it shows that, if we allow ourselves the stronger assumption \eqref{old.scaling1},
then certain difficulties in the proof of Theorem \ref{main} can be avoided.

\begin{lemma}
\label{extractingdivergentpart}
If $(u_\e)\subseteq H^1(\Om;\C)$ is a family satisfying \eqref{scaling2}
and \eqref{old.scaling1},
then
\beq \label{uptoconstant}\int_\Om e^{2d}_\e(u_\e)\geq n\pi L\abs{\log \e}+\pi n (n-1)L\abs{\log h_\e}- C\eeq
and
\begin{equation}\label{dzu2}
\int_\Om\abs{\p_z u_\e}^2 dx\;dz\leq C
.\end{equation}
\end{lemma}

In the next section, when proving Theorem \ref{main} under the weaker hypotheses
\eqref{scaling1} -- \eqref{scaling1aa}, we only know at the outset
that $\int_\Omega|\partial_z u_\ep|^2 dx\,dz = o(\logeps)$, see \eqref{uzolog}. This is considerably
weaker than \eqref{dzu2}, and this weakness introduces
some complications into our arguments.

\begin{proof}
According to Lemma \ref{L3a}, if $z$ belongs to the set  $\calG^\ep_2$, defined in  \eqref{GBep2.def}, 
then the lower energy bound  \eqref{L3a.c1} holds.
In addition, it follows from Proposition \ref{P1}
and the definition of $\calG^\ep_2$ that
\[
\int_\omega e_\ep^{2d}(u_\ep(x,z))dx \ \ge \pi(n+\theta)\logeps 
\qquad \mbox{  for $z\not\in \calG^\ep_2$.}
\]
From this and  \eqref{L3a.c1},
\begin{align*}
\int_\Omega e_\ep^{2d}(u_\ep)dx\,dz 
&\ \ge \ (L  -  |\calG^\ep_2|) \pi(n+\theta)\logeps \\
& \ + \ \int_{z\in\calG^\ep_2}\left(n\pi|\log\ep|-\pi n(n-1)\log\|J_x u_\ep(\cdot, z)-n\pi\delta_0\|_{F(\om)}-C\right)\,dz \\
&\ge  n \pi L \logeps -C - \pi n(n-1) 
 \ \int_{z\in\calG^\ep_2}\log\|J_x u_\ep(\cdot, z)-n\pi\delta_0\|_{F(\om)}\,dz .
\end{align*}
To estimate the integral, note that by Jensen's inequality
and \eqref{old.scaling1},
\begin{multline*}
-\int_{\calG^\ep_2} \log\|J_x u_\ep(\cdot,z)- n\pi\delta_0\|_{F(\om)}dz \\
\ge 
-|\calG^\ep_2| \log \left(\frac 1{|\calG^\ep_2|}\int_{\calG^\ep_2} \|J_x u_\ep(\cdot,z)-n\pi\delta_0\|_{F(\om)}dz \right)
\ge - L \log h_\e - C.
 \end{multline*}
This proves \eqref{uptoconstant}. 
Finally, \eqref{dzu2}  is immediate from \eqref{uptoconstant} and \eqref{scaling2},
since $e_\e(u) = e_\e^{2d}(u) + \frac 12 |\frac{\partial u}{\partial z}|^2$.
\end{proof}

\section{Compactness and lower bound}\label{section3}

In this section we prove part $(a)$ of Theorem \ref{main}, consisting of the
compactness and lower bound assertions.
In view of Lemma \ref{old.new}, it suffices
to consider assumptions \eqref{scaling1} -- \eqref{scaling1aa}, together with \eqref{scaling2}.

Our strategy will be to rescale on a scale $\ell_\ep$ 
chosen to facilitate the proof of the
compactness assertions. We will also
obtain  energy lower bounds that depend on $\ell_\ep$ in
a way that will allow us to conclude, 
only in the final step of the proof, that in fact $\ell_\ep = h_\ep$.
This will require us to be rather careful about how some of our
estimates depend on certain constants, such as $c_3$, defined below.

Thus the rescaling will turn out to be the same as that
in the statement of the theorem, and the lower energy bounds 
we prove will reduce to \eqref{uniflowse}.

Thus, we fix a sequence $(u_\ep)$ satisfying \eqref{scaling1} and \eqref{scaling2},
and we define
\begin{equation}
\ell_\ep :=  \max \left\{ h_\e\, , \,  \left ( \frac 1{c_3\logeps} \int_\Omega |\partial_z u_\ep|^2\, dx \, dz 
\right)^{1/2}\right\} \ 
\label{ellep.def}\end{equation}
for a constant $c_3$ that will be specified later.
(We will invoke assumptions \eqref{scaling1a} and \eqref{scaling1aa} only when they are needed). It follows from the definition of $\ell_\ep$ and from \eqref{uzolog} that
\begin{equation}\label{hell}
h_\ep \le \ell_\ep = o(1)\qquad\mbox{ as }\ep\to 0.
\end{equation}
Throughout the rest of this section, we will use the notation
\begin{equation}
\omega_\ep = \ell_\ep^{-1} \omega,
\qquad v_\ep(x,z) = u_\ep(\ell_\ep x, z) \quad \mbox{ for }(x,z)\in \Omega_\ep :=\omega_\ep\times (0,L).
\label{ell.rescale}\end{equation}
A change of variables shows that
\begin{equation}
\frac 1\logeps \int_{\Omega_\ep}|\partial_z v_\ep|^2 \, dx\, dz = 
\frac 1{\ell_\ep^2 \logeps} \int_\Omega |\partial_z u_\ep|^2 \, dx \, dz \le  c_3
\label{c5new}\end{equation}
and that, for $\ep' = \e/\ell_\ep$, 
\[
\int_{\Omega_\ep} e^{2d}_{\ep'}(v_\e) \,dx\, dz = 
\int_{\Omega} e^{2d}_\ep(u_\ep) \, dx\,dz \le C \logeps.
\]
From \eqref{hell} and \eqref{hep.def},  
\begin{equation}\label{ep.epprime}
\frac{\abs{\log\ep'}}{\abs{\log\ep}}\to 1\mbox{ as }\ep\to 0,
\end{equation}
and recalling that $e_{\ep'}(v_\ep) = e^{2d}_{\ep'}(v_\ep) + \frac 12 |\partial_z v_\ep|^2$, 
it follows that 
\begin{equation}
\frac 1{|\log\e'|} \int_{\Omega_\ep} e_{\ep'}(v_\ep)\,dx\,dz \le C \ .
\label{c5primenew}\end{equation}

At different stages in the proof of \eqref{comp} we will need to invoke a general compactness result for Ginzburg-Landau functionals of \cite{JSo, abo}.

\begin{theorem}
Let $U$ be a bounded, open subset of $\R^3$.
Assume $(u_\e)\subseteq W^{1,2}(U;\C)$ is such that 
\[\limsup_{\ep\to 0}\frac{1}{\abs{\log\ep}}\int_U e_\ep(u_\ep)dx<\infty.\]
Then, there exists a subsequence $\ep_k\to 0$ and an integer multiplicity rectifiable $1$-current $J$ such that $\frac 1\pi Ju_{\ep_k}$ converges to  $J$ in $W^{-1,1}(U)$. 
In addition, one has the following uniform lower semicontinuity:
\[\liminf_{k\to\infty}\frac{1}{\abs{\log\ep_k}}\int_U e_{\ep_k}(u_{\ep_k})dx\geq \pi M_U(J).\]
\label{UKC}
\end{theorem}

This was first proved in \cite{JSo}, which in fact established compactness in the
$(C^{0,\alpha}_c(U))^*$ norm for all $0<\alpha\le 1$. 
A different proof, which established compactness in the flat norm
$F(U)$, was subsequently given in \cite{abo}, which also proved a corresponding upper bound.

We immediately conclude from \eqref{c5primenew} and Theorem \ref{UKC} (the statement considers a family of maps indexed by $\e '$ but it can be easily seen to also apply to $(v_\e)$) that
there exists an integer multiplicity $1$-current $J$ in $\R^2\times (0,L)$ such that
\beq\label{defJrescaled}
\frac 1\pi \star Jv_\e\rightarrow J\mbox{ in }W^{-1,1}(B(R)\times(0,L)) \quad\mbox{ for all }R >0.
\eeq
Our goal is to show that, roughly speaking, $J$ consists of $n$ graphs  of $H^1$ functions 
over the vertical 
segment $(0,L)$, possibly together with other
pieces that the vertical component $J_xv_\e$ of the vorticity fails to record.

In view of \eqref{c5new}, we may also assume, after passing to a further subsequence, that there exists a 
measure $\mu$ on $[0,L]$ such that 
\begin{equation}
\mu_\ep \rightharpoonup \mu\ \ \mbox{ weakly as measures, \ where }
\quad
\mu_\ep(A) :=  \int_{\om_{\e}\times A }\frac{\abs{\p_z v_{\e}(x,z)}^2}{\abs{\log\e}}dx \, dz.
\label{mu.def}\end{equation}
For this subsequence, general properties of weak convergence of measures imply that
\[
\mu((0,L)) 
\le \liminf_{\e\to 0} \int_{\Omega_\ep}
 \frac{\abs{\p_z v_{\e}(x,z)}^2}{\abs{\log\e}}dx \, dz 
=
\liminf_{\e\to 0} \int_{\Omega}
 \frac{\abs{\p_z u_{\e}(x,z)}^2}{\ell_\ep^2 \logeps} dx \, dz.
\]

The objective now is to identify the filaments.
To that end we will  find a countable dense subset of heights such that, among other things, the slices of the Jacobians at these heights converge to $\pi$ times the sum of $n$ Dirac masses; these are the candidates for the values of $f$ at these heights.
We then establish the existence of a unique $H^1((0,L))$ extension of $f$. This will require control
on the modulus of continuity of $f$; obtaining this control is our first task.

\subsection{Modulus of continuity}

We first establish the basic estimate, see \eqref{basicest} below, that lets us control $\| f'\|_{L^2}^2$
by $\int_\Omega |\partial_z u_\ep|^2$. 
At this stage we do not yet need all the hypotheses of Theorem \ref{main}.

\begin{lemma}\label{labels}
Assume that $(u_\ep)$ satisfies  \eqref{scaling1}, \eqref{scaling2}, 
define $(v_\e)$ by rescaling as in \eqref{ellep.def}, \eqref{ell.rescale},
and define $\mu$ by \eqref{mu.def}.

Assume that  $\{z^\ep_1\}$ and $\{z^\ep_2\}$ are sequences  in $[0,L]$
such that $z^\ep_j\rightarrow z_j$ for $j=1,2$, with $0\le z_1 < z_2 \le L$,
and that the following
conditions hold for $j=1,2$ (perhaps after passing to a subsequence):
\begin{eqnarray}
&&J_x v_\ep(\cdot, z^\ep_j)  \to \pi \sum_{i=1}^{n(z_j)} \delta_{p_i(z_j)} 
\qquad
\mbox{ in }W^{-1,1}(B(R)), \ \ \mbox{ for all }R>0,
\label{zzz1}\end{eqnarray}
(for certain  points $\{ p_i(z_j)\}_{i=1}^{n(z_j)}$, not necessarily distinct) and
\begin{equation}
\limsup_{\e\to 0} \abs{\log\e}^{-1}
\int_{\om} e^{2d}_\e(u_\ep(x,z^\ep_j))dx \le M
\label{zzz2}\end{equation}
for some $M >0$. 
Finally, assume that $n(z_i)= n$
for either $i=1$ or $2$. 
Then $n(z_1) = n(z_2) = n$, and 
\begin{equation}
\frac{\pi}{2}\min_{\sigma\in S^n}\sum_{i=1}^n\frac{\abs{p_i(z_1)-p_{\sigma(i)}(z_2)}^2}{z_2-z_1}
\ \le \  \frac 12 \mu( [z_1,z_2]) \, .
\label{basicest}\end{equation}
\end{lemma}

Here we follow the idea in the proof of Proposition 3 in \cite{waveJerrard},
where however one has somewhat more information, such as bounds on
$\int_{\omega_\ep}  e_\ep^{2d}(v_\ep)$ that are uniform for $z \in (z_1,z_2)$,
as well as distinct limiting vortex curves that are known not to intersect.

\begin{proof}
We first present the proof in the basic case when 
$z^\ep_j = z_j$ for all $\ep$, for $j=1,2$. 
Given a small number $0< \tau$,
we define $\psi_\ep^{\tau}: \om_\ep \times\R \to \R^2$
as follows:
\beq\label{psidt.def}
\psi_\ep^{\tau}(x,z)=
\begin{cases}
v_\ep(x,z_1)\quad &\mbox{ if } z\leq z_1/\tau\\
v_\ep(x,\tau z) &\mbox{ if }z_1/\tau<z< z_2/\tau\\
v_\ep(x,z_2)&\mbox{ if } z_2/\tau\leq z.
\end{cases}
\eeq

{\bf Step 1.}
Let us write $n_i = n(z_i)$ for $i=1,2$, and (as above)
$\e' = \e/\ell_\ep$.
We first prove that $n_1 = n_2 = n$, and that
\begin{multline}
\liminf_{\ep\to 0}\frac{1}{\abs{\log\ep'}}
\int_{z_1/\tau}^{z_2/\tau}\int_{\om_\ep}e_{\ep'}(\psi_\ep^{\tau})dx\,dz\\
\geq \pi\min_{\sigma\in s^n}\sum_{i=1}^n\left(\tau^{-2}(z_1-z_2)^2+ |p_i(z_1)-p_{\sigma(i)}(z_2)|^2\right)^{\frac12} \ .
\label{Massbound}
\end{multline}
For $\delta\ge 0$ we will use the notation
\[
I_\delta = (\frac{z_1}\tau-\delta,  \frac{z_2}\tau+\delta), \qquad U_\delta = \R^2\times I_\delta.
\]

First, some changes of variable show that
\begin{multline}\label{pdt.changevar}
\int_{I_\delta}\int_{\om_\ep}e_{\ep'}(\psi_\ep^{\tau})dx\,dz 
=
\delta \sum_{i=1,2}\int_{\om }e^{2d}_\ep(u_\e(x, z_i)) dx\\
+
\frac 1 \tau
\int_{z_1}^{z_2}\int_{\om_\e} e_{\ep'}^{2d}(v_\ep) dx\, dz 
+
\frac \tau {2} \int_{z_1}^{z_2}\int_{\om_\e} |\partial_z v_\ep|^2 dx\, dz \, .
\end{multline}
In light of \eqref{c5new},  \eqref{ep.epprime}, \eqref{c5primenew},  and \eqref{zzz2},
there is thus a constant $C = C(\tau, \delta)$ such that
\[
\frac{1}{\abs{\log\ep'}}\int_{I_\delta}\int_{\om_\ep}e_{\ep'}(\psi_\ep^{\tau})dx\,dz\leq C
\qquad
\mbox{ for all }\e\in (0,1].
\]
It therefore follows from Theorem \ref{UKC} that there exists a $1$-current $\hat J$  in $U_\delta$
such that,
after passing to a subsequence if necessary,
\[
\frac 1 \pi J\psi_\ep^{\tau}\to \hat{J}\qquad\mbox{ in }W^{-1,1}(B(R)\times I_\delta), \quad\mbox{ for all }R>0.
\]
Moreover, $\hat{J}$ is integer multiplicity rectifiable, with $\p\hat{J}=0$ in $U_\delta$, and
\beq
\liminf_{\ep\to 0}\frac{1}{\abs{\log\ep'}}\int_{I_\delta}\int_{\om_\ep}e_{\ep'}(\psi_\ep^{\tau})dx\,dz
\geq \pi M_{U_\delta}(\hat{J}) .
\label{MJhat}
\eeq
We want to estimate $M_{U_\delta}(\hat J)$.
Note  first from \eqref{zzz1} and definition of $\psi_\ep^{\tau}$ that,
using the notation \eqref{Gammaf.def}, we have
\begin{align}
\hat{J}=\pi \sum_{i=1}^{n_1}  T_{l^i_-}\mbox{  in  } \R^2\times (\tau^{-1}z_1-\delta,\tau^{-1}z_1),
\quad\mbox{ for $l^i_-(z) = (p_i(z_1),z)$.}
\label{JhatTlower}
\\
\hat{J}=\pi \sum_{i=1}^{n_2}  T_{l^i_+}\mbox{  in  }\R^2\times (\tau^{-1}z_2,\tau^{-1}z_2+\delta),
\quad\mbox{ for $l^i_+(z) = (p_i(z_2),z)$.}
\label{JhatTupper}
\end{align}

On the other hand, recalling \eqref{characofonecurrents}, there exist
Lipschitz curves $\{g_i\}_{i\in I}$ such that
\[\hat{J}=\sum_{i\in I}T_{g_i},\quad 
\qquad \p T_{g_i}=0
\mbox{ in  }U_\delta, \ \  \mbox{ for all }i.
\]
and $M(\hat J) = \sum_{i\in I}\mbox{length}(g_i)$.
In particular, certain of these curves must coincide (after reparametrization) with $l^i_-(z)$
for $z\in (\tau^{-1}z_1-\delta,\tau^{-1}z_1)$. We may thus choose to label and parametrize
these $\{ g_i\}$
so that $g_i(z) = l^i_-(z)$ for $z\in  (\tau^{-1}z_1-\delta,\tau^{-1}z_1)$,
for $i=1,\ldots, n_1$ (see figure \ref{f3} below).


\begin{figure}[!ht]
\begin{minipage}{0.8\linewidth}
\includegraphics[trim = 0mm 0mm 0mm 0mm, clip, width=10cm,
  height=10cm,
  keepaspectratio,]{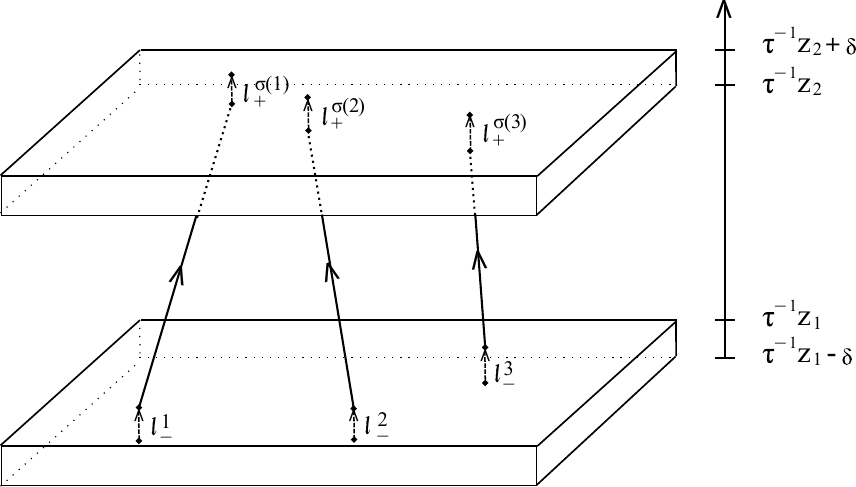}
\end{minipage}
\caption{ Depiction of a $\hat{J}$ whose associated $g_i$'s are of minimal length.}
\label{f3} \end{figure}


Furthermore, because $\p T_{g_i}=0$ for all $i$, and 
all $T_{l^i_\pm} $ are oriented in the same way, we conclude 
each $g_i$, for $i=1,\ldots,n_1$, must connect to one of the curves $l^j_+$,
$j=1,\ldots, n_2$, and each $l^j_+$  must connect to one $g_i$. 
It follows that $n_1=n_2 = n$ (since $n_i = n$ for one of $i=1,2$,
by assumption), and that
there is some $\sigma\in S_m$ such that
$g_i$ connects to $l^{\sigma(i)}_+$. 
Then elementary geometry implies that
\[
\mbox{length}(g_i) \ge2 \delta + | l^i_-(\frac{z_1}\tau) -  l^{\sigma(i)}_+(\frac{z_2}\tau) |
= 2\delta +  \left( \frac{(z_2-z_1)^2}{\tau^2} + |p_i(z_1) - p_{\sigma(i)}(z_2)|^2\right)^\frac 12.
\]
Adding over $i=1,\ldots, n$, we deduce that
\begin{multline*}
\liminf_{\ep\to 0}\frac{1}{\abs{\log\ep'}}
\int_{I_\delta}\int_{\om_\ep}e_{\ep'}(\psi_\ep^{\tau})dx\,dz\\
\geq \pi\min_{\sigma\in S^n}\sum_{i=1}^n\left(\tau^{-2}(z_1-z_2)^2+ |p_i(z_1)-p_{\sigma(i)}(z_2)|^2\right)^{\frac12} + 2 \pi n \delta.
\end{multline*}
Then \eqref{Massbound} follows by sending $\delta\searrow 0$. This involves an interchange
of limits on the left-hand side, which is justified since
\[
\frac 1{|\log\ep'|}\int_{I_\delta\setminus I_0}\int_{\om_\ep}e_{\ep'}(\psi_\ep^{\tau})dx\,dz
\overset{\eqref{zzz2}}
\le \frac{|\log\ep|}{|\log\ep'|} M \pi \delta \overset{\eqref{ep.epprime}}\le 2M\pi\delta
\]
for all positive $\delta$ and all sufficiently small $\ep>0$.

\medskip
{\bf Step 2.}
We next show that
\begin{eqnarray}
&&\limsup_{\ep\to 0}\frac{1}{\abs{\log\ep'}}\int_{z_1/\tau}^{z_2/\tau}\int_{\om_\ep}e_{\ep'}^{2d}(\psi_\ep^{\tau})dx\,dz
= \frac{n\pi}{\tau}(z_2-z_1)
\label{psi2dup}
\end{eqnarray}
Indeed, changing variables as in
\eqref{pdt.changevar}
we have
\[
\int_{z_1/\tau}^{z_2/\tau}\int_{\om_\ep}e_{\ep'}^{2d}(\psi_\ep^{\tau})dx\,dz
=
\frac 1 \tau\int_{z_1}^{z_2}\int_{\om_\e} e_{\ep'}^{2d}(v_\ep) dx\, dz 
=
\frac 1 \tau\int_{z_1}^{z_2}\int_\om e_\ep^{2d}(u_\ep) dx\, dz .
\]
Then the claim  follows by dividing by $|\log\ep'|$, recalling \eqref{ep.epprime}, 
using \eqref{Uon2D}, with $S =(z_1,z_2)$,
and sending $\ep\to 0$.

\medskip
{\bf Step 3.}
Standard properties of weak convergence imply that
\[
\mu([z_1,z_2])
\ge
\liminf_{\ep\to 0} \mu_\ep( [z_1,z_2]) 
=
\liminf_{\ep\to 0}\frac{1}{2}\int_{z_1}^{z_2}
\int_{\om_\ep}\frac{\abs{\p_z v_\ep}^2}{\abs{\log\ep}}dx\,dz .
\]
And from the previous steps we deduce that
\begin{align*}
\liminf_{\ep\to 0}\frac{1}{2}\int_{z_1}^{z_2}
&\int_{\om_\ep}\frac{\abs{\p_z v_\ep}^2}{\abs{\log\ep}}dx\,dz 
=
\liminf_{\ep\to 0}\frac{1}{\tau}\int_{\tau^{-1}z_1}^{\tau^{-1}z_2}\int_{\om_\ep}
\frac{\abs{\p_z \psi_\ep^{\tau}}^2}{2\abs{\log\ep}}dx\,dz
\\
&=\liminf_{\ep\to 0}\frac{1}{\tau}\int_{\tau^{-1}z_1}^{\tau^{-1}z_2}\int_{\om_\ep}
\frac{e_\ep( \psi_\ep^{\tau})-e_\ep^{2d}( \psi_\ep^{\tau})}{\abs{\log\ep}}dx\,dz\\
&\ge \frac{\pi}{\tau}\min_{\sigma\in S^m}\sum_{i=1}^n
\left(\left(\frac{(z_1-z_2)^2}{\tau^2}+ |p_i(z_1)-p_{\sigma(i)}(z_2)|^2\right)^{\frac12}-\frac{(z_2-z_1)}\tau\right).
\end{align*}
Since the left-hand side is independent of $\tau$, we can take the $\tau\to 0$
limit of the right-hand side to 
deduce \eqref{basicest}. 

{\bf Step 4}. Now assume that $z^\ep_1,z^\ep_2$ depend nontrivially on $\ep$.
For each $\ep$, define
\[
V_\ep(x,z)=
\begin{cases}
v_\ep(x,z_1^\ep)\quad &\mbox{ if } z\leq z_1^\ep\\
v_\ep(x, z) &\mbox{ if }z^\ep_1<z< z^\ep_2\\
v_\ep(x,z^\ep_2)&\mbox{ if } z^\ep_2 \le z.
\end{cases}
\]
Then for any $Z_1< z_1<z_2<Z_2$, we may apply the previous case
on the (fixed) interval $(Z_1,Z_2)$, since $v_\e(x,z^\ep_j) = V_\ep(x, Z_j)$
for $j=1,2$ and all sufficiently small $\ep$. Then \eqref{basicest}
implies that 
\[
\frac{\pi}{2}\min_{\sigma\in S^m}\sum_{i=1}^n\frac{\abs{p_i(z_1)-p_{\sigma(i)}(z_2)}^2}{Z_2-Z_1}
\ \le \  \frac 12  \tilde \mu( [Z_1,Z_2]) \ \le \frac 12 \mu([Z_1,Z_2]),
\]
where $\tilde \mu$ is the measure generated as in \eqref{mu.def}, but by $|\partial_z V_\ep|^2$ rather than $|\partial_z v_\ep|^2$.
We conclude the proof by letting $Z_1\nearrow z_1$ and $Z_2\searrow z_2$.\end{proof}

\subsection{Compactness and  lower bounds at {\em a.e} height}

We will  use the notation
\begin{equation}\label{xiep.def}
\xi_\ep(z) :=
 \int_\omega e^{2d}_\e(u_\ep(x,z) \, dx -  \big[
  n(\pi \logeps+\gamma)   + n(n-1)\pi  |\log\ell_\ep |
- n^2 \pi H_\omega(0,0) \big].
\end{equation}

\begin{lemma}\label{lemma11}
Assume that $(u_\ep)\subset H^1(\Omega;\C)$ satisfies \eqref{scaling1} -- \eqref{scaling1aa}
and \eqref{scaling2}.

$(a)$ There exist $\ep_0>0$ and  $C>0$ 
\begin{equation}
\xi_\ep(z) \ge - C \qquad \mbox{ for all $z\in \calG^\ep_2$
and $0<\e<\ep_0$}.
\label{qep.uniform}\end{equation}
where $\calG^\ep_2$ was defined in \eqref{GBep2.def}.

$(b)$ In addition, suppose that for some $z$ and some sequence $\e_k \searrow 0$, 
\begin{equation}\label{zinGep}
z\in \calG^{\e_k}_2\quad \mbox{ for all sufficiently large $k$}.
\end{equation}
Then, after possibly passing to a further subsequence, there
exist points $q_i(z), i=1,\ldots, n$ such that 
\begin{equation}\label{n.positive}
J_x v_\ep(\cdot, z)\rightarrow \pi \sum_{i=1}^n \delta_{q_i(z)} \quad\mbox{ in }
W^{-1,1}(B(R)),\mbox{ for all }R>0,
\end{equation}
\begin{equation}\label{log.lbd1}
\liminf_{k\to \infty}\xi_{\ep_k}(z) \ge -\pi \sum_{i\ne j} \log|q_i(z) - q_j(z)|,
\end{equation}
and, setting $c_4 := \max_i |q^0_i|$, where $q^0_1,\ldots , q^0_n$
appear in \eqref{scaling1a}, 
\begin{equation}\label{qi.bound}
|q_i(z)| \le c_4 + ( \frac{c_3\,L}{\pi})^{1/2} \, \qquad\mbox{ for all }i.
\end{equation}
\end{lemma}

In our notation, as with $q_i(z)$ above, we will consistenly fail
to indicate the dependence of various limiting quantities on the
subsequence that generates them.

The uniform lower bound \eqref{qep.uniform}, needed for our $\Gamma$-limit lower bound, is proved using a compactness argument, and the proof of the lemma begins by assembling the necessary compactness assertions.


\begin{proof}
{\bf Step 1.} Recall that if  $z\in \calG^\ep_2$, then
$u_\ep(\cdot, z)$ satisfies the hypotheses of Lemma \ref{L3b}.
There thus exist points $\{ p^\ep_i(z) \}_{i=1}^n$
that satisfy
\eqref{L3b.c1}, \eqref{L3b.c2}, \eqref{L3a.sharper}.
To express these conclusions in terms of $v_\e$, note that 
by rescaling, one has 
\[
\| J_x v_\ep(\cdot,z) - \pi \sum_{i=1}^n \delta_{\ell_\e^{-1}p^\ep_i(z)} \|_{F(\omega_\ep)} 
\le
\ell_\ep^{-1}\| J_x u_\ep(\cdot,z) - \pi\sum_{i=1}^n \delta_{p^\ep_i(z)} \|_{F(\omega)} . 
\]
This is straightforward to check from the definition \eqref{FvsW11} of the flat norm.
Thus \eqref{L3b.c1}, \eqref{L3b.c2} imply
that 
for $q^\ep_i(z) := p^\ep_i(z)/\ell_\ep$,
\begin{equation}\label{qep1}
\| J_x v_\ep(\cdot,z) - \pi \sum_{i=1}^n\delta_{q^\ep_i(z)} \|_{F(\omega_\ep)} 
\le \ell_\ep^{-1}\ep^a ,
\end{equation}
\begin{equation}\label{qep2}
|q_i^\ep(z) - q_j^\ep(z)| \ge \ell_\e^{-1}\ep^b \qquad\mbox{ for }i\ne j \, .
\end{equation}

Now consider any sequence $(\ep_k)$ tending to $0$ and a sequence of points
$(z^{\ep_k})\subset \calG^{\ep_k}_2$ and $z^{\ep_k}\rightarrow z$,
for some $z\in [0,L]$.
We may pass to a further subsequence and relabel if necessary to
find some integer $n(z)\le n$, and points
$q_i(z)\in \R^2$ for $i=1,\ldots, n(z)$, such that 
\begin{equation}\label{q.converge}
q^{\e_k}_i(z^{\ep_k})\rightarrow q_i(z) \mbox{ for }1 \le i\le n(z), 
\qquad
|q^{\e_k}_i|\rightarrow \infty\mbox{ for }  n(z)< i \le n.
\end{equation}
Then it follows easily from \eqref{qep1} and  standard properties of the $W^{-1,1}$ norm
(and is easily checked from the definition \eqref{FvsW11}) that 
\begin{equation}\label{vep.comp1}
J_xv_{\ep_k}(\cdot, z^{\ep_k}) \rightarrow \pi \sum_{i=1}^{n(z)} \delta_{q_i(z)} \ 
\mbox{ in }W^{-1,1}(B(R)), \  \ \mbox{ for every }R>0.
\end{equation}

{\bf Step 2}. We  claim that
under these conditions, 
$n(z)=n$, and \eqref{qi.bound} holds.

We first assume that $z\ne z_0$, where
$z_0$ is the height
appearing in \eqref{scaling1a}, \eqref{scaling1aa}.
For concreteness we assume that $z<z_0$; the other case is identical.
We want to apply Lemma \ref{labels} along the subsequence fixed above,
with $z^{\ep_k}_1 = z^{\ep_k}$  and $z^{\ep_k}_2 = z_0$.
To verify the hypotheses of the lemma, we first  rescale
\eqref{scaling1a} to find that
\[
\| J_x v_\ep(\cdot, z_0)  - \pi \sum_{i=1}^n \delta_{h_\ep q^0_i/\ell_\ep} \|_{W^{-1,1}(\omega_\ep)} = o(\frac {h_\e}{\ell_\ep}) = o(1)
\]
as $\ep\to 0$. Since $\ell_\ep \ge h_\ep$ by construction,
we may assume after passing to a subsequence that $h_\ep/\ell_\ep \to \alpha$
as $\ep \to 0$, for some $\alpha\in [0,1]$.
Then
\begin{equation}\label{zzero}
J_x v_\ep(\cdot, z_0)  \to \pi \sum_{i=1}^n \delta_{\alpha  q^0_i }
\quad\mbox{ in }W^{-1,1}(B(R)) \ \mbox{ for all }R>0.
\end{equation}
This is one of the hypotheses on $(z^{\ep_k}_2)$ in Lemma \ref{labels}. The other hypothesis follows directly from \eqref{scaling1aa}.  The same hypotheses are satisfied
by $(z^{\ep_k}_1)$, by \eqref{vep.comp1} and the fact that $z^{\ep_k}_1 \in \mathcal G^{\ep_k}_2$.
We may therefore apply this lemma to find that $n(z) = n$ if $z\ne z_0$. 
In addition, since $\mu([0,L]) \le c_3$ due to \eqref{c5new} and \eqref{mu.def},
we deduce from \eqref{basicest} that for some $\sigma\in S^n$, 
\[
\sum_{i=1}^n |q_i(z) - \alpha q^0_{\sigma(i)}|^2 \le  \frac {c_3}\pi |z-z_0| \le c_3\frac L\pi
\]
which implies that  $\max |q_i(z)|  \le c_4 + \sqrt{c_3 L/\pi}$ if $z\ne z_0$. 

If $z = z_0$, we apply Lemma \ref{labels} with $z^\ep = z^1_\ep$, and
$z^2_\ep$ any sequence  in $\calG^\ep_k$ such that $z^\ep_2\rightarrow z_2\ne z_0$,
along some subsequence such that $J_x v_{\ep} (\cdot, z^\ep_z)$ converges
to a limit as in \eqref{zzz1}. Since we may take $z_2$ as close as we like to
$z_0$, we may repeat the arguments from above to easily conclude
that \eqref{qi.bound} holds
in this case as well.

{\bf Step 3}. 
We claim that 
there exists $\ep_0>0$ 
such that 
\begin{equation}
\max_i |q^\ep_i(z)| <  c_4 +2 \sqrt{ c_3 L/\pi}
\qquad \mbox{ for all $0<\e<\e_0$ and $z\in \calG^\ep_2$.}
\label{qep.uniform1}\end{equation}
Assume toward a contradiction that \eqref{qep.uniform1} fails. Then we may find
a sequence $(\ep_k)$ tending to $0$ and a sequence of points
$(z^{\ep_k})\subset \calG^{\ep_k}_2$ such that
\begin{equation}\label{q.toofar}
\max_i  |q^{\ep_k}_i(z^{\ep_k})|  \ge c_4 + 2\sqrt{c_3 L/\pi} \qquad\mbox{ for all $k$}.
\end{equation}
We may pass to a subsequence such that $z_{\ep_k}$ converges to
a limit $z$, and in addition \eqref{q.converge} and \eqref{vep.comp1} hold,
with $n(z)=n$. Clearly, \eqref{q.toofar} implies that $\max_i |q_i(z)|\ge 
c_4+ 2\sqrt{c_3L/\pi} $, which is impossible in view of
 \eqref{qi.bound}. This contradiction completes the proof.

{\bf Step 4}. 
For $z\in \calG^\ep_2$, 
estimate \eqref{L3a.sharper} states that
\[
\int_{\omega}e^{2d}_\e(u_\e(x,z)) dx \ge n(\pi \logeps+\gamma) + W_\omega(p^\ep_1(z),\ldots, p^\ep_n(z))
- C \e^{(a-b)/2},
\]
where according to   \eqref{expforren},
\[
W_\omega(p^\ep_1(z),\ldots, p^\ep_n(z)) =
-\pi\left( \sum_{i\ne j} \log|p^\ep_i(z)-p^\ep_j(z)| + \sum_{i,j}H_\omega(p^\ep_i(z), p^\ep_j(z)) \right).
\]
We also know from \eqref{qep.uniform1} that
$|p^\ep_i(z)|\le C \ell_\ep$
 for $0<\ep<\ep_0$, with a constant independent of $z$. 
Since $H_\omega$ is smooth in the interior of $\omega\times \omega$,
it follows that $H_\omega(p^\ep_i(z), p^\ep_j(z)) =
H_\omega(0,0) + O(\ell_\ep)$ for every $i,j$. 
Writing $p^\ep_i = \ell_\ep q^\ep_i$ as in Step 1,
we deduce that 
\begin{equation}\label{log.lbd1a}
\xi_\ep(z) \ge -\pi \sum_{i\ne j} \log|q^\ep_i(z) - q^\ep_j(z)|  - C \ell_\ep
- C\ep^{(a-b)/2} \ .
\end{equation}
Conclusion \eqref{qep.uniform} follows
immediately from this estimate together with \eqref{qep.uniform1}.

Now assume that $z$ satisfies \eqref{zinGep}.
From Steps 1 and 2 we know that 
we may find a subsequence
such that \eqref{vep.comp1} holds,
with $n(z)=n$. 
This is \eqref{n.positive}. In addition, it follows immediately from \eqref{log.lbd1a} that \eqref{log.lbd1}
is satisfied as well. We have already verified in Step 2 above that 
\eqref{qi.bound} holds, so this completes the proof.

\end{proof}

\subsection{Identifying the filaments}\label{identfil}

We next use the basic estimate \eqref{basicest} 
to choose a subsequence for which the vorticities converge at {\em a.e.} height.
To this end, it is convenient to  introduce some notation.
Let $X$ denote the quotient space $(\R^2)^n/S^n$. 
Thus, points in $X$ consist of equivalence classes 
in $(\R^2)^n$, where $p \sim p'$ if
there exists some permutation $\sigma$ such that
$p_i = p'_{\sigma(i)}$ for all $i=1,\ldots, n$.
The equivalence class containing $p$ will be denoted $[p]$.
The natural notion of distance in $X$ is
\begin{equation}\label{dX.def}
d_X([p],[p'])^2 = \min_{\sigma \in S^n}
 \sum_{i=1}^n |p_i - p'_{\sigma(i)}|^2, \qquad
\mbox{ for }p, p' \in (\R^2)^n.
\end{equation}
We will write 
\begin{equation}
\delta_{[p]} := \sum_{i=1}^n \delta_{p_i}.
\label{delta[p]}\end{equation}
Note that this is well-defined in the sense that the measure on the right-hand side
depends only on the equivalence class.
Similarly, the function
\[
[q]\in X\mapsto - \pi \sum_{i\ne j} \log|q_i(z)-q_j(z)|
\]
is also well-defined, since the right-hand side is invariant under permutations of the indices.
We remark that we adopt the convention that $-\log(0)=+\infty$.

In the following lemma, we identify the limiting vorticity by a map
$(0,L)\to X$, which in effect means that at this stage we do not worry about labelling the
points.

\begin{lemma}\label{L.X}
Assume that $(u_\ep)\subset H^1(\Omega;\C)$ satisfies \eqref{scaling1} - \eqref{scaling1aa}
and  \eqref{scaling2}.

Then there exists a set $H_G\subset (0,L)$ of full measure, a subsequence $(\e_k)$,
and  a function $ [q]: (0,L)\to X$, 
such that for every $z\in H_G$ (as well as for $z=z_0$, from \eqref{scaling1a}), 
\begin{equation}\label{Xcomp1}
J_x v_{\ep_k}(\cdot, z) \rightarrow \pi \delta_{[q](z)} \ \mbox{ in }W^{-1,1}(B(R))\quad\mbox{ for every }R>0
\end{equation}
and
\begin{equation}
\liminf_{\ep\to 0} \
\xi_\ep(z) \ge - \pi \sum_{i\ne j}\log|q_i(z) - q_j(z)| .
\label{Xcomp1a}
\end{equation}
Moreover, for every $z<z'$, 
\begin{equation}\label{Xcomp2}
\pi \frac {d_X([q](z), [q](z')\,)^2}{|z-z'|} \le  \mu( (z,z')).
\end{equation}
\end{lemma}

\begin{proof}
Recall from Lemma \ref{L.BadE.est} 
that under our hypotheses,  $|\calB^\ep_2| = o(1)$ as $\ep\to 0$.
We may therefore choose a subsequence $(\ep_k)$ 
such that 
$\sum |\calB^{\ep_k}_2| <\infty$.
(Note, we will shortly pass to further subsequences.)
Then by the Borel-Cantelli Lemma,
the set
\[
H_G := \bigcup_{\ell=1}^\infty \bigcap_{k=\ell}^\infty \calG^{\e_k}_2
\] 
has full measure in $(0,L)$.

In view of Lemma \ref{lemma11}, we
may  now choose subsequences and invoke a diagonal argument to find a 
set $H_G^0 = \{ z_i\}_{i=0}^\infty \subset H_G$, dense in $(0,L)$, such that
\eqref{n.positive} and \eqref{log.lbd1} hold for every $z_i$. 
Applying Lemma \ref{labels} along this subsequence with $z^\ep_1 = z_i, z^\ep_2=z_j$
for pairs of points $z_i,z_j$ in $H_G^0$,
we find that 
\begin{equation}\label{pre-Holder}
\pi \frac  { d_X( [q(z_i)] ,[ q(z_j)] )^2}{|z_j-z_i|} \le \mu([z_i,z_j]) \le \mu([0,L])
\qquad\mbox{ whenever } z_i< z_j.
\end{equation}
In particular, this states that the map
\[
z_i\in H_G^0 \mapsto [q(z_i)]\in X
\]
is H\"older continuous. It thus has a unique extension to a continuous
map, say $[q]:(0,L)\to X$, such that $[q](z_i) = [q(z_i)]$ for all $i$.
For any $z<z'$, we may fix sequences
$(z_j), (z_j')\in H_G$
such that $z_j\searrow z, z'_j\nearrow z'$. Then
\[
\pi\, d_X ([q](z), [q](z') )^2 
=  \pi \lim_{j\to \infty}
 d_X ([q](z_j), [q](z'_j)) ^2
\le |z-z'| \mu((z,z')) \, .
\]
This proves \eqref{Xcomp2}.

We finally claim that {\em without passing to any further subsequences},
\eqref{Xcomp1} holds for {\em every} $z\in H_G$.
If not, then we could find some $z\in H_G$ and a further subsequence
of the chosen subsequence $(\e_k)$ such that 
\eqref{n.positive} holds but \eqref{Xcomp1} fails.
Then we see from \eqref{basicest} that
\[
\frac \pi 2 d_X( [q(z_i)] ,[ q(z)] )^2 \ \le \  |z_i - z|\ \mu([0,L])
\]
for every $z_i\in H_G^0$. Since $H^0_G$ is dense and $[q(z_i)] =
[q](z_i) \to q[z]$ as $z_i\to z$, it follows that $[q(z)] = [q](z)$,
a contradiction. This proves the claim.

\end{proof}

We next show that $[q]:(0,L)\to X = (\R^2)^ n/S^n$
admits a suitable lifting to a map $(0,L)\to (\R^2)^n$.
This will complete the identification of the limiting vortex filaments.

\begin{lemma}\label{L.choosef}
Let $[q]:(0,L)\to X$ satisfy  \eqref{Xcomp2} for some measure $\mu$ on (0,L).

Then there exists a  function $f \in H^1((0,L),(\R^2)^n)$
such that $[f(z)] = [q](z)$ for all $z\in (0,L)$.
Moreover, whenever $0\le z<z'\le L$,
\begin{align}\label{f.choice}
\pi \sum_{j=1}^n \frac { |f_j(z) - f_j(z')|^2}{|z-z'| } 
&\le \mu((z,z')) \, ,  \qquad\qquad\mbox{ and }
\\
\label{f.choice2}
\pi  \int_z^{z'}\sum_i |f_i'|^2\,dz &\le \mu((z, z'))\, .
\end{align}

\end{lemma}

\begin{proof}
We will write
\[
\calD^i := \{ kL/2^i \ : k=1,\ldots, 2^i-1\}, \qquad \calD := \cup_i \calD^i.
\]
We refer to elements of $\calD$ as dyadic heights.
For every $i$, we can choose $f^i :\calD^i\to (\R^2)^n$
such that 
\begin{equation}\label{lifting1}
[f^i(z)] = [q](z) \quad\mbox{ for all }z\in \calD^i,
\end{equation}
and
\begin{equation}\label{lifting2}
\sum_{j=1}^n |f^i_j(z) - f^i_j(z')|^2 = d_X( [q](z), [q](z'))^2,\qquad
\mbox{ if }z,z'\in \calD^i, \ |z-z'|=2^{-i}L.
\end{equation}

Now fix $i_0$, and for $l>i_0$, and let $f^{i_0,l}$ denote
the restriction of $f^l$ to $\calD^{i_0}$.
For every $i$, it is clear that $f^{i_0,l}$ satisfies \eqref{lifting1}.
For every $[q]\in X$, there are at most $n!$ points $p\in (\R^2)^n$
such that $[p] = [q]$, and hence there are only finitely many maps
$\calD^i\to (\R^2)^n$ that satisfy \eqref{lifting1}. By
the pigeonhole principle, we can thus 
find a subsequence $l_m\to\infty$ along which
$f^{i_0,l_m}(z)$ is independent of $l_m$, for all large enough $m$,
for every $z\in \calD^{i_0}$.

We then define $f(z)$ for $z\in \calD^{i_0}$ by  requiring that
\begin{equation}
f(z) =  f^{i_0, l_m}(z) 
\quad\mbox{ for all sufficiently large $m$}
\label{f.on.Di}\end{equation}

Now fix $z,z' \in \calD^{i_0}$ and some $l$ from the subsequence $(l_m)$ such that \eqref{f.on.Di} holds.
We assume that $z<z'$, and we 
let $z_s := z+ 2^{-l}s$.
Then Jensen's inequality implies that
\begin{align}
\pi \sum_{j=1}^n  \frac{ |f_j(z) - f_j(z')|^2}{z'-z}
&=
\pi \sum_{j=1}^n \frac 1{z'-z}
\left| \sum_{s=1}^{2^l(z'-z)} f^{l}_j(z_s) - f_j^{l}(z_{s-1}) \right|^2
\nonumber\\
&\le
\pi \sum_{j=1}^n
\sum_{s=1}^{2^l(z'-z)}  2^{l}\left| f^{l}_j(z_s) - f_j^{l}(z_{s-1}) \right|^2
\nonumber\\
&
\overset{\eqref{lifting2}}=
\pi\sum_{s=1}^{2^l(z'-z)}  2^{l} \, d_X(\, [q](z_s) , [q](z_{s-1})\, )^2 
\nonumber\\
&
\overset{\eqref{Xcomp2}}\le
\sum_{s=1}^{2^l(z'-z)}  \mu( (z_{s-1}, z_s) \,)  \le \mu( \, (z,z')\, ) \ .
\label{lifting.est}
\end{align}
When invoking \eqref{lifting2}, we have used the fact that every $z_s$ belongs to $D^{l}$.

Now let $i_1>i_0$ be some element of the chosen
subsequence, for example $l_1$, and let ${f}^{i_1,l_m}$ be
the restriction of $f^{l_m}$ to $\calD^{i_1}$. Arguing as above, 
we  find a further subsequence, still denoted $(l_m)$, 
along which $f^{i_1,l_m}$ is eventually independent of $m$,
and we define ${f}(z)$ for $z\in \calD^{i_1}$ as in
\eqref{f.on.Di}. Note that this is consistent with 
the earlier definition of points in $\calD^{i_0}\subset \calD^{i_1}$,
and that \eqref{lifting.est} holds exactly as before.

Continuing in this way, we define ${f}(z)$ for every $z\in \calD$,
such that $[ f(z)] = [q](z)$, and
\eqref{f.choice} holds for every pair of points in $\calD$.
In particular, it follows that ${f}$ is continuous, as a map 
$\calD \to(\R^2)^n$, and hence has a unique continuous
extension to a function, still denoted ${f}$, mapping $(0,L)\to (\R^2)^n$.
It then easily follows that \eqref{f.choice} holds
and  that $[{f}(z)] = [q](z)$ for all $z\in (0,L)$.

\vskip.1in
\begin{center}
\begin{figure}[!ht]
\begin{minipage}{0.9\linewidth}
\includegraphics[trim = 0mm 0mm 0mm 0mm, clip, width=11cm,
  height=11cm,
  keepaspectratio,]{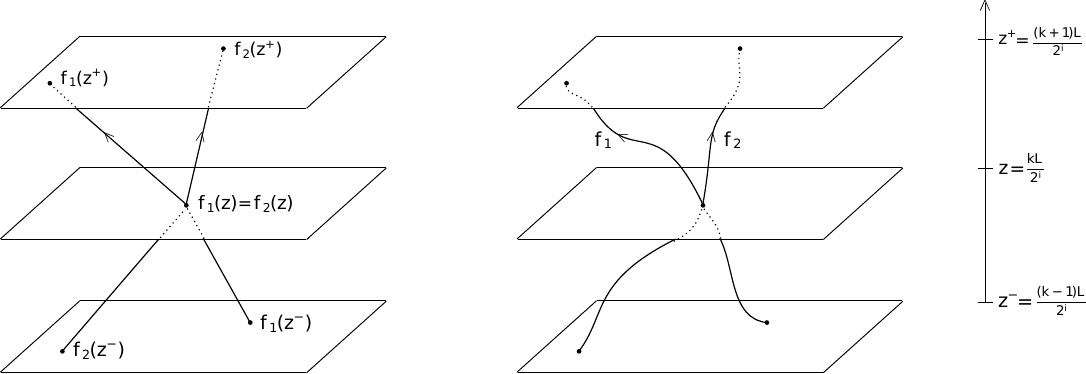}
\end{minipage}
\caption{$2^{-i}$-scale approx. of $f$ with labels and its Lipschitz continuous limit (right). This example of $f$ admits two consistent labellings. }
\end{figure}
\end{center}

Finally, 
we may approximate ${f}$ by a sequence of maps ${f}^\ell:(0,L)\to (\R^2)^n$ that agree with 
${f}$ at a finite number of points and interpolate linearly between these. It
follows from \eqref{f.choice} that such
maps belong to $H^1((0,L);(\R^2)^n)$, with  $\pi \sum\| (f^\ell_i) '\|_{L^2}^2 \le \mu((0,L))$.
Then standard arguments imply that 
\[
(f^\ell_i)' \rightharpoonup f_i' \mbox{ weakly in $L^2((0,L))$, for every $i$,}
\]
and as a result, further standard arguments imply that
\[
\pi \sum_{i=1}^n \int_0^L |f_i'(z)|^2 \ dz \le
\liminf_\ell 
\pi \sum_{i=1}^n \int_0^L |f^\ell_i{}'(z)|^2 \ dz \le  \mu((0,L)).
\]
This proves that $f\in H^1((0,L);(\R^2)^n)$.
The same argument may be carried out on any given subinterval $(z,z')\subset (0,L)$, which
proves that \eqref{f.choice2} holds.
\end{proof}

\subsection{Conclusion of the proof of compactness and lower bound}\label{lyhcoinciden}
\begin{lemma}\label{lb.by.Fatou}
Along the subsequence $(\ep_k)$ found in Lemma \ref{L.X},
\begin{equation} 
\label{e2d.lbd}
\liminf \int_0^L \xi_\ep(z)\, dz
\\
\ge 
-\pi \int_0^L \sum_{i\ne j} \log|f_i(z)- f_j(z)| \, dz.
\end{equation} 
\end{lemma}

\begin{proof}
We have already proved in \eqref{Xcomp1a} that 
\[
\liminf_{\ep\to 0}\xi_\ep(z) \ge   - \pi \sum_{i\ne j} \log|f_i(z)- f_j(z)| 
\]
for {\em a.e.} $z$. Moreover, it follows from Proposition
\ref{P1} that there exists some $\e_0>0$ such that 
$\xi_\ep(z)\ge 0$ for all $z\in \calB^\ep_1$, when $0<\e<\e_0$.
And we have shown in Lemma \ref{lemma11} that, taking $\e_0$ smaller
if necessary, $\xi_\ep(z)\ge -C$ for all $z\in \calG^\e_1$, for $\e<\e_0$.
Thus the conclusion follows from Fatou's Lemma.
\end{proof}

We can now present the

\medskip

\begin{proof}[ Conclusion of the proof of part (a) of Theorem \ref{main} .]
{\bf Step 1}.
We first claim that  if the parameter $c_3$ in the definition of $\ell_\ep$
is taken to be large enough, then
\begin{equation}
\ell_\ep  =  h_\ep 
\qquad\mbox{ for all small $\ep$.}
\label{ellep.bound}\end{equation}
If this fails, then 
we see from the definition  \eqref{ellep.def} of $\ell_\ep$ that 
\[ 
\int_\Omega |\partial_z u_\ep|^2dx\,dz =  c_3(\frac{\ell_\ep}{h_\ep})^2
\] 
Then the definitions of  $G_\ep$ and  $\xi_\ep$
(see \eqref{Gep.def},  \eqref{xiep.def}) imply that
\begin{equation}\label{Gep.split}
G_\ep(u_\ep) = \int_0^L\xi_\ep(z) \, dz - n(n-1)\pi  L\log  (\frac {\ell_\ep}{h_\ep})
+ \frac {c_3}2 (\frac {\ell_\ep}{h_\ep})^2 . 
\end{equation}
Recalling from \eqref{scaling2} that $G_\ep(u_\ep) \le c_1<\infty$ for all $\ep$,
we deduce from  \eqref{e2d.lbd}
that
\[
\limsup_{\ep\to 0} \left[  -n(n-1)\pi L \log(\frac{\ell_\ep}{h_\ep}) + \frac {c_3}2 (\frac {\ell_\ep}{h_\ep})^2 \right]
\le c_1 + \pi  \int_0^L\sum_{i\ne j} \log|f_i(z) - f_j(z)|\, dz  .
\]
If $c_3\ge n(n-1)\pi L$, then the function $s\mapsto -n(n-1)\pi L \log s + \frac {c_3}2 s^2$ is
increasing on $[1,\infty)$.
 Since $\frac{\ell_\ep}{ h_\ep}\ge 1$, it follows that the left-hand side of the above inequality is
 greater than or equal to
 $\frac {c_3}2$. On the other hand,
it follows from \eqref{qi.bound} that $|f_i(z)| \le c_4 + \sqrt{L c_3/\pi}$ for
all $i$ and $z$.
 Putting these together,
we obtain
\[
\frac 12 c_3 \le c_1+ n(n-1)L\pi \log(2c_4+ 2\sqrt{Lc_3/\pi}).
\]
We now fix $c_3$ large enough that this yields a contradiction; then \eqref{ellep.bound} follows.

{\bf Step 2}. Since $\ell_\ep = h_\ep$, and recalling \eqref{c5new}
and \eqref{mu.def},
we can rewrite
\[
G_\ep = \int_0^L \xi_\ep(z)\, dz +  \frac 12 \mu_\ep([0,L]).
\]
Then the $\Gamma$-limit lower bound \eqref{uniflowse}
follows immediately from \eqref{e2d.lbd} and \eqref{f.choice2}.

{\bf Step 3}.
The claim that $[f(z_0)] = [q^0]$ is a consequence of the proof of Lemma \ref{L.X}
(which in particular shows that \eqref{Xcomp1} holds for $z=z_0$),
assumption \eqref{scaling1a}, and Lemma \ref{L.choosef}.

It therefore only remains to improve the convergence already
established in Lemmas \ref{L.X}, \ref{L.choosef} above,
by showing that $J_xv_\e \to \delta_{[f(\cdot)]}$ in 
$W^{-1,1}(B(R)\times (0,L))$ for every $R>0$.
First, recall from \eqref{defJrescaled} that the family of $1$-currents $(\star Jv_\ep)$ is precompact in 
$\cup_RW^{-1,1}(B(R)\times (0,L))$. Then the relationship \eqref{starJ.Jx} between
$\star Jv_\ep$ and $J_x v_\ep$ implies that
$( J_x v_\e)$ is also precompact in the same topology.
So we only need to identify the limit. 
To do this, fix
$\phi\in W^{1,\infty}_c(\R^2\times (0,L))$, and let
$\phi^\ep(x,z) :=\phi(\frac x{h_\ep},z)$ and 
\[
\Phi_{\ep}(z) 
:=\int_{\omega_\ep} \phi(x,z) J_x v_{\ep}(x, z)\, dx
=\int_{\omega} \phi^\ep( x ,z) J_x u_\ep(x,z)\, dx \, .
\]
It follows from  Lemma \ref{L.X} that $\Phi_\ep(z) \to \Phi(z) = \pi \int \phi(x,z) \delta_{[f(z)]}$ for {\em a.e. $z$}, along the chosen  subsequence.
In addition, Theorem 2.1 \cite{JSo}  implies that
for every $z\in (0,L)$,
\[
 |\Phi_\ep(z)| \le   C(\omega) 
(\| \phi^\ep\|_{L^\infty} + \e^\alpha \|\phi^\ep\|_{W^{1,\infty}})
\left (1 + \int_\omega \frac{e^{2d}_\e(u_\ep(\cdot, z))}{\logeps} dx \right)
\]
for some $\alpha>0$.
We then see from Lemma \ref{c.subset} that for any measurable $S\subset (0,L)$,
\[
\int_S |\Phi_\ep(z)|\,dz \le C \| \phi\|_{W^{1,\infty}} ( |S| + o(1)) \ \quad\mbox{ as }\e\to 0,
\]
where the $o(1)$ term is independent of $S$.
Thus the sequence  $(\Phi_\ep)$ is asymptotically uniformly integrable,  so the Vitali Convergence
Theorem implies that in fact $\Phi_\ep \rightarrow \Phi$ in $L^1$.
This implies that
\[
\int_{\Omega_\ep} \phi \ J_xv_\e \ dx \, dz  = 
\int_0^L \Phi_\ep(z)\, dz  \rightarrow \int_0^L \Phi(z)\,dz =
\int \phi \delta_{[f(\cdot)]}.
\]
\end{proof}

\section{Upper bound and improved compactness}\label{section4}

In this section we prove the remaining results of Theorem \ref{main}.

\subsection{Improved compactness for ``tight'' sequences}\label{icfts}

We first prove that sequences which attain the $\Gamma$-limit lower bound
are compact in a stronger sense than previously established in  \eqref{comp}.
The results of the previous section are all available, 
as we continue to assume  hypotheses
\eqref{scaling2}, together with either
\eqref{scaling1} -- \eqref{scaling1aa}  or \eqref{old.scaling1},
although some of these are by now redundant.

The proof requires a measure-theoretic lemma that is proved at the end of this subsection.

\begin{proof}[Proof of part $(c)$ of Theorem \ref{main}]
Let $(u_\e)$ be a sequence satisfying \eqref{comp} and \eqref{recseqbd}.
Recall (see \eqref{defJrescaled}, \eqref{mu.def}) that, up to subsequence, there exists an integer multiplicity rectifiable  $1$-current $J$ such that $\frac 1 \pi \star Jv_\ep\to J$ in $W^{-1,1}(B(R)\times(0,L))$ for every $R>0$.

We have also shown (see \eqref{xiep.def},  \eqref{e2d.lbd}, and recall that in fact $\ell_\ep = h_\ep$) that
\begin{eqnarray}\int_0^L\int_{\om}e_\e^{2d}(u_\e)dx\,dz&\geq&
n\pi L\abs{\log\e}+\pi n(n-1)L\abs{\log h_\e}\nonumber\\
&&-\pi\sum_{i\neq j}\int_0^L\log\abs{f^i(z)-f^j(z)}dz+\kappa_n(\Om)+o_\ep(1) ,
\end{eqnarray}
where $\kappa_n(\Omega)$ was defined in \eqref{kappan.def}.
Finally, recall that there 
exists a measure $\mu$ on $[0,L]$ such that $\mu_\ep\rightharpoonup\mu$ weakly as measures, where
\[
\mu_\ep(A):=\int_{\om_\ep\times A}\frac{\abs{\p_zv_\ep(x,z)}^2}{\abs{\log\ep}}dx\,dz.
\]
and that (see \eqref{f.choice2}) for all measurable $A\subseteq (0,L)$
\beq
\mu(A)\geq \pi \int_A\abs{f'(z)}^2 dz, \qquad\mbox{ where }|f'|^2 := \sum_j |f_j'|^2.
\label{mu.f.A}
\eeq

But then \eqref{recseqbd} implies that
\begin{eqnarray*}
\begin{matrix}
\int_0^L\int_{\om}e^{2d}(u_\e)dx\,dz-n\pi L\abs{\log\e}
\\-\pi n(n-1)L\abs{\log h_\e}
\end{matrix}
\longrightarrow
-\pi\sum_{i\neq j}\int_0^L\log\abs{f^i(z)-f^j(z)}dz+\kappa_n(\Om),\nonumber
\end{eqnarray*}
as $\ep\to 0;$ and
\beq
\mu((0,L))= \pi \int_0^L\abs{{f}'}^2 dz.
\label{mu.f.OL}
\eeq
Since $\abs{{f}'}^2$ is nonnegative, we deduce from \eqref{mu.f.A} and \eqref{mu.f.OL} that $\mu$ is absolutely continuous with respect to Lebesgue measure with density
\[
d\mu= \pi\abs{{f}'}^2 dz.
\]
In particular $\mu$ has no atoms.

Our goal is to show that 
\[
S := J - \sum_{i=1}^n \Gamma_{f_i} = 0.
\]
We  know that for every $\phi\in C^\infty_c(\R^2\times (0,L))$,
\[
\frac 1\pi \star Jv_\ep(\phi \,dz) = \int \phi J_x v_\ep \, dx\,dz \rightarrow \int \phi \delta_{[f(\cdot)]} 
= \sum_{i=1}^n \Gamma_{f_i} (\phi\,dz),
\]
and it follows that $S(\phi\,dz) = 0$ for all such $\phi$. Then it follows from 
Lemma \ref{L.levelset} below that
$S$ may be written as a linear combination of currents
\[
S = \sum_{j\in I} T_{\gamma_j}
\]
where each $T_{\gamma_j}$ is supported in a horizontal plane $\{ z=  z_j\}$.
Assume toward a contradiction that  $T_{\gamma_j}$ is nonzero for $j=1$, say,
and 
let $\ep_0,\delta_0>0$ be two small numbers to be chosen later.
We know from \eqref{Uon2D}, \eqref{ep.epprime} and a change of variables that 
\[ 
\lim_{\ep\to 0} \frac{1}{\abs{\log\ep}}\int_{z_1-\delta_0}^{z_1+\delta_0}\int_{\om_\ep}e_{\ep'}^{2d}(v_\ep)dx\,dz = 
2n\pi\delta_0 
\] 
for $\ep' = \ep/h_\ep$. At the same time we know that
\[
\frac 12
\mu_\ep((z_1-\delta_0,z_1+\delta_0))\to\frac{\pi}{2}\int_{z_1-\delta_0}^{z_1+\delta_0}\abs{{f}'}^2 dz.\]
We can choose $\delta_0$ small enough so that
\[\max\left\{2n\pi\delta_0, \frac{\pi}{2}\int_{z_1-\delta_0}^{z_1+\delta_0}\abs{{f}'}^2 dz \right\}<\frac{\ep_0}{2}.\]
In light of this we have
\beq
\liminf_{\ep\to 0}\frac{1}{\abs{\log\ep}}\int_{z_1-\delta_0}^{z_1+\delta_0}\int_{\om_\ep}e_{\ep'}(v_\ep)dx\,dz<\ep_0.
\label{CTRD1}
\eeq
On the other hand, for every $\delta_0>0$, the mass of the limiting current $J$
in $B(R) \times (z_1-\delta_0, z_1+\delta_0)$ is at least $M(T_{\gamma_1})$,
for $R$ large enough, depending on the support of $T_{\gamma_1}$.
Thus, by applying Theorem \ref{UKC} in $B(R) \times (z_1 -\delta_0, z_1+\delta_0)$
for this choice of $R$, we find that 
\beq
\liminf_{\ep\to 0}\frac{1}{\abs{\log\ep}}\int_{z_1-\delta_0}^{z_1+\delta_0}\int_{\om_\ep}e_{\ep'}(v_\ep)dx\,dz\geq \pi M(T_{\gamma_1}).
\label{CTRD2}
\eeq
Choosing $\ep_0< \pi M(T_{\gamma_1})$ makes \eqref{CTRD1} and \eqref{CTRD2} incompatible.
Hence $S=0$, which completes the proof of \eqref{ICITCOT}.
\end{proof}

\begin{lemma}\label{L.levelset}
Let $S$ be an integer multiplicity rectifiable $1$-current in $\R^2\times (0,L)$
such that $M(S)<\infty$ and $\partial S = 0$ in $\R^2\times (0,L)$.
Assume in addition that $S(\phi \,dz ) = 0$ for all $\phi\in C^\infty_c((\R^2\times (0,L))$.

Then $S$ may be written as a sum of  $1$-currents
\[
S = \sum_{i\in I} T_{\gamma_i}
\]
where each $T_{\gamma_i}$ is supported in a  hyperplane $\{ z = a_i\}$
for some $a_i\in (0,L)$.
\end{lemma}

\begin{center}
\begin{figure}[!ht]
\begin{minipage}{0.5\linewidth}
\includegraphics[trim = 5mm 0mm 0mm 10mm, clip, width=7cm,
  height=7cm,
  keepaspectratio,]{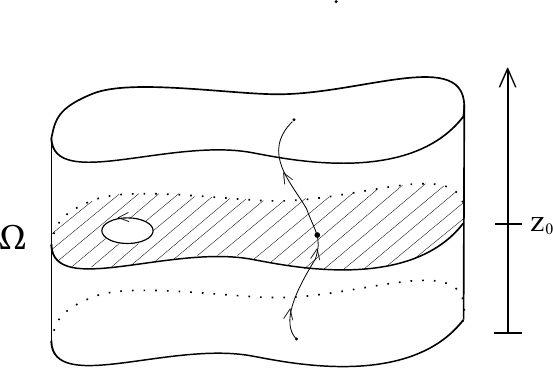}
\end{minipage}
\caption{A current with an indecomposable component supported at $\{z=z_0\}.$}
\end{figure}
\end{center}

\begin{proof}
We first claim that $\partial S = 0$ in $\R^3$.
To see this, fix any $\phi \in C^\infty_c(\R^3)$. For any $\chi\in C^\infty_c((0,L))$
(which we identify in the natural way with a function on $\R^2\times (0,L)$
depending only on the $z$ variable) we have
\[
S(\chi d\phi)  = S( d(\chi \phi)) - S( \phi d \chi) = \partial S (\chi \phi ) - S(\phi \chi' \ dz) = 0
\]
We can thus take a sequence of functions $\chi_k$ such that $\chi_k\nearrow {\bf 1}_{(0,L)}$
to conclude that $S(d\phi)=0$. Since $\phi$ was arbitrary, it follows that $\partial S = 0$
in $\R^3$ as claimed. 

The remainder of the proof is classical; we recall the arguments for the convenience of the reader.
Using the decomposition theorem for $1$-dimensional integer multiplicity currents,
see \cite{Federer} 4.2.25 (to which we refer for the definition of indecomposable)
we may write  $S$ as a sum of indecomposable $1$-currents, say $S = \sum_{i\in I}T_{\gamma_i}$, with $\sum_I M(T_{\gamma_i}) 
= M(S)$ and $\sum_I M(\partial T_{\gamma_i}) = M(\partial S) = 0$.
It follows that $\partial T_{\gamma_i}=0$ and that 
$T_{\gamma_i}(\phi \,dz)= 0$ for every $i$ and  all $\phi \in C^\infty_c$. 
By basic properties of slicing of currents (see for example \cite{Federer} 4.3.2)
this implies that $\langle T_{\gamma_i}, \zeta, z\rangle = 0$ for {\em a.e.} $z\in (0,L)$, where $\zeta:\R^2\times (0,L)\to (0,L)$ is the projection onto the vertical axis, $\zeta(x,z) = z$.
Then Solomon's Separation Lemma  \cite{Solomon}
implies that every $T_{\gamma_i}$ is supported in a level set of
$\zeta$.   

We remark that Solomon's Lemma applies to indecomposable currents $T$ in $\R^n$ such that
$M(T) + M(\partial T)<\infty$; it is for this reason that we verified that $\partial S = 0<\infty$ in $\R^3$.
\end{proof}

\subsection{Constructing a recovery sequence}\label{section5}

In this part, given $ f \in H^1((0,L); (\mathbb{R}^2)^n)$ we build  sequences of maps $(u_\e)\subset H^1(\Om;\C)$ whose Ginzburg-Landau energy recovers the nearly parallel vortex filaments energy $G_0( f) .$
Part $(b)$ of Theorem \ref{main} follows from

\begin{proposition}(Recovery Sequence)
Let ${f}\in H^1\left((0,L);(\R^2)^n\right).$ Then there exists a sequence $(u_\e)\subset H^1(\Om;\C)$ such that

\begin{itemize}
\item[(a)] $\norm{\star J v_\ep-\pi \sum_{i=1}^n\Gamma_{f_i}}_{F(B(R)\times (0,L))}\to 0$ for every $R>0$, as $\e\to 0,$

\item[(b)] $\lim_{\e\to 0}G_{\e}(u_e) \leq G_0( f).$
\end{itemize}
\label{PRSeq}
\end{proposition}

As it is usual in this kind of construction, we define a sequence of trial maps $( u_\ep)$
based on canonical harmonic maps with prescribed singularities in $\om$, introduced in \cite{BBH}, section I.3.
This provides us with the right estimates for $e_\ep^{2d}.$
Then, we show that the $L^2$-norm of $ f'$ is well approximated by $\pi\int_\Om\abs{\p_z  u_\ep}^2 dx\,dz.$
Finally, the task is to show that a sequence defined in this way satisfies $(a)$.

First note that to prove Proposition \ref{PRSeq}, it suffices to prove it for smooth ${f}$ satisfying
\beq\label{noncollision}
\inf_{\begin{matrix}z\in(0,L)\\ i\neq j\in\{1,\ldots,n\}\end{matrix}}
\abs{f^i(z)-f^j(z)}  >0.
\eeq
In fact, once we prove Proposition \ref{PRSeq} for such configurations, a density (and diagonal) argument yields the result.

We will write $f_\e = (f_{\e,1},\ldots, f_{\e,n}) = h_\ep f$.
We will take our trial functions to have the form
\begin{equation}
 u_\ep(x,z) :=  \prod_{j=1}^n \left[   e^{i\beta (x, f_{\ep,j}(z)) }  U_\ep(x-  f_{\e,j}(z)) \right]
\label{hat.u.ep}\end{equation}
for certain functions $\beta,  U_\e$ that we now describe.
First,  if we write $x\in \R^2$ in polar coordinates $(r\cos \theta, r\sin \theta)$, then
we define
\begin{equation}\label{Uhat.outer}
 U_\ep(x) =  e^{i\theta} \qquad \mbox{ if }|x|\ge \sqrt \ep,
\end{equation}
and in $B(\sqrt\ep)$, we choose $ U_\ep$ to minimize
$\int_{B(\sqrt\ep)}e_\ep^{2d}(u)$ among all functions $u: B(\sqrt\ep)\to \C$
such that $u = e^{i\theta}$ on $\partial B(\sqrt \ep)$. Thus,
using notation introduced in \eqref{BBHconstant2},
\begin{equation}\label{inner.energy}
\int_{B(\sqrt\ep)} e^{2d}_\e( U_\ep) \, dx \ = I(\sqrt \e, \e) \ .
\end{equation}
The intermediate length scale $\sqrt \ep$ is chosen
rather arbitrarily -- we just need some radius $r_\ep$ such that $\ep \ll r_\ep \ll h_\ep$.
Theorem 11.2  in \cite{PR} implies that $ U_\ep$ has the form
\[
 U_\ep=  \rho_\ep(r) e^{i\theta} 
\qquad\mbox{with }0 \le \rho_\ep \le 1 \mbox{ for all }r >0
\]
for all sufficiently small $\ep>0$.
Next, $\beta(x, y)$ is chosen so that for every $y\in \omega$,
\begin{equation}
\begin{cases} \Delta_x \beta(x,y) = 0\  &\mbox{for $x$ in }\omega, \\
\nu(x) \cdot \nabla_x \left(e^{i\beta(x,y)}  U_\ep(x - y) \right) = 0 \quad &\mbox{ for $x\in \partial \omega$},\end{cases}
\label{beta.def}\end{equation}
for all $\ep$ small enough such that $| U_\ep(x - y)|=1$ for $x\in \partial \omega$.
This defines $\beta$ up to a constant, which we fix by requiring that $\int_\omega \beta(x,y)dx =0$
for all $y$.

We will only need a couple of facts about $\beta$.
First,
\begin{equation}\label{beta.lip}
|\nabla_y \beta(x,y)| \le C(r) \mbox{ for all }(x,y)\in \omega\times B(r), \qquad\mbox{ if }B(r) \subset \omega .
\end{equation}
This is a straightforward consequence of standard elliptic regularity results.
Second, we have the estimates
\begin{equation}\label{uhatren}
\int_\omega e^{2d}( u_\ep(x,z) \,)\,dx \le
n(\pi\abs{\log\ep}+\gamma)
+W_\om(f_{\e,1}(z), \ldots, f_{\e,n}(z))+C \ep
\end{equation}
and 
\beq
\norm{J_x{u}_{\ep}(\cdot, z) -\pi \delta_{[f_{\ep} (z)]} }_{W^{-1,1}(\om)}\leq C\ep
\label{uhatJac}
\eeq
with constants independent of $z$.
These are proved\footnote{Although the descriptions are a little
different,
our function $ u_{\ep}(\cdot,z)$ is {\em  exactly} the same as 
$u^{r,\e}_\star(\cdot, a,d)$ from \cite{JSp2}, once we set  $r = \sqrt \ep$, $a = (f_{\e,1}(z), \ldots, f_{\e,n}(z))$ and $d = (1,\ldots, 1)$.
If one wants to check this, it is helpful to note that 
$\nabla_x\beta(x, y) = \nabla_x^\perp H(x, y)$, for $H$
solving \eqref{H.def}. 

Experts will recognize that the definition of $\beta$ is such that  
$\frac{ u_\ep(\cdot, z)}{| u_\ep(\cdot, z)|}$ is exactly the canonical harmonic
map with singularities at $(f_{\e,1}(z),\ldots, f_{\ep,n}(z))$
and natural boundary conditions, as  introduced in \cite{BBH} section I.3,
and \eqref{uhatren} expresses this fact.
 }  in  Lemma 14 in \cite{JSp2}.

Rewriting the right-hand side of \eqref{uhatren} using  \eqref{expforren} and \eqref{kappan.def}, and integrating in $z$, 
we have
\begin{eqnarray}
\int_\Om e_\ep^{2d}(u_\ep(x,z)\,)dx\,dz&\leq&
 \pi L n \abs{\log \ep} + \pi L n(n-1)|\log h_\ep|\nonumber\\
&&-\pi  \sum_{i\ne j} \int_0^L\log\abs{f^i(z)-f^j(z)}+\kappa_n(\Om)+\calO\left(h_\ep\right).
\label{RS2Df}
\end{eqnarray}
The dominant contribution
to  the $\calO(h_\ep)$  errors term
comes from the integral of $\sum_{i,j} (H(f_{\ep,i}(z), f_{\ep,j}(z)) - H(0,0))$. 

Also, by setting $ v_\ep(x,z) =  u_\ep(h_\ep x, z)$ and rescaling, we deduce from
\eqref{uhatJac} that
\beq
\sup_{z\in(0,L)}\norm{J_x  v_\ep (\cdot,z)-\pi \delta_{[f(z)] }  }_{W^{-1,1}(\omega_\ep)}
\leq C\ep h_\ep^{-1}\to 0, \quad \mbox{ as }\ep\to 0\, . \label{RS2DJf}
\eeq

In light of this, to prove part $(b)$ of Proposition \ref{PRSeq}, it remains to show
\begin{lemma}
\beq
\limsup_{\ep\to 0}\frac 12\int_\Om\abs{\p_z  u_\ep}^2 dx\,dz\leq
\frac{\pi}{2}\int_0^L\abs{ f'}^2 dz.
\label{FTPoRS}
\eeq
\label{LFTPoRS}
\end{lemma}
\begin{remark}
\label{remarkRS}
We note that, in light of \eqref{RS2DJf}, once \eqref{FTPoRS} is established, we can appeal to Theorem \ref{main} $(c)$ to conclude $\star Jv_\ep\to\pi  \sum \Gamma_{f_i}$, thereby
proving part $(a)$ of Proposition \ref{PRSeq}.
\end{remark}

\begin{proof}[Proof of Lemma \ref{LFTPoRS}]
Let us write 
$\beta_\ep(x,z) := \sum_{i=1}^n \beta(x ; f_{\e,i}(z))$,
so that 
\[
 u_\ep(x,z) = e^{i\beta_\ep(x,z)}\prod_{j=1}^n  U_\e(x-f_{\e,j}(z)).
\]
Then since $\| U_\ep\|_{L^\infty} = 1$,
\[
|\partial_z  u_\ep(x,z)|^2
\le
|\partial_z \beta_\ep(x,z)|^2 + \sum_{i-1}^n 
\left|\partial_z \left(  U_\ep(x - f_{\ep,i}(z))\right) \right|^2
\]
To estimate the first term, note that
\begin{align*}
\partial_z \beta_\ep(x,z) = \sum_i \nabla_y \beta(x; f_{\ep,i}(z))\cdot  f_{\ep,i}'(z).
\end{align*}
Since $ f_{\e,i}'(z) = h_\ep f_i'(z)$, it follows from \eqref{beta.lip} that
\[
| \partial_z \beta_\ep(x,z)|^2 \le C h_\ep^2 \sum_{i=1}^n |f_i'(z)|^2 = C h_\ep^2 |f'(z)|^2.
\]
Next, for every $i=1,\ldots, n$, 
\[
\left|\partial_z \left(  U_\ep(x - f_{\ep,i}(z))\right) \right|^2
= 
h_\ep^2 \left| \nabla_x  U_\ep(x - f_{\ep,i}(z)) \cdot f_i'(z)  \right|^2 .
\]
Thus, fixing $R > \mbox{diam}(\omega)$, 
\begin{align*}
\int_\omega 
\left|\partial_z \left(  U_\ep(x - f_{\ep,i}(z))\right) \right|^2 \, dx
&\le
\int_{B(R, f_{\ep,i}(z)) }
h_\ep^2 \left|  f_i'(z)\cdot \nabla_x  U_\ep(x - f_{\ep,i}(z))  \right|^2 \\
&=
\int_{B(R,0)}
h_\ep^2 \left|   f_i'(z)\cdot \nabla  U_\ep(x)   \right|^2 .
\end{align*}
But the rotational symmetry of $ U_\ep$ implies that
for any vector $v\in \R^2$,
\[
\int_{B(R,0)}
h_\ep^2 \left|  v\cdot \nabla  U_\ep(x)    \right|^2 
=
\int_{B(R,0)}
h_\ep^2 \left|  v^\perp \cdot\nabla  U_\ep(x)  \right|^2 
\]
And since
\[
\left| v\cdot \nabla  U_\ep(x)  \right|^2  +
\left| v^\perp \cdot \nabla  U_\ep(x)   \right|^2  = 
|v|^2 \left| \nabla  U_\ep(x)  \right|^2 
\]
we conclude that
\[
\int_\omega 
\left|\partial_z \left(  U_\ep(x - f_{\ep,i}(z))\right) \right|^2 \, dx
\ \le 
\frac 12 |f_i'(z)|^2 \int_{B(R,0)} \frac{ |\nabla  U_\e(x)|^2}\logeps \, dx
\]
To estimate the right-hand side, we first remark that the constant
$I(r, \e)$ appearing on the right-hand side of \eqref{inner.energy}
is known to satisfy
\[
I(r, \ep ) = \pi \log \left( \frac r \e \right) + \gamma + \calO\left( \left(\frac  \e r \right)^2\right),
\]
see \cite{JSp1}, Lemma  6.8. It follows from this and \eqref{Uhat.outer} that
\[
\frac 12 \int_{B(R,0)} \frac{ |\nabla  U_\e(x)|^2}\logeps \, dx
\le \frac 1{\logeps} \left(\pi \log \frac R \e + \gamma + \calO(\e)\right)
= \pi + \calO \left( \frac 1\logeps \right).
\]
The above estimates combine to show that
\[
\int_\omega |\partial_z u_\ep(x,z)|^2 \,dx \le  \pi \sum_{i=1}^n |f_i'(z)|^2 + \calO \left( \frac{1+ |f_i'(z)|^2}
\logeps \right)
\]
where the error terms are uniform with respect to $z$. The proof is concluded by integrating in $z$.
\end{proof}

\section{The proof of Theorem \ref{minimizers}}\label{section6}

To deduce Theorem \ref{minimizers} from Theorem \ref{main},
we must first verify that it is possible to construct recovery sequences ---
that is, sequences verifying the conclusions of part $(b)$ of
Theorem \ref{main} --- that in addition satisfy the boundary 
conditions \eqref{dirichlet}. Our first lemma
addresses difficulties that arise when the boundary
vortex locations 
$( q^0_1(z), \ldots, q^0_n(z))$, for $z \in \{0,L\}$, are not distinct.

\begin{lemma}\label{L.fep}
Let $f = (f_1,\ldots, f_n)\in H^1([0,L]; (\R^2)^n)$.
Then for all sufficiently small $\delta>0$, where the smallness condition may depend on $f$,
there exists $f^\delta \in H^1([0,L]; (\R^2)^n)$
of the form
\begin{equation}
f^\delta_i = \begin{cases}
f_i &\mbox{ in }\{0, L\}\\
f_i+ a^\delta_i &\mbox{ in }[\delta^{1/2}, L-\delta^{1/2}]\\
\mbox{affine}&\mbox{ in }[0, \delta^{1/2}]\mbox{ and in }[L-\delta^{1/2}, L]
\end{cases}
\label{fep.def1}\end{equation}
for certain  $a^\delta_i\in \R^2$ such that $|a^\delta_i|\le \delta^{1/3}$ and
\begin{align}
\label{fep2}
\max_{z\in [0,L]}|f^\delta_i(z) - f_i(z)|& \le  o(\delta^{1/4})\qquad\mbox{ for all }
i=1,\ldots n,\\
\label{fep3}
|f^\delta_i(z) - f^\delta_j(z)|&\ge \delta^{1/2} \min \{z, L-z,\delta^{1/2} \}  ,\qquad
\mbox{ for }i\ne j.
\\
\label{fep4}
 \int_0^{\sqrt \delta} |(f^\delta)'|^2 + \int_{L-{\sqrt \delta}}^L |(f^\delta)'|^2&= o(1) \quad\mbox{ as }\delta \searrow 0,
\end{align}
and
\begin{equation}
\label{fep5}
\int_{\sqrt \delta}^{L-\sqrt\delta}\left( \frac \pi 2 |(f^\delta)'|^2- \pi \sum_{i\ne j}\log |f^\delta_i - f^\delta_j|\right) \, dz
\to G_0(f) \mbox{ as }\delta\to 0\ .
\end{equation}

\end{lemma}

\begin{proof}
{\bf Step 1}.
We only need to explain how to choose the
points $a^\delta_i$ such that the stated properties hold.
For this, we introduce some notation:
let
\begin{align*}
f_{ij}(z) 
&:= f_i(z)-f_j(z), \\
N^\delta_{ij} &:= \{ x\in \R^2 : \dist(x, Image(f_{ij}))  < \delta \}\\
\mathcal N^\delta_{ij} &:= \{ a\in (\R^2)^n : a_i - a_j \in N^\delta_{ij} \}.
\end{align*}

We will show below that we can find
\begin{equation}\label{chooseaep}
a^\delta\in
 \{ a\in (\R^2)^n : |a_i| \le \delta^{1/3}\mbox{ for all $i$, } \ a\not\in \cup_{i,j} \mathcal N^\delta_{ij} \}.
\end{equation}
First, we assume that we have found such a point $a^\delta$,
and we check that $f^\delta$ then has the desired properties.
For $z\in [\delta^{1/2}, L-\delta^{1/2}]$, we deduce 
from \eqref{fep.def1} and \eqref{chooseaep} that $|f^\delta_i(z) - f_i(z)|\le \delta^{1/3}$,
which implies \eqref{fep2} for these $z$.
In particular, since
\[
|f^i(z)- f^i(0)| \le \sqrt{z} (\int_0^z |f'|\,dz)^{1/2} = o(\sqrt z) \quad\mbox{ as }z\searrow 0,
\]
it follows that 
$|f^\delta_i(\delta^{1/2}) - f_i(0)| \le o(\delta^{1/4})$.
Then for $z\in [0,\delta^{1/2}] $,
it follows that 
\[
|f^\delta_i(z) - f_i(z)| \le 
|f^\delta_i(z) - f_i(0)| +|f_i(0) - f_i(z)|  = o(\delta^{1/4})
\quad\mbox{ as }\delta\to 0.
\]
A similar argument of course holds near $z=L$, completing the proof of
\eqref{fep2}.
To verify \eqref{fep3}, first suppose that $\delta^{1/2} < z<L-\delta^{1/2}$.
Since $a^\delta\not\in \mathcal N^\delta_{ij}$, the definitions imply that 
for every $z\in (0,L)$
\[
\delta \le | f_{ij}(z) - (a^\delta_i-a^\delta_j)| = |f^\delta_i(z) - f^\delta_jz)|.
\]
This is \eqref{fep3} for $z\in [\delta^{1/2}, L-\delta^{1/2}]$. For other values
of $z$ it follows from the fact that $f^\delta$ is affine in $[0,\delta^{1/2}]$
and in $[L-\delta^{1/2}, L]$. 

One may deduce \eqref{fep4} rather directly from
the above estimates, which imply that $|(f^\delta)'(z)| = o(\delta^{-1/4})$ for $z\in [0,\delta^{1/2}] \cup [L-\delta^{1/2}, L]$.

Next we observe that for $z\in[\delta^{1/2}, L-\delta^{1/2}]$,
if  $i\ne j$ and $\delta$ is small enough, then 
\[ 
|f^\delta_i(z) - f^\delta_j(z)| \ge
\begin{cases}
|f_i(z)-f_j(z)|^4\quad&\mbox{ if }|f_i(z)-f_j(z)|\le \delta^{1/4}\\
\frac 12|f_i(z) - f_j(z)|\quad&\mbox{ otherwise} 
\end{cases}
\] 
This follows from \eqref{fep3}  in the first case, and 
in the second case from the fact that
$|a^\delta_i|\le \delta^{1/3}$. Since $f$ is bounded,
it follows that $|f^\delta_i-f^\delta_j| \ge C^{-1}|f_i-f_j|^4$
for $C = C(f)$, and hence that
\[
-\pi \sum_{i\ne j}\log|f^\delta_i(z)-f^\delta_j(z)| \le 
-4\pi \sum_{i\ne j}\log|f_i(z)-f_j(z)|  + C
\]
if $z\in[\delta^{1/2}, L-\delta^{1/2}]$.
Thus we can apply the dominated convergence theorem to the logarithmic
interaction terms in \eqref{fep5}. Since  $(f^\delta)' = f'$ in $[\delta^{1/2}, L-\delta^{1/2}]$,
the other terms converge trivially, and \eqref{fep5} follows.

{\bf Step 2}.
To find $a^\delta$ satisfying \eqref{chooseaep}, it suffices to show that
\begin{equation}\label{notaep}
\calL^{2n}\left(
\cup_{i\ne j}\calB^\delta_{ij}
\right) \ll  \delta^{2n/3},\quad
\mbox{ for }
\calB^\delta_{ij} := \{  a\in \mathcal N^\delta_{ij} : |a_i|< \delta^{1/3} \mbox{ $\forall$  $i$} \ \},
\end{equation}
since then the set in \eqref{chooseaep}, which is the complement  of $\cup \calB^\delta_{ij}$
in $\{a :\max |a_i|\le \delta^{1/3}\}$, must have positive measure.
To prove \eqref{notaep}, first note that for every $i,j$
\[
\int_0^L |f_{ij}'(z)|\,dz \le \sqrt L \left( \int_0^L |f_i'(z) - f_j'(z)|^2
\right)^{1/2} \le C \| f'\|_{L^2((0,L)} = C.
\]
The image of $f_{ij}$ is therefore a connected curve of length at most $C$.
It follows that $\calL^2(N^\delta_{ij}) \le C \delta$. 
Next, if we define $P_{ij} : (\R^2)^n\to \R^2$ by $P_{ij}(a_1,\ldots, a_n) = a_i-a_j$, then
$\mathcal N^\delta_{ij} = P_{ij}^{-1}(N^\delta_{ij})$, and one easily computes that
the Jacobian of $P_{ij}$ (in the sense of the coarea formula) is
$JP_{ij}=1$. Then the coarea formula implies that
\begin{multline*}
\calL^{2n}(\calB^\delta_{ij}) = 
\int_{\{ \max_i |a_i|\le \delta^{1/3}\}}JP_{ij} \, 
{\bf 1}_{ P_{ij}(a)\in N^\delta_{ij} } d\calL^{2n}(a) \ \\
=
\int_{N^\delta_{ij}} \calH^{2n-2} \left( \{ a : \max_i|a_i|\le \delta^{1/3}, P_{ij}(a)=x\}\right) dx
\le C \delta^{(2n-2)/3} \calL^2(N^\delta_{ij}).
\end{multline*}
Thus $\calL^{2n}(\calB^\delta_{ij}) \le C \delta^{(2n+1)/3}$
for all $i\ne j$. This implies \eqref{notaep} for all sufficiently small $\delta>0$.

\end{proof}

\begin{lemma}\label{RS+BC}
Assume that $(w^z_\ep)\subset H^1(\omega;\C)$ are sequences satisfying
\eqref{formofw}  -- \eqref{qep.lim} for $z\in \{0,L\}$, and that 
$f\in H^1([0,L];(\R^2)^n)$ satisfies
\[
f_i(0) = q^0_i(0),\qquad 
f_i(L) = q^0_{\sigma(i)}(L) \quad\mbox{ for $i=1,\ldots, n$}
\]
for some $\sigma\in S^n$.
Then there is a sequence $(u_\ep)\subset H^1(\Omega;\C)$ such that 
(defining $v_\ep$ as usual by rescaling)
\[
\star J v_\ep \rightarrow \pi \sum_{i=1}^n \Gamma_{f_i} \mbox{ in } \cup_{R>0}F(B(R)\times (0,L)) , \qquad \liminf_{\ep \to 0} G_\ep(u_\ep) \le G_0(f)
\]
and $u_\ep(x,z) = w^z_\ep(x)$ for $x\in \omega$ and $z\in \{0,L\}$.
\end{lemma}

\begin{proof}
Let $f_\ep= h_\ep f^{\delta_\ep}$, where $f^\delta$ is the regularization of $f$ provided by Lemma \ref{L.fep} and $\delta_\ep = \ep^{1/3}$.
We then define
\[
\tilde u_\ep(x,z) := 
\prod_{j=1}^n \left[   e^{i\beta (x, f_{\ep,j}(z)) } \tilde U_\ep(x-  f_{\e,j}(z), z) \right]
\]
where $\beta$ is defined in \eqref{beta.def}, 
and (for $\hat U_\ep$ from \eqref{Uhat.outer}, \eqref{inner.energy}) we set
\[
\tilde U_\ep(x,z): = 
\begin{cases}
\zeta_\ep(x) +  \ep^{-1/6}z(\hat U_\ep(x) - \zeta_\ep(x)) &\mbox{ if } 0\le z \le \ep^{1/6} \\
\hat U_\ep(x) &\mbox{ if } \ep^{1/6}\le z \le L-\ep^{1/6}\\
\zeta_\ep(x)+ \ep^{-1/6}(L -  z) (\hat U_\ep(x) - \zeta_\ep(x)) &\mbox{ if } L- \ep^{1/6} \le z \le L \ .
\end{cases}
\]
(Recall that $\zeta_\e$ is defined in Section \ref{S1.1}.)
The flat-norm convergence of $\star J v_\ep$ follows very much as in the proof of
Proposition \ref{PRSeq}, so we omit the details.

{\bf Step 1}. 
For $\delta^{1/2} = \ep^{1/6} \le z \le L - \ep^{1/6}$, 
Lemma \ref{L.fep} implies that 
\begin{equation}
|f_{\ep,i} - f_{\ep,j}| \ge 2 \sqrt \ep \qquad\mbox{ for sufficiently small }\ep.
\label{goodsep}\end{equation}
The significance of this stems from the fact that
$\sqrt \ep$ is the intermediate length scale that we 
chose (rather arbitrarily) in the construction of $\hat U_\ep$. Once
vortices are separated by larger distances (that is, once \eqref{goodsep} holds), 
all the estimates
in the proofs of Proposition \ref{PRSeq} and Lemma \ref{LFTPoRS}
are uniform, and it follows that
\begin{multline*}
\int_{\ep^{1/6}}^{L-\ep^{1/6}} 
\int_\omega e_\ep(\tilde u_\ep)\,dx\,dz
- n \pi L(\logeps+\gamma) - n(n-1)\pi L |\log h_\ep| - \kappa_n(\Omega)\
\\
 \le 
\int_{\ep^{1/6}}^{L-\ep^{1/6}}
\frac \pi 2 |(f^{\delta_\ep})'|^2 - \pi \sum_{i\ne j}\log|f^{\delta_\ep}_i - f^{\delta_\ep}_j| \, dz
+ o(1)
\end{multline*}
as $\ep \to 0$.
In view of \eqref{fep5}, we therefore only need to show that
\begin{equation}\label{shoulders}
\int_0^ {\ep^{1/6}} 
\int_\omega  e_\ep(\tilde u_\ep) + 
\int_{L-\ep^{1/6}}^L
\int_\omega  e_\ep(\tilde u_\ep) \to 0 .
\end{equation}
We will only consider the first integral, since the estimate of the second is identical.

{\bf Step 2}. Both $\hat U_\ep$ and $\zeta_\ep$ have the form
$\rho_\ep(r) e^{i\theta}$
for $\rho_\ep$ satisfying
\begin{equation}
\rho_\ep(0) = 0, \qquad 0 \le \rho_\ep' \le \frac C \ep, \qquad 1 \ge \rho_\ep(s) \ge 1 - C \frac \ep s \ .
\label{formofmod}\end{equation}
Thus $\tilde U_\ep(\cdot, z)$ also has this form for every $z$.
It follows that $\partial_z \tilde u_\ep$
is a sum of terms of the form already estimated in Lemma \ref{LFTPoRS}, 
together with some new terms of the form
\[
\ep^{-1/6}( U_\ep(\cdot - a) - \zeta_\ep(\cdot - a)) * \mbox{(smooth functions)}
\]
where the smooth functions have modulus less than $1$.
It also follows from \eqref{formofmod}
that $|(\hat U_\ep - \zeta_\ep)(x)| \le \min \{ 1, C \ep / |x| \}$
and hence that for every $a$,
\[
\| \hat U_\ep(\cdot - a) - \zeta_\ep(\cdot - a)\|_{L^2(\omega)}^2 \le
\int_0^{\mbox{\scriptsize diam}(\omega)} \min(1, C \ep/s)^2 s \, ds \le  C \ep^2 \logeps
\]
as long as $0< \ep < 1/2$ say, where $C = C(\omega)$.
Using these facts  and arguing as in the proof of Lemma \ref{LFTPoRS},
one checks that
\[ 
\int_0^{\ep^{1/6}} \int_\omega |\partial_z \tilde u_\ep|^2dx\,dz
\le
C \int_0^{\ep^{1/6}} |(f^{\delta_\ep})'|^2 + o(1)  \  \overset{\eqref{fep4}}= \ o(1)
\]
as $\ep\to 0$.

{\bf Step 3}. We now write
\[
\tilde u_\ep^j := e^{i\beta (x, f_{\ep,j}(z)) } \tilde U_\ep(x-  f_{\e,j}(z), z),
\]
so that $\tilde u_\ep = \prod_{i=1}^n \tilde u_\ep^j$.
Since $\|\tilde u_\ep^j\|_{L^\infty} \le 1$ for every $j$, it is clear that
\[
| \nabla_x \tilde u_\ep|^2 \le (\sum_j |\nabla_x\tilde u_\ep^j|)^2 \le n \sum_j|\nabla_x\tilde u_\ep^j|^2.
\]
Also, for $a,b\in [0,1]$, by rearranging the inequality $0 \le (1-a^2)(1-b^2)$,
we find that $(1-a^2b^2) \le (1-a^2) + (1-b^2)$. By induction, we deduce that
\[
( 1 - |\prod \tilde u_\ep^j|^2)^2 \le \left(\sum (1 - |\tilde u_\ep^j|^2) \right)^2
\le n \sum (1 - |\tilde u_\ep^j|^2)^2.
\]
Therefore
\[
e^{2d}_\ep(\tilde u_\ep) \le n \sum_{j=1}^n  e^{2d} (\tilde u_\ep^j).
\]
We can then appeal to rather standard estimates of $ e^{2d} (\tilde u_\ep^j)$ to conclude that
\[
\int_\omega e^{2d} (\tilde u_\ep^j(x,z))dz \le \pi \logeps + C
\]
for every $z$, so it follows that
\[
\int_0^{\ep^{1/6}} \int_\omega e^{2d}(\tilde u_\ep) dx\,dz
\le  \ep^{1/6} n^2 \sup_{j,z} \int_\omega e^{2d} (\tilde u_\ep^j(x,z))dx \le
\ep^{1/6} n^2( \pi \logeps + C)\ ,
\]
completing the proof of \eqref{shoulders}.
\end{proof}

\begin{proof}[Proof of Theorem \ref{minimizers}]

Let  $u_\ep$ minimize $F_\ep$ 
in 
\[
\mathcal A_\ep :=
\{ u\in H^1(\Omega;\C) : \ u(x, 0) = w_\e^0(x), \ \ u(x,L) = w_\e^L(x) \}
\]
for boundary data as described in \eqref{formofw}.
We want to verify that the sequence $(u_\ep)$ generated in this way satisfies the
hypotheses of part $(a)$ of Theorem \ref{main}.  It is straightforward to check from
\eqref{formofw}--\eqref{qep.lim} that assumptions \eqref{scaling1a} and \eqref{scaling1aa}
are satisfied,
and \eqref{scaling2} follows from Lemma \ref{RS+BC}.
The only remaining hypothesis to check is that
\begin{equation}\label{orhtc}
\| \star Ju_\ep - n\pi \Gamma_0\|_{W^{-1,1}(\Omega)}\to 0 \quad\mbox{ as }\ep \to 0.
\end{equation}
To verify this we argue as in Lemma \ref{labels}. Thus, for $\delta>0$ to be specified below,
we let $\Omega^\delta := \omega \times(-\delta, L+\delta)$, and we extend
each $u_\ep$ to a function on $\Omega^\delta$, still denoted $u_\ep$, such that
$u_\ep(x,z) = w^0_\ep(x)$ for $z<0$, and 
$u_\ep(x,z) = w^L_\ep(x)$ for $z>L$. From \eqref{scaling2} and Theorem \ref{UKC},
we may pass to a subsequence and find an integer multiplicity $1$-current $J$ in $\Omega^\delta$
such that $\partial J = 0$ in $\Omega^\delta$, 
\[
\frac 1\pi \star J u_\ep \rightarrow J\mbox{ in }W^{-1,1}(\Omega^\delta),
\qquad\qquad M_\Omega(J) < M_{\Omega^\delta}(J)\le  n L + 2\delta M.
\]
We must show that the restriction of $J$ to $\Omega$ is $n\Gamma_0$.
It is for this that we will need our assumption that
\[
L < 2 R\qquad\mbox{ where }R=\dist(0,\partial \omega).
\]

Due to our assumptions on $w_\ep^{0,L}$, 
it is straightforward to check
that the restriction of $J$
to $\omega\times (-\delta, 0)$, for example,
consists of $n$ copies of the segment $\{0\}\times(-\delta,0)$
(with the correct, ``upward" orientation, henceforth not mentioned),
with a similar statement in $\omega \times(L,L+\delta)$.



We may write $J = \sum_{i\in I} T_{\gamma_i}$,
where each $\gamma_i$ is a Lipschitz curve without boundary
in $\Omega^\delta$. Then there must be subsets $I^0, I^1$ of $I$, 
both with cardinality $n$, such that
\begin{align*}
&\mbox{ for $i\in I^0$, }T_{\gamma_i} =  \{0\}\times (-\delta, 0)\qquad 
\mbox{ in } \omega\times (-\delta, 0), \\
&\mbox{ for $i\in I^1$, }T_{\gamma_i} = \{0\}\times (L,L+\delta)\quad \mbox{ in } \omega\times (L, L+\delta). 
\end{align*}
Clearly, if $i\in I^0\cap I^1$, then $\gamma_i$
connects a copy of $\{0\}\times (-\delta, 0)$ to a copy of $\{0\}\times (L,L+\delta)$,
and must have length at least $L+2\delta$.

On the other hand, if $i\in I^0\setminus I^1$, then $\gamma_i$ must connect a copy
of $\{0\}\times (-\delta, 0)$ to $\partial \omega \times[0,L]$, and must have length at
least $\delta +R$. The same applies to $i\in I^1\setminus I^0$.
Thus, if $n_0 := \# (I^0\cap I^1)$, then
\begin{align*}
nL+ 2M\delta \ge M_{\Omega^\delta}(J) &
\ge 
\sum_{i\in (I^0\cup I^1)}M_{\Omega^\delta}(T_{\gamma_i} )\\
& \ge 
n(L+2\delta) + (n-n_0)(2R-L) .
\end{align*}
Since $\delta$ may be chosen arbitrarily small, we may deduce that 
$n=n_0$, and hence that $I^0=I^1$. Similar considerations show that
$I = I^0$ -- that is, there are no indecomposable pieces other than the $n$
curves that connect $\{0\}\times (-\delta, 0)$ to $\{0\}\times(L,L+\delta)$.
Finally, the same argument shows that none of these $n$ curves can have length greater than
$L+2\delta$. Thus every curve coincides with $\{0\}\times(0,L) = \Gamma_0$ in $\Omega$.
This says that the restriction of $J$ to $\Omega$ is exactly $n\Gamma_0$, 
which completes the proof of \eqref{orhtc}.
 
Having verified \eqref{orhtc}, we may apply part (a) of
Theorem \ref{main} to conclude that
$\star \frac 1\pi Jv_\ep$ is precompact in $W^{-1,1}(B(R)\times (0,L))$ for every $R>0$,
and that every limit of a convergent subsequence has the form
$\pi \sum_{i=1}^n \Gamma_{f^*_i}$, where
\[
f^*\in \mathcal A_0 := 
\{ f\in H^1([0,L]; (\R^2)^n ) \ : \exists \sigma\in S^n, \ f_i(0) = q^0_i(0),  \ f_i(L) = q^0_{\sigma(i)}(L) 
  \}
\]
for $(q^0_i(z))$, $z= 0,L$ appearing in 
\eqref{formofw}, \eqref{qep.lim}; and that
\[
G_0(f^*) \le \liminf_{\ep\to 0}G_\ep(u_\ep) 
\]
along the subsequence.
Since $u_\ep$ minimizes $F_\ep$ in $\mathcal A_\ep$, it follows
from Lemma \ref{RS+BC} that 
\[
\liminf_{\ep\to 0} G_\ep(u_\ep) \le \inf_{f\in \mathcal A_0} G_0(f).
\]
Therefore $f^*$ minimizes $G_0(\cdot)$ in $\mathcal A_0$, as was to be shown.
\end{proof}

\section{The proof of Theorem \ref{localmin}}\label{sectionforlocalmin}

We first need a Lemma that relates local minimizers with respect to two
different topologies.

\begin{lemma}\label{L19}
Assume that $f^*\in H^1((0,L);(\R^2)^n)$ is a strict local minimizer of
$G_0$ in $\mathcal A_0$ and that $f^*_i(z)\ne f^*_j(z)$ for all  $z\in (0,L)$ and 
$i\ne j$.

Then there exists $\delta^*, R_1>0$ such that 
\[
\mbox{ if }f\in \mathcal A_{0}\ \ \  \mbox{ and }  \ \ \ 
0< \|  \pi \sum_{i=1}^n ( \Gamma_{f_i} - \Gamma_{f^*_i} )\|_{F(B(R)\times (0,L))} < \delta^*
\ \ \mbox{ for some }R\ge R_1,
\]
then $G_0(f) > G_0(f^*)$. 
\end{lemma}

\begin{proof}
Consider any sequence  $(f^k)\in \mathcal A_0$ such that 
\begin{equation}\label{fk.choice}
G_0(f^k) \le G_0(f^*),\qquad \quad \|  \sum_{i=1}^n ( \Gamma_{f^k_i}  -  \Gamma_{f^*_i})\|_{F(B(R)\times (0,L))} \to 0 \, 
\end{equation}
for some $R\ge R_1$, where $R_1$ will be fixed below.
We will show that $f^k\to f^*$ in $H^1((0,L);(\R^2)^n)$ as $k\to \infty$, after 
a possible relabelling.
Since $f^*$ is a local minimizer, this will prove that $f^k = f^*$ for all sufficiently 
large $k$. Because the sequence $(f^k)$ is arbitrary, this will establish the lemma.

First, by a $1$-dimensional Sobolev embedding theorem,
\[
\| f^k \|^2_{L^\infty} \le  \| f^k(0)\|^2 +  L \| (f^k)'\|_{L^2}^2 = C + C \| (f^k)'\|_{L^2}^2.
\]
where the constants depend on the boundary data for $\mathcal A_0$ and on $L$.
Also, it is clear that
\[
-\pi \int_0^L \sum_{i\ne j} \log |f^k_i - f^k_j| \ge -\pi n(n-1) L \log( 2 \|f^k\|_{L^\infty}).
\]
It follows that 
\[
\| f^k\|_{L^\infty}^2 - C \log (\|f^k\|_{L^\infty}) - C \le CG_0(f^k) \le C G_0(f^*).
\]
Since $x^2- C\log x\to +\infty$ as $x\to +\infty$, there exists some constant $C$ such that
\begin{equation}
 \int_0^L \sum_{i\ne j} \log |f^k_i - f^k_j| \le C, \qquad 
\| f^k\|_\infty + \| (f^k)'\|_{L^2} \le C  
\qquad\mbox{ for all $k$}.
 \label{calAstar}\end{equation}
We now fix $R_1 := \sup_k \| f^k\|_\infty +1$.

Now for $\delta \in (0,L/2)$, define $r = r(\delta) $ by
\[
r(\delta) :=
 \min \left( \ \{ 1\}  \cup   \left\{ \frac 14 | f^*_i(z) - f^*_j(z)| :  i\ne j,  \delta  \le z \le L - \delta \right\} \, \right)  .
\]
We claim that for every $\delta$ as above,
if $k$ is large enough, then for every $i$,
\begin{equation}
\min_j | f^*_i(z) - f^k_j(z)| <  r(\delta) \mbox{ for }  \delta<z<L-\delta.
\label{tube}\end{equation}
Suppose toward a contradiction that \eqref{tube} fails for some $i$ and $z$. 
Since both $f^*$ and $f^k$
are $C^{0,\frac 12}$, with modulus of continuity depending on $f^*$
but independent of $k$, see \eqref{calAstar}, it follows that there is
an interval $I^k$ such that 
\[
\min_j |f^*_i(z) - f^k_j(z)| \ge \frac 12 r(\delta)\mbox{ for all }z\in I^k,\qquad |I^k|\ge c = c(f^*).
\]
For such $z$, it follows that 
\begin{multline*}
\| \sum_{j=1}^n ( \delta_{f^*_j(z)}   -  \delta_{f^k_j(z)}) \|_{F(B(R))} \\
 \ge \int_{B(R)} ( \frac 12 r(\delta) - |f^*_i(z) - x|)^+ \Big(\sum_{j=1}^n ( \delta_{f^*_j(z)}   -  \delta_{f^k_j(z)}|)\Big)(dx)
=  \frac 12 r(\delta). 
\end{multline*}
Hence for all $k$,
\[
\int_0^L \| \sum_{j=1}^n ( \delta_{f^*_j(z)}   -  \delta_{f^k_j(z)}) \|_{F(B(R))}  dz \ge \frac12 \,  |I^k|  \, r(\delta) > c > 0.
\]
On the other hand, we may use \eqref{slice.est} to estimate
\begin{align*}
\int_0^L \| \sum_{i=1}^n ( \delta_{f^*_i(z)}   -  \delta_{f^k_i(z)}) \|_{F(B(R))}\, dz
&=
\int_0^L  \|\langle  \sum_{i=1}^n ( \Gamma_{f^*_i}  -  \Gamma_{f^k_i}), \zeta , z\rangle  \|_{F(B(R))}\, dz
\\
&\le \|  \sum_{i=1}^n ( \Gamma_{f^*_i}  -  \Gamma_{f^k_i})\|_{F(B(R)\times (0,L))}  \to 0,
\end{align*}
a contradiction. Hence \eqref{tube} holds.

Since the balls $B(f^*_i(z),  r(\delta))$ are disjoint for  $\delta < z <L-\delta$,
by definition of $r(\delta)$, it follows from \eqref{tube} that for every $z$ in this range,
each of these balls contains exactly one point $f^k_j(z)$. We may relabel the $(f^k_j)$
such that at height $z=L/2$ for example, $|f^*_i(L/2) - f^k_i(L/2)|< r(\delta)$ for all $i$.
Then \eqref{tube} and the continuity of $f^*, f^k$ imply that 
$|f^*_i(z) - f^k_i(z)|< r(\delta)$ for all $i$ and all $z\in (\delta, L-\delta)$.

Now  we let $k\to \infty$ and, using \eqref{calAstar}, 
extract a subsequence such that $f^k$ converges weakly in $H^1$, and thus uniformly,
to a limit $f^\infty\in \mathcal A_0$. It follows from the above that
$|f^*_i(z) - f^\infty_i(z)| \le r(\delta)$ for all $i$ and all $z\in (\delta, L-\delta)$.
Since $\delta$ is arbitrary, we conclude that $f^\infty = f^*$.
Thus in fact $f^k\rightharpoonup f^*$ in $H^1$, without passing to a subsequence.
Then the choice of $(f^k)$ and  standard lower semicontinuity arguments imply that 
\[
G_0(f^*) \ge \liminf G_0(f^k) \ge G_0(f^*).
\]
It follows in particular that $\int |(f^k)'|^2dz \to \int|(f^*)'|^2 dz$. This allows
us to improve weak $H^1$ convergence to strong $H^1$ convergence, completing the proof.

\end{proof}

The proof of Theorem \ref{localmin}  now follows standard arguments.

\begin{proof}[Proof of Theorem \ref{localmin}]
Fix $0<\delta^1 <\delta^*$, for $\delta^*$ from Lemma \ref{L19}, and let $u_\ep$ minimize $F_\ep$
in
\[
\bar {\mathcal A}_{\ep,\delta^1} := \{ u\in \mathcal A_\ep \ : \ \| \star J u - \pi \sum_{i=1}^n\Gamma_{h_\ep f^*_i}\|_{F(\Omega)}  \le  h_\ep\delta^1 \}  \ .
\]
Existence of a minimizer is rather standard; see for example \cite{MSZ}, Theorem 4.2
for a very similar argument.
We claim that if $\ep$ is small enough
\begin{equation}\label{uepinAep}
 \| \star J u_\ep  - \pi \sum_{i=1}^n\Gamma_{h_\ep f^*_i}\|_{F(\Omega)}  <  h_\ep\delta^1
\qquad \mbox{ or equivalently}, \ \ u_\ep \in  \mathcal A_{\ep,\delta^1}.
\end{equation}
Toward this end, consider a hypothetical sequence in $\bar {\mathcal A}_{\ep,\delta^1} $
for which  \eqref{uepinAep} fails, so that the constraint holds with equality. 
Then Lemma \ref{RS+BC} and the definition of $\mathcal A_\ep$ 
imply that  $(u_\ep)$ satisfies the hypotheses of Theorem \ref{main}.
Thus, defining as usual $v_\ep( x,z)= u_\ep(h_\ep x, z)$, we conclude that
\[
\star Jv_\ep\to \pi \sum_{i=1}^n\Gamma_{f_i} \qquad \mbox{ in }F(B(R)\times (0,L))
\mbox{ for every }R>0.
\] 
for some $f\in H^1((0,L);(\R^2)^n)$. Moreover, again appealing to Lemma \ref{RS+BC}
(the recovery sequence with correct boundary conditions) and Theorem \ref{main},
we find that
\[
G_0(f) \le G_0(f^*).
\]
Also, by rescaling our assumption that 
$\| \star J u_\ep  - \pi \sum_{i=1}^n\Gamma_{h_\ep f^*_i}\|_{F(\Omega)}  = h_\ep \delta^1$
and taking limits, we conclude that
$\|   \pi \sum_{i=1}^n ( \Gamma_{ f_i}- \Gamma_{ f^*_i} ) \|_{F(B(R)\times (0,L))}  = \delta^1$
for every $R>0$.

This is impossible, in view of Lemma \ref{L19}, so \eqref{uepinAep} must be true.

Then it is well-known (see for example \cite{MSZ}) and not hard to check that $\mathcal A_{\ep,\delta^1}$ is an open set in $H^1$, and so \eqref{uepinAep} implies that $u_\ep$ is a local minimizer.
By arguing in this way for a sequence $\delta^k\searrow 0$, 
and employing  a diagonal argument, we may generate 
a sequence $(u_\ep)$ such that the rescaled Jacobians $(Jv_\ep)$
converge  as required to $\pi \sum \Gamma_{f^*_i}$.
\end{proof}

\appendix
\section{}\label{App.A}

We finally present the proof of Lemma \ref{L3b},
which establishes certain properties of $u_\ep(\cdot, z)$ for
$z\in \calG^\ep_2$.
The main point of the proof is contained in the following

\begin{lemma}
There exist positive numbers $\theta, a,b$, 
depending on $n$ and $\omega$, such that $b<a$, and the following holds:

Assume that  $w\in H^1(\omega;\C)$ satisfies
\begin{equation}\label{L3b.h1bis}
\int_\omega e^{2d}(w)(x) \,dx \le  \pi(n+\theta)\logeps
\end{equation}
and 
\begin{equation}\label{L3b.h2bis}
\int_\omega \phi(x) J_x w(x)\, dx \ge n\pi -1 \quad
\mbox{for some $\phi \in W^{1,\infty}_0(B(r^*))$ with $Lip(\phi) \le 4/r^*$,}
\end{equation}
where $r^* :=  \min\{1,  \dist(0,\partial \omega)\}$.

Then if $\e$ is sufficiently small, there exists $p = (p_1,\ldots, p_n)$ such that
\begin{equation}
\| J_xw - \pi \sum \delta_{p_i} \|_{W^{-1,1}(\omega)} \le  \e^{a}
\label{L3bbc1}\end{equation}
and
\begin{equation}
\dist(p_i, \partial \omega)\ge \frac {r^*} {8}\mbox{ for all }i, \qquad
|p_i - p_j| \ge \e^{b} \quad\mbox{ for all $i\ne j$}
\label{L3bbc2}\end{equation}
\label{L3b.bis}\end{lemma}

To obtain Lemma \ref{L3b} from the above, note that
if $z$ satisfies the hypotheses
of Lemma \ref{L3b}, then $w = u_\ep(\cdot, z)$
satisfies \eqref{L3b.h1bis}, \eqref{L3b.h2bis}. Indeed, \eqref{L3b.h1bis}
is exactly \eqref{L3b.h2}, and \eqref{L3b.h2bis} is a consequence of 
Lemma \ref{GvsFlat}, see \eqref{GvsF2}. Thus
\eqref{L3b.c1} \eqref{L3b.c2}
follow directly from \eqref{L3bbc1}, \eqref{L3bbc2}.
The final conclusion \eqref{L3a.sharper}
is then an immediate consequence\footnote{
If one wants to check this, note from \eqref{L3b.c1}, \eqref{L3b.c2}  that $\rho_\alpha \ge \frac 12\e^b$,
and that one may take $s_\e = \e^a$, where the notation $\rho_\alpha$ and $s_\e$ appears in 
\cite{JSp2}. From this one easily checks that hypotheses
in \cite{JSp2} relating $s_\ep$ and $\rho_\alpha$ are satisfied here.
Note also that
$\Sigma^\ep_\Omega(u;\alpha, d)$
appearing in \cite{JSp2} is defined to be
$\int_{\Omega} e^{2d}_\ep(u))dx \ - \  n(\pi\vert \log\e \vert +\gamma) +  W_{\omega}(\alpha_1,\ldots,\alpha_n)
$ when $d = (1,\ldots, 1)$, which is the relevant case here.
So conclusion \eqref{L3a.sharper} is exactly a lower
bound for $\Sigma^\ep_\omega(u(\cdot, z), p^\ep, d)$ with $d$ as above.}
 of 
Theorem 2 in \cite{JSp2}, whose
hypotheses are implied by  
\eqref{L3b.c1}, \eqref{L3b.c2}.

Throughout the proof of Lemma \ref{L3b.bis}, we will assume
that $w$ is smooth. The general case follows from this by a standard
mollification argument.

Our proof relies on a vortex ball construction, as introduced by \cite{Jlower} and \cite{S}. 
We recall some ingredients that we will need below.
Our presentation mostly follows that of \cite{Jlower} and \cite{JSo}, which can also
be used as sources, adapted to our needs, for background on topics such as the degree
$\deg(w;\partial O)$.

We will use the notation
\begin{align}
S  &:= \{x\in \omega : |w(x)| \le \frac 12 \}, \label{S.def}
\\
S_E &:= \cup \{ \mbox{components $S_i$ of $S$}  \ : \ \deg(w;\partial S_i)\ne 0 \}.\label{SE.def}
\\
S_E^\ep &:= \cup \{ \mbox{components $S_i$ of $S$}  \ : \ \deg(w;\partial S_i)\ne 0,\ \dist(S_i,\partial\omega)\ge \ep \}.\label{SEep.def}
\end{align}

If $O$ is an open subset of $\omega$ such that $\partial O \cap S_E=\emptyset$, then
\begin{equation}\label{dg.def}
\dg(w; \partial O) := \sum \big\{ \deg(u;\partial S_i)  : \mbox{components $S_i$ of $S_E$ such that }S_i\subset O \big\}.
\end{equation}
We also define
\begin{align}
\label{lambdastandard.def}\lambda_\ep(r,d) 
&:= \min_{m\in [0,1]} \left(\frac {m^2 d^2  \pi}r + \frac 1{c\e}(1-m)^2\right).
\\
\Lambda_\ep(s)
&:= \ \int_0^s ( \lambda_\e(r;1) \wedge \frac 1{C\e})\, dr.
\nonumber
\end{align}

\begin{lemma}\label{eondb}
Assume that $B(r,x)\subset \omega$ for some $r\ge \ep$, and that $|\deg(w;\partial B(r,x))| = d>0$.
Then
\begin{equation}
\int_{\partial B(r,x)}e^{2d}_\e(w) \, dx  \  \ge \    \lambda_\ep(r,d)\  \ \ge \  \lambda_\ep(\frac rd, 1).
\label{gllbd.1}\end{equation}
\end{lemma}

For a proof, see for example \cite{Jlower} Theorem 2.1,
or consult  Lemma \ref{L8} above for a very similar argument.

By integrating  \eqref{gllbd.1}, one obtains lower
bounds for the energy $e^{2d}_\e(w)$ on 
an annulus on which one has some information about the degree.
These bounds are naturally
expressed in terms of $\Lambda_\ep$. The main point of the
vortex ball construction is to assemble these estimates in a clever way. 
These arguments lead for example to the following:

\begin{lemma}\label{L.vballs}
For all $\sigma \ge  r_0 := C \e \int_\omega e^{2d}_\e(w)dx$,
there exists a collection $\calB(\sigma) = \{ B^\sigma_k\}_{k=1}^{k(\sigma)}$ 
of disjoint balls such that 
\begin{align}
S_E^\ep&\subset \cup_k B^\sigma_k
\label{vballs1.bis}\\
\int_{B^\sigma_k\cap \om} e^{2d}_\e(w)\, dx  &\ge \frac {r_k^\sigma}\sigma \Lambda_\ep (\sigma),\qquad\quad\hspace{4em}
\mbox{ for }
\, r^\sigma_k := \operatorname{radius}(B^\sigma_k)
\label{vballs2.bis}\\
r_k^\sigma &\ge \sigma |d^\sigma_k|
\qquad\mbox{ if }B^\sigma_k\subset \omega , \qquad
\mbox{ for }
d^\sigma_k := \dg(w, \partial B^\sigma_k)  \,.
\label{vballs3.bis}
\end{align}
Moreover, if $x^\sigma_k$ denotes the
center of $B^\sigma_k$, then
\begin{equation}\label{balls.flatest}
\| Jw - \pi \sum d_k\delta_{x^\sigma_k} \|_{W^{-1,1}(\omega)} \le 
C ( r_0 + \sum_k r_k^\sigma) \int_\omega e^{2d}_\e(w)\,dx.
\end{equation}
Finally, $\Lambda_\ep(\sigma)\ge \pi \log \frac\sigma \ep - C$ for all $\sigma$.
\end{lemma}

\begin{proof}
In Proposition 6.4 in \cite{JSo} it is proved that a collection of
balls satisfying \eqref{vballs1.bis} -- \eqref{vballs3.bis} exists for every $\sigma$ larger than some $r_0$.
The fact that one may take $r_0= C \e \int_\omega e^{2d}_\e(w)dx$
follows  from the proofs of Propositions 6.2 and 6.4,  \cite{JSo}. 

To prove \eqref{balls.flatest}, one modifies  $\{ B^\sigma_k\}$ to obtain
a collection of balls that covers all of $S$, and
whose radii sum to at most $r_0+ \sum r^\sigma_k$.
This relies on a lemma due to \cite{S}, which shows that
$S$ may be covered by a collection of disjoint balls
whose radii sum to at most $r_0$.
Then, \eqref{balls.flatest} follows from standard arguments, 
see for example \cite[Theorem 6.1]{SandSerfbook}.
\end{proof}




With the above result as our starting point, we can now present the

\begin{proof}[proof of Lemma \ref{L3b.bis} ]
We fix positive numbers  $\theta, a, b$ such that
\[
\frac {n+\theta}{1-2a} < n+1, \qquad
\theta \le  b < a.
\]
For example, if $2b = a = \frac 1{3(n+1)}$ then both inequalities may be satisfied by
a sufficiently small positive $\theta$.

{\bf Step 1}.
First, consider the collection of balls $\calB(\sigma)$ from Lemma \ref{L.vballs},
with $\sigma = \ep^{2a}$.
We will write\ $\{ x^\sigma_i\}$ for the center of 
$B^\sigma_i$. 
Then 
\eqref{vballs2.bis} implies that 
\begin{equation}
\sum_k \int_{B^\sigma_k\cap \omega} e^{2d}_\e(w)dx
\ge
\sum_k \frac {r^\sigma_k}\sigma \Lambda_\ep(\sigma) \ge 
\sum 
\frac{r^\sigma_k}\sigma 
(1-2a)\left(   \pi\logeps - C \right) .
\label{L3b.pf00}\end{equation}
Then from  \eqref{L3b.h1bis} and  the choice of $a$,  we find that
for small enough $\e$,
\begin{equation}\label{L3b.pf0}
\sum_k \frac{r^\sigma_k}\sigma
\le  \frac{(n+\theta) }{(1- 2a ) }\frac{\logeps}{\logeps-C} < n+1\qquad\mbox{ and thus }\quad
\sum_k |d^\sigma_k|  \le n.
\end{equation}
On the other hand,  from \eqref{balls.flatest}, \eqref{L3b.pf0} and \eqref{L3b.h1bis}
we have
\begin{equation}
\| J_x w - \pi \sum d_i^\sigma \delta_{x^\sigma_i} \|_{W^{-1,1}(\omega)} \le   C\sigma(n+1) \int_\omega e^{2d}_\e(w) \le \e^a
\label{L3b.pf1}\end{equation}
if $\e$ is small enough.
This and  \eqref{L3b.h2bis} imply that 
\[
n\pi - 1 \le  \pi\sum_k d_k^\sigma \phi(x^\sigma_k) + 4 \e^a/r^* \ .
\]
If $\e$ is small enough, then by comparing this with \eqref{L3b.pf0}
and recalling that $0\le \phi \le 1$,
we see that
\begin{equation}
\mbox{$d^\sigma_k\ge 0$ for all $k$, }\qquad\qquad
\sum_k d^\sigma_k = n
\label{L3b.pf2}\end{equation}
and
\[ 
\mbox{$\phi(x^\sigma_k) \ge 1 - \frac 1\pi > \frac 12$ for all $k$ such that 
$d^\sigma_k >0$.}
\] 
Since $\mbox{Lip}(\phi) \le 4/r^*$ and 
supp$(\phi)\subset B(1)$, it follows that 
\begin{equation}
|x^\sigma_i| \le  \frac 78 r^* \qquad\mbox{ for $k$ such that }d^\sigma_k>0.
\label{L3b.pf4}\end{equation}

We also remark that \eqref{vballs3.bis}, \eqref{L3b.pf0}, and \eqref{L3b.pf2}
imply that
\[ 
n \sigma \le \sum r^\sigma_k < (n+1)\sigma \quad\mbox{ for small }\e
\] 
and thus
\begin{equation}\label{L3b.pf5}
d^\sigma_k \sigma \le  r^\sigma_k <  (d_k^\sigma +1)\sigma \quad\mbox{ for all }k,
\qquad\qquad
\sum_{d^\sigma_k=0} r^\sigma_k < \sigma.
\end{equation}
In particular, $\max_k r^\sigma_k \le (n+1)\sigma$.

To complete the proof of the lemma, it suffices to show that
\begin{equation}\label{L3b.mainclaim}
d^\sigma_k \le 1 \mbox{ for all $k$}, \qquad\mbox{ and } \quad
|x^\sigma_k - x^\sigma_\ell| \ge \e^b \mbox
{ if }d^\sigma_k, d^\sigma_\ell \ne 0\mbox{ and }k\ne \ell.
\end{equation}
Indeed, once we know \eqref{L3b.mainclaim},
then it follows from \eqref{L3b.pf2}
that there are exactly $n$ points $x^\sigma_k$
for which $d^\sigma_k$ is nonzero,
and that $d^\sigma_k = 1$ for all of these.
We take $\{ p_1,\ldots, p_n\}$ to be these
points. Then \eqref{L3b.pf4}, \eqref{L3b.mainclaim}
imply that \eqref{L3bbc2} holds, and 
\eqref{L3b.pf1} reduces to \eqref{L3bbc1}.

\medskip

{\bf Step 2}.
To start the proof of \eqref{L3b.mainclaim}, let
\[
\wt \calB^0 := \{ B^\sigma_k \in \calB(\sigma) \ : d^\sigma_k \ne 0\}.
\]
We claim that there is a collection 
$\widetilde \calB^1 = \{ \wt B^1_k\}$ of (at most $n$) closed pairwise disjoint balls
such that 
\begin{equation}
\bigcup_{\widetilde \calB^0} B^\sigma_k \subset \bigcup_{\wt \calB^1}\wt B^1_k
\label{sb1}\end{equation}
and every ball in $\wt \calB^1$ has the same radius $\wt r^1$, with
\begin{equation}
\wt r^1 \le  C(n)\sigma.
\label{sb2}\end{equation}
Such a collection can be found as follows:
\begin{itemize}
\item first replace every $B^\sigma_k\in \wt\calB^0$
by a concentric ball of radius  $\max_{\wt \calB^0} r^\sigma_k \le (n+1)\sigma$,
\item enclose intersecting balls in larger balls, without increasing the sum of the
radii, to obtain a new pairwise disjoint collection, with fewer balls.
\item repeat: increase the radii of the remaining balls, as necessary, until they are
the same size, then combine balls that intersect. After finitely many steps (at most
$n-1$ mergings) this produces a collection satisfying \eqref{sb1}, \eqref{sb2}
\end{itemize}
Let
\[
\wt R^1 :=  \sup 
 \Big\{ R \in ( \wt r^1, \frac 1{8} ) \ : \ \{ B(R, \wt x^1_i)\}\mbox{  are pairwise disjoint } \Big\}
\]
where $\{ \wt x^1_i \}$ denotes the centers  of  the balls in $\calB^1$.
We now proceed inductively, using the same procedure to find collections $\wt B^j = \{\wt B^j_i \}$ of (at most $n+1-j$)
balls such that
for $j\ge 2$,
\begin{equation}
\bigcup_{\widetilde \calB^{j-1}} B(\wt R^{j-1}, \wt x^{j-1}_k )\subset \bigcup_{\wt \calB^j}\wt B^j_k
\label{sbj1}\end{equation}
and all balls are closed and pairwise disjoint, with the same radius
\begin{equation}
\wt r^j \le  C(n) \wt R^{j-1}
\label{sbj2}\end{equation}
and where
\[
\wt R^j :=  \sup 
 \Big\{ R \in ( \wt r^j, \frac 1{8} ) \ : \ \{ B(R, \wt x^j_i)\}\mbox{  are pairwise disjoint } \Big\}
\]

Let $J$ denote the first $j$ for which either $\wt R^j = \frac {r^*}{8}$ or $\wt r^{j+1} \ge \frac {r^*}{8}$.
With each step the number of balls decreases, and
if there is only one ball left, it can expand unimpeded, so it is clear that $J\le n$.
It follows from \eqref{L3b.pf4} that the interiors of all the balls are contained in $\omega$.
Note also that
\begin{equation}
\wt R^J \ge \frac 1{8C(n)}.
\label{L3b.pf5b}\end{equation}
This is clear if $\wt R^J = \frac {r^*}8$, and if $\wt r^{J+1} \ge \frac {r^*}8$ then it follows from 
\eqref{sbj2}.

To prove \eqref{L3b.mainclaim}, it now suffices to show that 
\begin{equation}
\mbox{ $\wt \calB^1$ consists of $n$ balls, all of degree $1$} ,\qquad \mbox{ and } \ \wt R^1 > 
\e^b/2.
\label{L3b.pf6}\end{equation}
{\bf Step 3}.
We now write $\wt A^j_k := B(\wt R^j, \wt x^j_k) \setminus B(\wt r_j, \wt x^j_k)$,
and we estimate the energy contained in these annuli.

Let us write $\wt d^j_k := \dg(u ; \partial \wt B^j_k) $, and note that 
\begin{equation}\label{djk.facts}
\mbox{ $\wt d^j_k>0$ for 
all $j,k$, \qquad 
$\sum_k \wt d^j_k = n$ for every $j$,
\qquad 
$\max_j \wt d^j_k \ge 2$ \ for  \ $j\ge 2$. }
\end{equation}
For every $j$ and $k$ we deduce from \eqref{gllbd.1} that
\begin{align}
\int_{\wt A^j_k} e^{2d}_\e(w) dx
&=
\int_{\wt r^j}^{\wt R^j} \int_{\partial B(r, \wt x^j_k)} e^{2d}_\e(w) d\calH^1 \ dr
\nonumber \\
&\ge
\int_{\wt r^j}^{\wt R^j} \int_{\partial B(r, \wt x^j_k)}\lambda_\e(r,\wt d^j_k) 
\,  {\bf 1}_{ r \not \in A^j_k} \ dr \ ,
\label{L3b.pf7}\end{align}
for
\[
A^j_k := \{ r\in (\wt r^j, \wt R^j ) : \dg(u, \partial B(r,\wt x^j_k) )\ne \wt d^j_k \} .
\]

In general, since $r\mapsto \lambda_\ep(r,d)$ is  a nonincreasing function for every $d$,
if $A$ is a measurable subset of an interval $(a,b)$, then
\[
\int_a^b \lambda_\ep(r,d) {\bf 1}_{r\not \in A} \, dr \ge \int_{a+|A|}^b \lambda_\ep(r,d)\, dr.
\]
Thus we would like to estimate the measure of $A^j_k$.
Toward this end, let 
\[
Z := \cup_{ \{ B^\sigma_k\in \calB(\sigma ) \, : \, d^\sigma_k = 0 \} } B^\sigma_k
\]
and note that if $\partial B(r,\wt x^j_k)  \cap Z = \emptyset$, then
$\dg(u, \partial B(r,\wt x^j_k))$ is well-defined and equals $\wt d^j_k$, 
as a consequence of the definition of $\dg$ together with \eqref{vballs1.bis}, \eqref{sb1}, 
\eqref{sbj1},  and the definition of $Z$.
So 
\[
A^j_k \subset \{ r\in (\wt r^j, \wt R^j ) : \partial B(r,\wt x^j_k) \cap Z \ne \emptyset \} .
\]
Since $Z$ is a union of balls whose radii sum to at most $\sigma$, see \eqref{L3b.pf5},
we conclude that $|A^j_k| \le 2\sigma$.
Thus for every $j,k$,
\begin{equation}\label{rjRj}
\int_{\wt A^j_k} e^{2d}_\e(w) dx
\ge \int_{\wt r^j_*}^{\wt R^j}\lambda_\ep(r, \tilde d^j_k) \, dr,
\qquad \wt r^j_* :=\min \{ \wt r^j+2\sigma, \wt R^j\} \overset{\eqref{sbj2}}\le C(n) \wt R^{j-1}
\end{equation}
Also, it is straightforward to check from the definition of $\lambda_\ep$ that
if $r\ge \e^a$ and $d \le n$, then
\[
\lambda_\e(r,d) \ge \frac{\pi d^2}r( 1 - C(n) \e^{1-a})
\]
Thus
\[
\sum_{j=1}^J \sum_k \int_{\wt A^j_k} e^{2d}_\e(w) \, dx
\ge (1 - C \e^{1-a}) \sum_j \pi \log(\frac {\wt R^j}{\wt r^j_*}) \left(\sum_k (d^j_k)^2 \right)
\]

{\bf Step 4}. We wish to show that neither of the conditions 
appearing in \eqref{L3b.mainclaim} can be violated.
We thus  consider two cases. 

{\bf Case 1: $d^1_k >1$ for some $k$.}
Then it follows from \eqref{djk.facts} that 
$\sum_k (d^j_k)^2 > n+2$ for all $j$, and hence
that
\begin{align*}
\sum_{j=1}^J \sum_k \int_{\wt A^j_k} e^{2d}_\e(w) \, dx
&\ge 
\pi(n +2) (1 - C \e^{1-a})  \sum_{j=1}^J \log(\frac {\wt R^j}{\wt r^j_*})
\\
&=
\pi (n+2)(1-C\e^{1-a}) 
\left[ \log (\frac{\wt R^J}{\wt r^1_*})
+ \sum_{j=2}^{J} \log (\frac{\wt R^{j-1}}{\wt r^{j}_*}) 
\right]
\end{align*}
But it follows from \eqref{rjRj}, 
\eqref{L3b.pf5b} and \eqref{sb2}
\[
\frac {\wt R^{j-1}}{\wt r^j_*} \ge \frac 1{C} , \qquad \frac {\wt R^J}{\wt r^1_*} \ge \frac 1{C \sigma}
= \frac 1{C\e^{2a}}
\]
for constants depending on $n$. It follows that
\[
\sum_{j=1}^J \sum_k \int_{\wt A^j_k} e^{2d}_\e(w) \, dx
\ge  \pi(n+2) 2a \logeps - C.
\]
Combining this with \eqref{L3b.pf00} we find that
\[
\int_\omega e^{2d}_\e(w) \ge  \pi(n + 4 a) \logeps - C,
\]
contradicting \eqref{L3b.h1bis} when $\ep$ is small enough.

{\bf Case 2: $d^1_k = 1$ for all $k$, but $\wt R^1 \le  \e^b/2$. }
This implies that $\wt r^2_* \le C e^b$. Then essentially the same arguments as
above show that
\begin{align*}
\sum_{j=1}^J\sum_k \int_{\wt A^j_k} e^{2d}_\ep(w) \,dx \ge
(1-C\e^{1-a}) \left[  n \pi \log(\frac {\wt R^1}{\wt r^1_*}) +
(n+2)\pi \sum_{j=2}^J \log(\frac {\wt R^j}{\wt r^j_*})
\right].
\end{align*}
Continuing to follow the previous case, from this one deduces that
\[
\int_\omega e^{2d}_\e(w) \ge  \pi(n+2b) \logeps - C,
\]
again contradicting \eqref{L3b.h1bis}.
This verifies \eqref{L3b.pf6}
and completes the proof of the lemma.
\end{proof}



\begin{thebibliography}{SCETC}



\bibitem{Aft-Riv}
A. Aftalion and T. Rivi\`ere, ``Vortex energy and vortex bending for a rotating Bose-Einstein condensate, {\it Phys Rev A.}, 64 (2001) 043611.

\bibitem{abo} 
G. Alberti, S. Baldo\ and\ G. Orlandi, ``Variational convergence for functionals of Ginzburg-Landau type'', {\it Indiana Univ. Math. J.}, 54 (2005), no.~5, 1411--1472.


\bibitem{BFT}  V. Barutello, D. Ferrario and S. Terracini,
``On the singularities of generalized solutions to n-body-type problems''. 
{\it Int. Math. Res. Not.} IMRN 2008, Art. ID rnn 069, 78 pp. 

\bibitem{BBH} F. Bethuel, H. Brezis and F. H\'elein, ``Ginzburg-Landau Vortices'', 
{\it Progress in Nonlinear Differential Equations and their Applications 13,} 
(Birkh\"{a}user Boston, Boston, MA, 1994). 

\bibitem{BBO} F. Bethuel, H. Brezis, and G.  Orlandi, 
`` Asymptotics for the Ginzburg-Landau equation in arbitrary dimensions",
{\it J. Funct. Anal.}  186  (2001),  no. 2, 432--520.


\bibitem{BBM} J. Bourgain, H. Brezis, and P.  Mironescu,  ``$H^{1/2}$ maps with values into the circle: minimal connections, lifting, and the Ginzburg-Landau equation",
{\it Publ. Math. Inst. Hautes Études Sci.}  No. 99  (2004), 1--115.

\bibitem{Chen} K. Chen, ``Existence and minimizing properties  of retrograde orbits to the three-body problem with various choices of masses'', {\it Ann. of Math. (2)} 167 (2008), no. 2, 325--348.

\bibitem{CM} A. Chenciner and R. Montgomery, ``A remarkable periodic solution of the three body problem in the case  of equal masses'', {\it Ann. of Math.} 152 (1999),  881--901.  


\bibitem{AC.ARMA} A. Contreras, ``On the First critical field in Ginzburg-Landau theory for thin shells and manifolds'', {\it Archive for Rational Mechanics and Analysis}, Volume 200, Issue 2, pp.563--611. (2011). 

\bibitem{DK} M. del Pino and M. Kowalczyk, ``Renormalized energy of interacting Ginzburg-Landau vortex filaments'', 
{\it J. Lond. Math. Soc. (2)}  77 (2008), no. 3, 647--665.

\bibitem{dPKPW}
M. del Pino, M. Kowalczyk, F.  Pacard,  and J Wei,
``The Toda system and multiple-end solutions of autonomous planar elliptic problems'',
{\it Adv. Math.} 224 (2010), no. 4, 1462 - 1516. 

\bibitem{dPKW}
M. del Pino, M. Kowalczyk,   and J Wei,
``The Toda system and clustering interfaces in the Allen-Cahn equation", 
{\it Arch. Ration. Mech. Anal.} 190 (2008), no. 1, 141 -- 187. 

\bibitem{dPKWY} 
M. del Pino, M. Kowalczyk,  J. Wei, and J. Yang,
``Interface foliation near minimal submanifolds in Riemannian manifolds with positive Ricci curvature."
{\it Geom. Funct. Anal.} 20 (2010), no. 4, 918--957. 

\bibitem{Federer}
H. Federer, ``Geometric Measure Theory"
{\it  Die Grundlehren der mathematischen Wissenschaften, Band 153},
Springer-Verlag, New York 1969.

\bibitem{FT} D. Ferrario and S. Terracini, ``On the existence of collisionless  equivariant minimizers for the classical n-boby problem'',
{\it Invent. Math.} 155 (2004) 305--362.  

\bibitem{Jlower} R. Jerrard, ``Lower bounds for generalized Ginzburg-Landau functionals'', SIAM Math. Anal. {\bf 30} (4), 721--746, 1999.

\bibitem{waveJerrard} R. Jerrard, ``Vortex dynamics for the Ginzburg-Landau wave equation'', {\it Calc. Var. and PDE,} vol. 9, 1-30, 1999.

\bibitem{JSo} R. Jerrard and H. Soner, ``The Jacobian and the Ginzburg-Landau energy'', {\it Calc. Var and PDE,} 14, 151-191, 2002.

\bibitem{JSp1} R. Jerrard and D. Spirn, ``Refined Jacobian estimates for Ginzburg-Landau functionals'', {\it Indiana Univ. Math. Jour.,} vol. 56, 135-186, 2007.

\bibitem{JSp2} R. Jerrard and D. Spirn, ``Refined Jacobian estimates and Gross-Pitaevsky vortex dynamics'', {\it Arch. Rat. Mech. Anal.,} vol. 190, 425-475, 2008.

\bibitem{JSt} R. Jerrard and P. Sternberg, ``Critical points via Gamma-convergence: general theory and applications'', {\it Jour. Eur. Math. Soc.,} vol. 11, no. 4, 705-753, 2009.

\bibitem{KPV}   C. Kenig, G. Ponce and L. Vega, ``On the interaction of nearly parallel vortex filaments'',
{\it Comm. Math. Phys.} 243 (2003) 471--483. 


\bibitem{KMD}  R. Klein, A. Majda, and K. Damodaran,   ``Simplified
equations for the interaction of nearly parallel vortex filaments'',
{\it J. Fluid Mech.} 228 (1995), 201--248.

\bibitem{LinWave}F. H. Lin, ``Vortex dynamics for the nonlinear wave equation", 
{\it Comm. Pure Appl. Math.} 52 (1999), no. 6, 737?761.

\bibitem{LR} F.-H. Lin and T. Rivi\`re, ``A quantization property for static Ginzburg-Landau vortices",
{\it Comm. Pure Appl. Math.}  54  (2001),  no. 2, 206--228.
 
\bibitem{LM} P. L. Lions and A. Majda. ``Equilibrium Statistical Theory for Nearly Parallel Vortex Filaments '' , {\it Communications on Pure and Applied Mathematics}, Vol. LIII, 0076--0142, (2000).

\bibitem{MSZ} J. A. Montero, P. Sternberg, and W.P. Ziemer,
``Local minimizers with vortices in the Ginzburg-Landau system in three dimensions". 
{\it Comm. Pure Appl. Math.} 57 (2004), no. 1, 99--125.


\bibitem{PR} F. Pacard and T. Rivi\`ere, ``Linear and nonlinear aspects of vortice. The Gunzburg-Landau model'', 
{\it Progress in Nonlinear Differential Equations and their Applications 39,} 
(Birkh\"{a}user Boston, Boston, MA, 2000). 

\bibitem{Riv}
T. Rivi\`ere, ` Line vortices in the $U(1)$-Higgs model",
{\it  ESAIM Contr\^ole Optim. Calc. Var.}  1  (1995/96), 77--167.

\bibitem{S} E. Sandier.  ``Lower bounds for the energy of unit vector fields and applications'',
{\it J. Funct. Anal.} 152 (1998), 379--403.

\bibitem{S2} E. Sandier, `` Ginzburg-Landau minimizers from $\R^{n+1}$ to $\R^n$ and minimal connections",
{\it Indiana Univ. Math. J.}  50  (2001),  no. 4, 1807--1844.

\bibitem{SSprod}
E. Sandier and S.  Serfaty, `` A product-estimate for Ginzburg-Landau and corollaries",
{\it J. Funct. Anal.}  211  (2004),  no. 1, 219--244.

\bibitem{SandSerfbook}
E. Sandier and S. Serfaty.
``Vortices in the magnetic Ginzburg-Landau model".
{\em Progress in Nonlinear Differential Equations and their Applications}, 70. 
Birkh\"{a}user Boston, Inc., Boston, MA, 2007.


\bibitem{Solomon} B. Solomon, ``A new proof of the closure theorem for integral currents'', {\it Indiana Univ. Math. J.,} 33 no.
3:393--418, 1984.

\end{thebibliography}
\end{document}